%% file: HausdorffDomain_revision_arxiv.tex
\documentclass[a4paper,11pt]{article}
\usepackage[caption = false]{subfig}\newcommand{\subfl}[1]{\subfloat[#1]}

\usepackage{amsmath}
\usepackage{amsfonts}
\usepackage{amsthm}
\usepackage{amssymb}
\usepackage{enumerate}
\usepackage{accents}
\usepackage{tikz}\usetikzlibrary{matrix, calc, arrows}
\usepackage{color}
\usepackage{soul}
\usepackage{hyperref}

\DeclareMathOperator{\supp}{supp}

\DeclareMathOperator{\diam}{diam}

\newcommand{\Hull}{\mathrm{Hull}}

\usepackage{bbm}

\definecolor{amcol}{rgb}{0.8,0,0}
\definecolor{dhcol}{rgb}{0,0.5,0}

\newcommand{\ellref}{\ell_\mathrm{ref}}

\input{macros}

\renewcommand{\bs}[1]{\boldsymbol{#1}}
\renewcommand{\bn}{{\bs{n}}}
\newcommand{\bm}{{\bs{m}}}
\newcommand{\bmp}{{\bs{m}'}}
\newcommand{\bnp}{{\bs{n}'}}

\renewcommand{\bx}{x}%
\renewcommand{\by}{y}%
\newcommand{\bbx}{\Psi}
\newcommand{\bby}{\widetilde{\Psi}}

\renewcommand{\tH}{\widetilde{H}{}}

\newcommand{\GG}{(\Gamma)}
\newcommand{\tr}{\mathrm{tr}_\Gamma}

\newcommand{\IH}{\mathbb{H}}
\newcommand{\IS}{\mathbb{S}}
\newcommand{\IL}{\mathbb{L}}

\newcommand{\IV}{\mathbb{V}}

\newcommand{\IC}{\mathbb{C}}

\newcommand{\IY}{\mathbb{Y}}
\newcommand{\IA}{\mathbb{A}}

\newcommand{\footremember}[2]{%
	\footnote{#2}
	\newcounter{#1}
	\setcounter{#1}{\value{footnote}}%
}
\newcommand{\footrecall}[1]{%
	\footnotemark[\value{#1}]%
} 

\definecolor{purple0}{rgb}{0.4,0,0.5}
\definecolor{orange}{rgb}{1,0.4,0}
\newcommand{\vb}{{\vec b}}

\title{Integral equation methods for acoustic scattering by fractals}

\author{
	A.\ M.\ Caetano\footremember{1}{Center for R\&D in Mathematics and Applications, Departamento de Matem\'atica, Universidade de Aveiro, Aveiro, Portugal},
	S.\ N.\ Chandler-Wilde\footremember{2}{Department of Mathematics and Statistics, University of Reading, Reading, United Kingdom},
	X.\ Claeys\footremember{3}{Laboratoire Jacques-Louis Lions, Sorbonne Universit\'e, Paris, France},
	\and
	A.\ Gibbs\footremember{4}{Department of Mathematics, University College London, London, United Kingdom}, %
	D.\ P.\ Hewett\footrecall{4}\  
	\ and\ A.\ Moiola\footremember{5}{Dipartimento di Matematica ``F. Casorati'', Universit\`a degli studi di Pavia, Pavia, Italy}
}

\date{}
\newcommand{\rev}[1]{#1}

\begin{document}

	\maketitle

\begin{abstract}
We study sound-soft time-harmonic acoustic scattering by general scatterers, including fractal
scatterers, in 2D and 3D space.
For an arbitrary compact scatterer $\Gamma$ we reformulate the Dirichlet boundary value problem for the Helmholtz equation as a first kind integral equation (IE) on $\Gamma$ involving the Newton potential. 
The IE is well-posed, except possibly at a countable set of frequencies, and reduces to existing single-layer boundary IEs when $\Gamma$ is the boundary of a bounded Lipschitz open set, a screen, or a multi-screen.
When $\Gamma$ is uniformly of $d$-dimensional Hausdorff dimension in a sense we make precise (a $d$-set), the operator in our equation is an integral operator on $\Gamma$ with respect to $d$-dimensional Hausdorff measure, with kernel the Helmholtz fundamental solution, and we propose a piecewise-constant Galerkin discretization of the IE, which converges in the limit of vanishing mesh width.
When $\Gamma$ is the fractal attractor of an iterated function system of contracting similarities 
we prove convergence rates under assumptions on $\Gamma$ and the IE solution,
and describe a fully discrete implementation using recently proposed quadrature rules for singular integrals on fractals.
We present numerical results for a range of examples and make our software 
available as a Julia code.
\end{abstract}

\maketitle

\section{Introduction} \label{sec:intro}

This paper, prepared in large part during a recent Isaac Newton Institute programme on multiple wave scattering, is concerned with %
the classical problem of scattering of time-harmonic acoustic waves in $\R^n$, $n=2$ or $3$, by a scatterer $\Gamma$ \rev{(assumed to be a compact subset of $\R^n$)} that may have multiple components or other complicated geometrical features.
\rev{We consider the sound-soft case, where the total field $u^t$ vanishes on $\Gamma$, %
and satisfies the Helmholtz equation $(\Delta+k^2)u^t=0$ for some wavenumber $k>0$ in the open set $\Omega:=\R^n\setminus\Gamma$.} 
\rev{Our focus is on integral equation (IE) formulations of this scattering problem and their numerical solution}\footnote{We note that our methods and results apply, with obvious modifications, to the analogous (yet simpler) problem in potential theory, in which the Helmholtz equation is replaced by the Laplace equation.}.
\rev{
Our particular interest is in scattering by fractals, which provide a model for the multiscale roughness of many natural and man-made scatterers. 

In a sequence of recent papers \cite{ScreenPaper,BEMfract,HausdorffBEM} we studied %
scattering by fractal planar screens (i.e.\ scattering in $\R^n$ by fractal subsets of $\R^{n-1}$). 
In the current paper we show how 
the results of \cite{ScreenPaper,BEMfract,HausdorffBEM} 
can be generalised to scattering by fractal subsets of $\R^n$ that are {\em not contained in a hyperplane}.} 
\rev{This is a significant novelty compared to previous contributions, greatly extending the class of scatterers to which our results apply. }  
\rev{This generalisation complicates the analysis, because when the scatterer is a planar screen as in \cite{ScreenPaper,BEMfract,HausdorffBEM}, irrespective of its smoothness, the scattering problem can be written as a coercive (sign-definite) variational problem by the results in \cite{CoercScreen2,Ha-Du:90,Ha-Du:92}, while in the general case considered here the integral operators involved are compact perturbations of coercive ones %
and we have to resort to Fredholm theory. This generalisation also forces us to use a novel Sobolev space setting for the integral equation.  
On the other hand, %
for the case where the scatterer is a planar screen, the analysis carried out here is 
in certain respects 
a 
simplification of that in \cite{ScreenPaper,BEMfract,HausdorffBEM}, 
as %
the trace operator from $\R^n$ to a hyperplane (see Lemma~\ref{ConnectionPlanarScreens}) 
does not play a role}.

The contributions of the paper are several. Firstly, we write down in \rev{Theorem~\ref{thm:equivalence}} %
a novel IE formulation of \rev{the} problem that applies for any compact $\Gamma\subset \R^n$. 
The scattered field is sought as an acoustic Newton potential $u=\mathcal{A}\phi$, with an unknown density \rev{$\phi$ on $\Gamma$ satisfying a}
first-kind IE
\begin{equation} \label{eqn:IE}
A\phi=g,
\end{equation}
\rev{for some data $g$ depending on the incident wave.}
The operator $A$, 
defined in \eqref{eq:ADef} \rev{below},  
is a generalisation of the classical single-layer boundary integral operator, in cases where this is well-defined. 
Importantly, \rev{we prove in Lemma~\ref{lem:coer} that} $A$ is a compact perturbation of a coercive operator, so that all Galerkin solution methods are convergent, \rev{provided $A$ is also injective}.
\rev{Using this result}, \rev{in Theorem~\ref{thm:equivalence} we}
provide an IE-based proof of well-posedness \rev{for scattering by a} general compact $\Gamma$, generalising existing IE-based proofs for cases where $\Omega$ is Lipschitz or smoother (e.g., %
\cite[Thm~9.11]{McLean}).

We focus mostly on the particular case where $\Gamma$ is a $d$-set (definition \eqref{eq:dset} below), which means (roughly speaking) that $\Gamma$ is uniformly of Hausdorff dimension $d$, for some integer or fractional $d\in (0,n]$.
If %
$\Gamma$ is a $d$-set \rev{and} %
$d\leq n-2$%
\rev{, then, as we explain in Remark~\ref{rem:cap},} the scatterer is invisible to incident waves.
To focus on cases where $u\neq 0$ we restrict our study to the range $n-2<d\leq n$.

\begin{figure}%
\centering
\subfl{Closure of a \\ bounded Lipschitz \\ open set}{\includegraphics[width = 32mm]{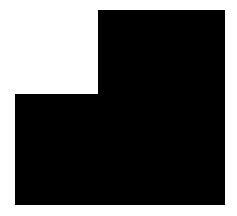}}
\hspace{2mm}
\subfl{Boundary of a \\ bounded Lipschitz \\ open set}{\includegraphics[width = 32mm]{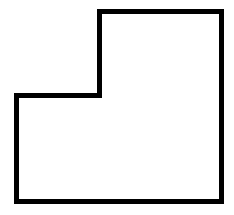}}
\hspace{2mm}
\subfl{Line segment \\ screen \hspace{5mm}}{\includegraphics[width = 32mm]{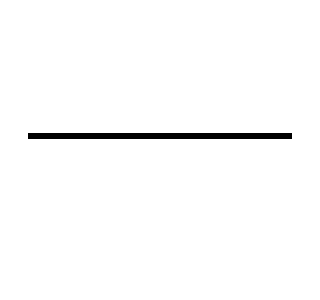}}
\subfl{Multi-screen}{\includegraphics[width = 32mm]{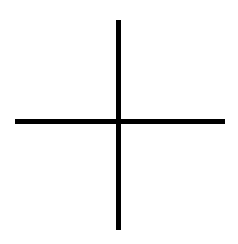}}\\
\subfl{Cantor set screen}{\includegraphics[width = 40mm]{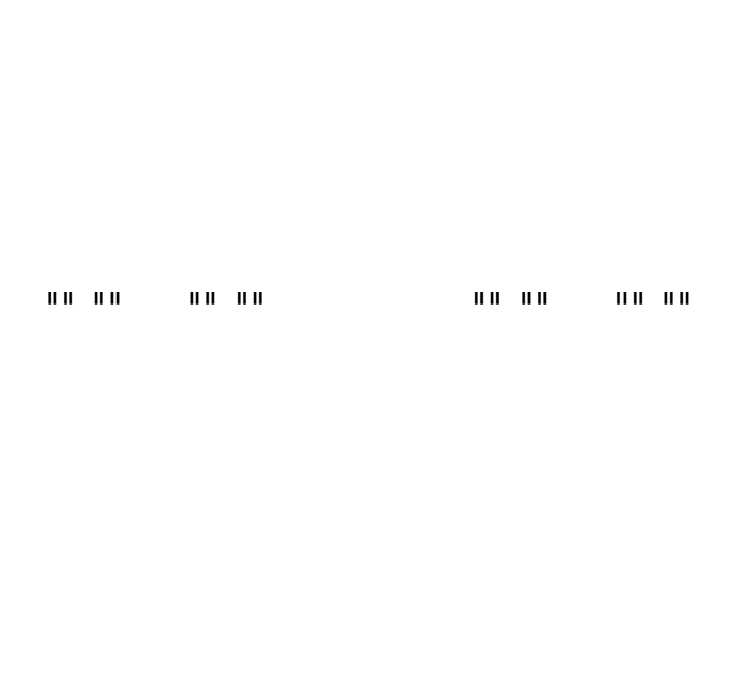}}
\hspace{4mm}
\subfl{Koch curve}{\includegraphics[width = 40mm]{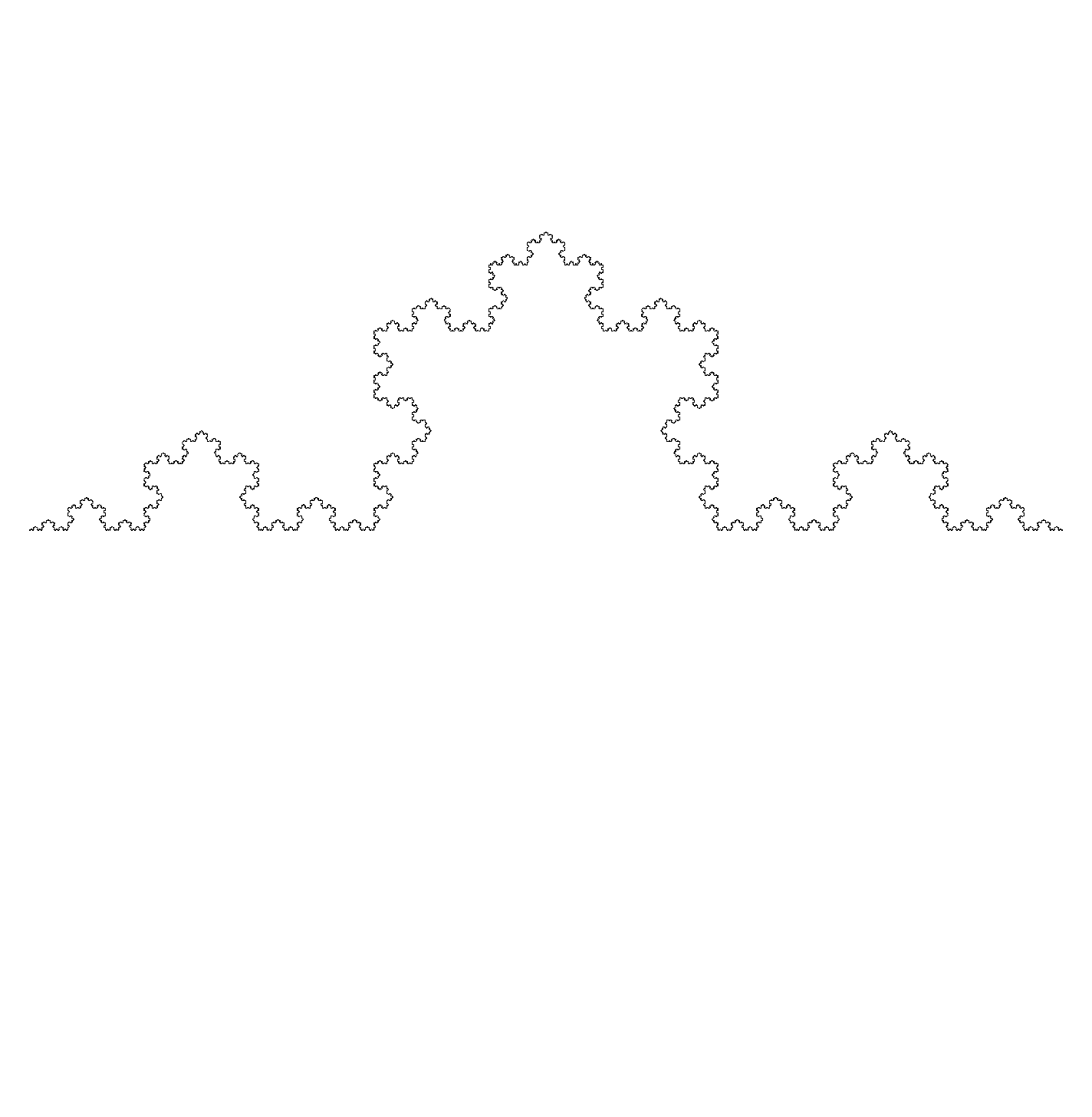}}
\hspace{4mm}
\subfl{Koch snowflake}{\includegraphics[width = 40mm]{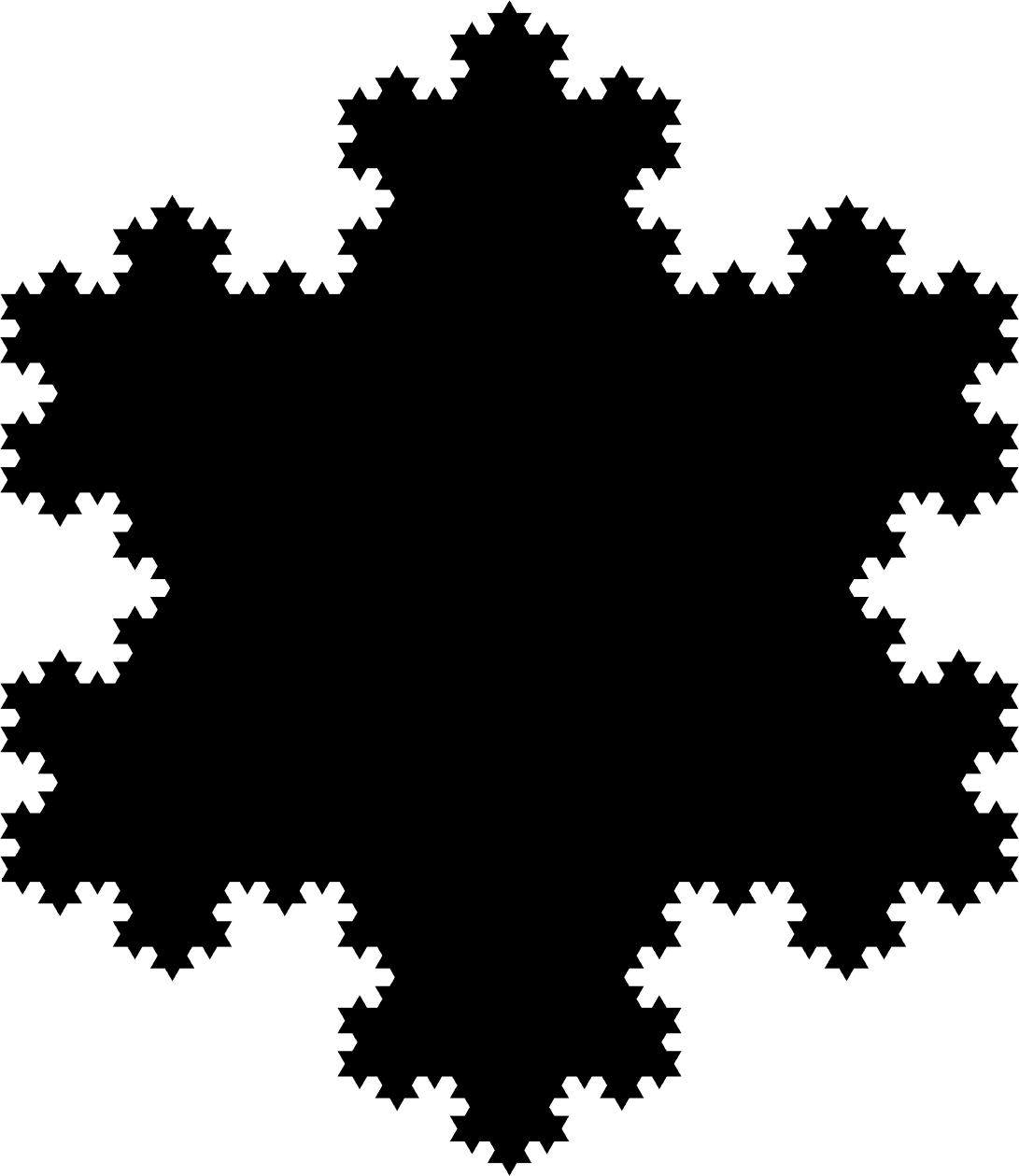}}
\caption{Examples of $d$-sets in two-dimensional space ($n=2$), with: a) $d=2$; b) $d=1$; c) $d=1$; d) $d=1$; e) $d= \log(2)/ \log(3) \approx 0.63$; f) $d = \log(4)/ \log(3) \approx 1.26$; g) $d=2$. For details see text \rev{of \S\ref{sec:intro}}.}
\label{fig:2Dexs}
\end{figure}

\begin{figure}%
\centering
\subfl{$\rho=1/2$, non-disjoint}
{\includegraphics[width=.40\linewidth]{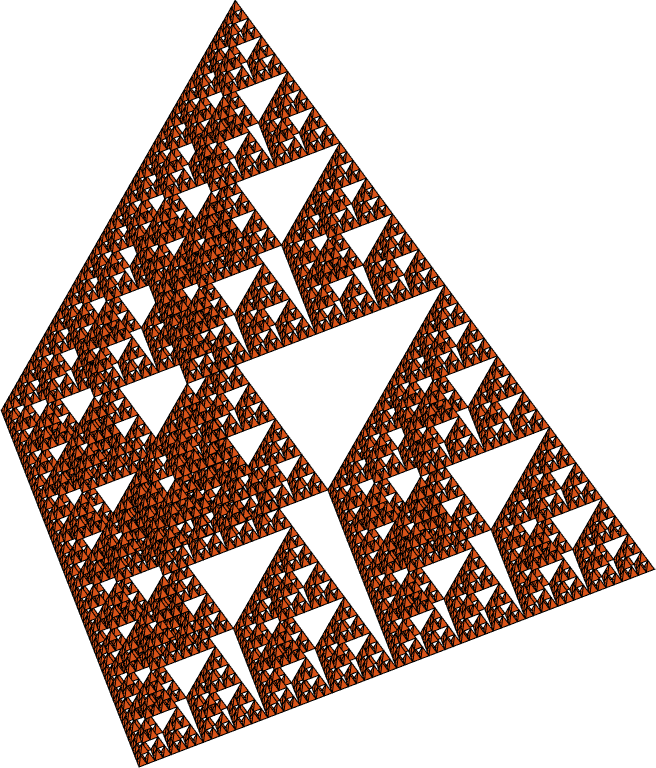}}
\hspace{15mm}
\subfl{$\rho=3/8$, disjoint}{\includegraphics[width=.40\linewidth]{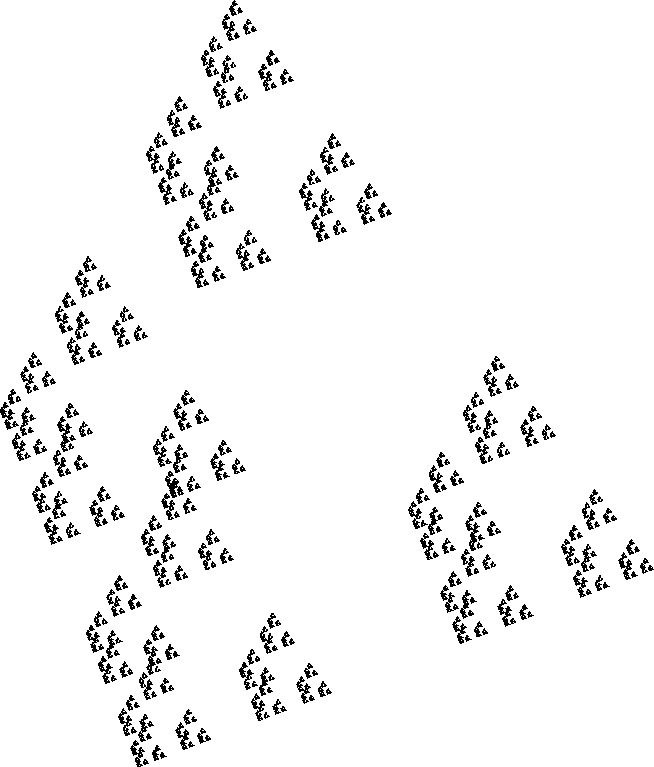}}
\caption{
Sierpinski tetrahedron $d$-sets in 3D space, attractors of the IFS \eqref{eq:TetIFS} for $\rho=1/2$ ($d=2$)
and $\rho=3/8$ ($d=\log{4}/\log(8/3)\approx 1.41$). 
}
\label{fig:Tet}
\end{figure}

A range of examples with $n-2<d\leq n$, relevant to our later discussions and computations, is pictured in Figures \ref{fig:2Dexs} and \ref{fig:Tet}. Figure \ref{fig:2Dexs} shows examples in 2D space ($n=2$), namely: (a) $\Gamma = \overline{D}$ is the closure of a bounded Lipschitz open set $D$ ($d=n=2$); (b) $\Gamma =\partial D$ is the boundary of the same set ($d=n-1=1$); (c) $\Gamma=[0,1]\times\{0\}$ is a line segment ($d=n-1=1$); (d)  $\Gamma=[-1,1]\times\{0\}\cup \{0\}\times [-1,1]$ is the cross formed by two line segments, an example of a multi-screen in the sense of \cite{ClHi:13} (all such multi-screens are $d$-sets with $d=n-1$); (e) $\Gamma = C\times\{0\}$, where $C\subset [0,1]$ is the classical middle-third Cantor set ($d=\log(2)/\log(3)\approx 0.6309$); (f) $\Gamma$ is the Koch curve ($d=\log(4)/\log(3)\approx 1.262$); (g) $\Gamma$ is the closure of the Koch snowflake domain ($d=2$).

Examples (c) and (e)-(g) in Figure \ref{fig:2Dexs} are all %
fixed points (attractors) of an iterated function system (IFS)
satisfying the standard open set condition (OSC) (we recall these definitions in \S\ref{sec:HausdorffMeasure}).
As we recall in \S\ref{sec:HausdorffMeasure}, every such IFS attractor is a $d$-set, with $d$ its fractal (Hausdorff) dimension. Figure \ref{fig:Tet} shows examples in 3D space ($n=3$) where $\Gamma$ is an IFS attractor that is a Sierpinski tetrahedron, with $d=2$ and $d=\log{4}/\log(8/3)\approx 1.41$. (We show numerical simulations for \rev{scattering by these shapes} %
in \S\ref{sec:Numerics}.)
These examples make clear that our results \rev{include} multiple scattering cases where $\Gamma$ has \rev{a complicated geometry} and/or multiple components. Indeed, the IFS attractor examples \rev{in} Figure \ref{fig:2Dexs}(e) and Figure \ref{fig:Tet}(b) are both fractal cases where the IFS is disjoint (as defined in \S\ref{sec:HausdorffMeasure}) so that $\Gamma$ is totally disconnected and has uncountably many components!

A key result, \rev{proved in Theorem~\ref{lem:ISasHauss}},
important both theoretically and computationally, is that in the $d$-set case we can interpret the
Newton potential $\mathcal{A}\phi$ as an integral with respect to $\mathcal{H}^d$, the $d$-dimensional Hausdorff measure.
Similarly, we show in \rev{Theorem~\ref{lem:ISasHauss}}
that the operator $A$ in \eqref{eqn:IE} %
can also be written equivalently as an integral operator \rev{$\IA$}
with respect to the $\mathcal{H}^d$ measure.

\rev{In certain special cases our formulation coincides with previously studied IE formulations.}
If $d=n$ (e.g., Figure~\ref{fig:2Dexs}(a), (g)), so that $\mathcal{H}^d$ is $n$-dimensional Lebesgue measure, \rev{$\IA$} %
is a volume integral \rev{operator} and \eqref{eqn:IE} is equivalent to a volume IE
on $\Gamma$. 
\rev{(Note however that, where $\Omega_+$ is the unbounded component of $\Omega$, the solution of \eqref{eqn:IE} is always supported in $\partial \Omega_+\subset \partial \Gamma$; see 
Remark \ref{rem:support}.)}
If \mbox{$d=n-1$} and $\Gamma$ is the whole or part of the boundary of a bounded Lipschitz open set (e.g., Figure~\ref{fig:2Dexs}(b), (c)), or is a multi-screen in the sense of \cite{ClHi:13} (e.g., Figure~\ref{fig:2Dexs}(d)), then $\mathcal{H}^d$ is standard surface measure (e.g., \cite[Theorem 3.8]{EvansGariepy}),  \rev{$\IA$} %
is a surface integral \rev{operator}, specifically an acoustic single-layer \rev{boundary integral operator}, and \eqref{eqn:IE} is equivalent to a standard first kind boundary IE (see Remark \ref{rem:2nd}). %
Finally, in the case when $\Gamma$ is a subset of a hyperplane (e.g., Figure~\ref{fig:2Dexs}(c), (e)), our formulation reduces to cases studied recently in \cite{ChaHewMoi:13,ScreenPaper,BEMfract,HausdorffBEM,CaChGiHe23}. In particular, \rev{as already noted above,} our results and methods build on those in \cite{HausdorffBEM,CaChGiHe23}, where we study scattering by fractal planar screens that are $d$-sets with $n-2<d\leq n-1$.

In the $d$-set case we show that the integral operator $\IA$ is a continuous mapping on a scale of Sobolev spaces on $\Gamma$ (Proposition \ref{lem:cont}), a first step in a regularity theory for solutions of \eqref{eqn:IE}. 
In the $d$-set case we are also able to propose (in \S\ref{sec:HausdorffIEM}\rev{, Eqn.~\eqref{eqn:Variational} in particular}) a \rev{piecewise-constant} %
Galerkin IE method (IEM) \rev{for the numerical solution of \eqref{eqn:IE}, which we prove (in Theorem~\ref{thm:Convergence}) is convergent as the mesh size $h$ tends to zero.
}
Moreover, the entries of the matrix and right-hand-side of the linear system %
defining the Galerkin solution are given explicitly as double and single integrals, respectively, with respect to $\mathcal{H}^d$ measure.

Our IEM is familiar in the case where $\Gamma$ is the whole or part of the boundary of a bounded Lipschitz open set, or is a multi-screen in the sense of \cite{ClHi:13}; our IEM is then a standard Galerkin boundary element method (BEM) with a piecewise-constant approximation space. In the case where $\Gamma$ is a $d$-set that is a planar screen our IEM coincides with that of \cite{HausdorffBEM}; indeed, our linear system is identical to that in \cite[Eqn (55)]{HausdorffBEM}.

Our strongest results (see \S\ref{sec:GalerkinBounds}) are for the special $d$-set case where $\Gamma$ is the attractor of an IFS satisfying the OSC (e.g., Figure~\ref{fig:2Dexs}(c), (e)-(g), Figure~\ref{fig:Tet}).
In this case\rev{, under appropriate assumptions on $\Gamma$, we are able to prove (in Theorems~\ref{thm:IEMConvergence} and \ref{thm:ApproxDisjoint}) convergence rates for our IEM.
The rate of convergence depends on the regularity of the solution $\phi$.
In Hypothesis~\ref{conj} and Remark \ref{rem:ConvRates} we introduce a hypothesis relating to this regularity and detail the resulting convergence rates.}  
In this 
case 
we also propose in \S\ref{sec:Quadrature} a fully discrete implementation, evaluating our Hausdorff single and double integrals using recently proposed quadrature methods for singular integrals on IFS attractors \rev{satisfying the OSC} \cite{HausdorffQuadrature,NonDisjointQuad}.

In \S\ref{sec:Numerics} we show computations using our fully discrete Galerkin IEM, solving \eqref{eqn:IE} and computing the scattered field $u=\mathcal{A}\phi$ for scatterers $\Gamma$ including the examples in Figure~\ref{fig:Tet}.
This section includes numerical experiments exploring the convergence of our method.
These suggest that our regularity \rev{hypothesis (Hypothesis~\ref{conj})} %
is true for many of the examples we study, %
and that the conditions of Theorem \ref{thm:IEMConvergence}\rev{, guaranteeing the validity of our convergence rate analysis}, 
may be satisfied generally whenever $\Gamma$ is an IFS attractor satisfying the OSC (establishing this, and proving some version of \rev{our regularity hypothesis}, are both open problems).

\rev{We end this introduction by noting that} an alternative approach to the simulation of scattering by fractals is to approximate the fractal by a smoother ``prefractal'' scatterer and apply a more conventional numerical method to the resulting approximate scattering problem. This was the approach taken in \cite{BEMfract,ImpedanceScreen}, and in the earlier work in \cite{jones1994fast} and \cite{panagouli1997fem} for Laplace and elasticity problems. An achievement of \cite{BEMfract,ImpedanceScreen} was to prove convergence (without rates) of conventional BEMs for prefractal approximations of  fractal \rev{planar} screen problems, via Mosco convergence techniques.
In principle, a similar analysis could be carried out for the problems under consideration in the current paper. However,
we do not pursue this here.

\section{Preliminaries}
\label{sec:prelim}
\rev{In this section we set some notation/terminology, and briefly review a number of known results that we will use later in the paper. 
Further details can be found in the references provided.} 

Throughout, for $n\in\N$ and $E\subset \R^n$, %
$\overline{E}$, $\partial E$, and $E^\circ:=\overline{E}\setminus \partial E$ denote the closure, boundary, and interior of $E$ with respect to the standard Euclidean metric on $\R^n$, and $E^c:= \R^n\setminus E$ its complement in $\R^n$. In the case that $E$ is measurable, $m(E)$ denotes its $n$-dimensional Lebesgue measure. $B_r(x)\subset \R^n$ denotes the closed ball of radius $r$ centred on $x$.

\subsection{Hausdorff measure and dimension, \texorpdfstring{$d$}{d}-sets, IFS attractors}
\label{sec:HausdorffMeasure}

For $0\leq d\leq n$ let $\cH^d$ denote the Hausdorff $d$-measure
on $\R^n$ and let $\dimH(S)\in [0,n]$ denote the Hausdorff dimension of $S\subset\R^n$ (see, e.g.,\ \cite{Fal}).
\rev{For convenience 
we adopt 
the normalisation of \cite[Def.~2.1]{EvansGariepy}, so that $\cH^d$ coincides with Lebesgue measure for $d=n$.}
As in \cite[\S1.1]{JoWa84} and \cite[\S3]{Triebel97FracSpec}, given $0<d\leq n$, we say a closed set $\Gamma\subset \R^n$ is a {\em$d$-set} if there exist $c_{2}>c_1>0$ such that
\begin{align}
\label{eq:dset}
c_{1}r^{d}\leq\mathcal{H}^{d}\big(\Gamma\cap B_{r}(x)\big)\leq c_{2}r^{d},\qquad x\in\Gamma,\quad0<r\leq1.
\end{align}
\rev{We note that $d$-sets are also termed \textit{Ahlfors $d$-regular} or \textit{Ahlfors-David $d$-regular sets}, e.g., \cite[p.~92]{Mattila95}.}
\rev{We note also that} if $\Gamma$ is a $d$-set then $\dimH(\Gamma)=d$. 

By an \textit{iterated function system of contracting similarities} (we abbreviate this whole phrase as  \textit{IFS})\footnote{A useful introduction to IFSs is \cite[Chap.~9]{Fal}; the website \cite{RiddleWebSite} gives many examples of IFSs and their attractors.}
we mean a collection
$\{s_1,s_2,\ldots,s_M\}$, for some $M\geq 2$,
where,
for each $m=1,\ldots,M$, $s_m:\R^{n}\to\R^{n}$, with
$|s_m(x)-s_m(y)| = \rho_m|x-y|$, $x,y\in \R^{n}$,
for some $\rho_m\in (0,1)$.
The attractor of the IFS is the unique non-empty compact set $\Gamma$ satisfying
\begin{equation} \label{eq:fixedfirst}
\Gamma = s(\Gamma), \quad \mbox{where} \quad s(E) := \bigcup_{m=1}^M  s_m(E), \quad  E\subset \R^{n}.
\end{equation}
We shall restrict our attention to \rev{\textit{OSC-IFSs}, i.e.\ }IFSs that satisfy the standard \textit{open set condition (OSC)} \cite[(9.11)]{Fal}, which implies that the attractor $\Gamma$ is a $d$-set (see, e.g.,\ \cite[Thm.~4.7]{Triebel97FracSpec}), where $d\in (0,n]$ is the unique solution of $\sum_{m=1}^M  (\rho_m)^d = 1$. 
For a \textit{homogeneous} \rev{OSC-}IFS, where $\rho_m=\rho \in (0,1)$ for $m=1,\ldots,M $, we have $d = \log(M)/\log(1/\rho)$. 
\rev{If an OSC-IFS is not homogeneous we say it is \emph{non-homogeneous}.} 
Returning to the general, not necessarily homogeneous case, the OSC also implies (again, see \cite[Thm.~4.7]{Triebel97FracSpec}) that $\Gamma$ is \textit{self-similar}, meaning that the sets
\begin{align}
\label{eq:GammamDef}
\Gamma_m:=s_m(\Gamma),\qquad m=1,\ldots,M,
\end{align}
satisfy
$\cH^d(\Gamma_{m}\cap \Gamma_{m'})=0$, $m\neq m'$,
so that $\Gamma$ is decomposed by \eqref{eq:fixedfirst} into $M$ similar %
copies of itself whose pairwise intersections have Hausdorff measure zero.
For many of our results we make the additional assumption that the sets $\Gamma_1,\ldots,\Gamma_M $ are disjoint. If this holds we say that the IFS attractor $\Gamma$ is \emph{disjoint}, the OSC is automatically satisfied (e.g., \cite[Lem.~2.5]{HausdorffBEM}), and $d<n$ (e.g.\ \cite[Lemma 2.6]{HausdorffBEM}). \rev{If $\Gamma$ is not disjoint we say it is \emph{non-disjoint}.}

The following construction makes clear that if $C$ is an IFS attractor in dimension $n-1$ then $C\times\{0\}$ is an IFS attractor in dimension $n$.
\begin{rem}[\bf Lifting attractors to higher dimensions] \label{rem:extendDim} Suppose that $M\geq 2$ and that
$S=\{s_1,s_2,$ $\ldots,$ $s_M\}$ is an IFS  on $\R^n$. For $m=1,\ldots,M$, define $\tilde s_m:\R^{n+1}\to\R^{n+1}$ by $\tilde s_m((x,t)) = (s_m(x),\rho_m t)$, for $x\in \R^n$, $t\in \R$. Then, for $m=1,\ldots,M$, 
$\tilde s_m$ is a contracting similarity with the same contraction factor $\rho_m$ as $s_m$, so that $\widetilde S =\{\tilde s_1,\tilde s_2,\ldots,\tilde s_M\}$ is an IFS  
on $\R^{n+1}$. Further, $\widetilde S$ satisfies the OSC/is disjoint if the same holds for $S$. 
If $\Gamma$ is the attractor of $S$ then $\widetilde \Gamma:= \Gamma\times \{0\}=\{(x,0):x\in \Gamma\}$ is the attractor of $\widetilde S$, and if the OSC holds for $S$ 
then $\Gamma$ and $\widetilde \Gamma$ are both $d$-sets with the same value of $d$.
\end{rem}
\begin{example}[\bf Cantor set examples of IFS attractors] \label{ex:Cantor} Let $S=\{s_1,s_2\}$, where $s_m:\R\to\R$, $m=1,2$, are defined, for some $\rho\in (0,1/2]$, by
$$
s_1(t) = \rho t, \qquad s_2(t) = 1+ \rho(t-1), \qquad t\in \R.
$$
Then $S$ is a homogeneous IFS 
with attractor $C$ that is the ``middle-$(1-2\rho)$'' Cantor set, given by
$$
C = \bigcap_{n=0}^\infty C_n, \quad \mbox{where} \quad C_0:=[0,1], \;\; C_n:= s(C_{n-1}), \;\; n\in \N,
$$
and $s$ is the mapping given by \eqref{eq:fixedfirst} on the set of subsets of $\R$. In the case $\rho=1/2$, $C=C_n=[0,1]$ for each $n$. If $0<\rho<1/2$ then $C_n$ is a union of $2^n$ disjoint closed intervals and $C_n$ is obtained from $C_{n-1}$ by removing the middle $(1-2\rho)$ from each of the intervals comprising $C_{n-1}$. This IFS satisfies the OSC (see \cite[\S9.2]{Fal}), so that (see above) it is a $d$-set with $d=\dim_H(C) = \log(2)/\log(1/\rho)$. The attractor $C$ is disjoint if $\rho<1/2$. By Remark \ref{rem:extendDim}, $C\times \{0\}\subset \R^2$, shown for $\rho=1/2$ and $\rho=1/3$ in Figure \ref{fig:2Dexs}(c) and (e), respectively, is also the attractor of an IFS, %
is a $d$-set with the same value of $d$, and is disjoint if $\rho<1/2$.
\end{example}
\begin{example}[\bf Koch curve] \label{ex:Koch} The Koch curve $\Gamma\subset \R^2$, shown in Figure~\ref{fig:2Dexs}(f), is the attractor of the homogeneous IFS %
$\{s_1,s_2,s_3,s_4\}$, where the mappings $s_m:\R^2\to \R^2$, $m\in\{1,\ldots,4\}$, are given by
\begin{align*}
\label{}
&s_1(x) = x/3, \;\; &&s_2(x) = Rx/3 + \left(\begin{array}{c}\frac{1}{3}\\0\end{array}\right),
\;\; \\ &s_3(x) = R^{-1}x/3 + \left(\begin{array}{c}\frac{1}{2}\\\frac{1}{2\sqrt{3}}\end{array}\right),
\;\; &&s_4(x) = x/3 + \left(\begin{array}{c}\frac{2}{3}\\0\end{array}\right),
\end{align*}
for $x\in \R^2$, where $R$ is the (orthogonal) rotation matrix for rotation counter-clockwise by angle $\pi/3$.
$S$ satisfies the OSC (e.g., \cite[Ex.~9.5]{Fal}) and so $\Gamma$ is a $d$-set with $d=\log(4)/\log(3)\approx 1.26$. By Remark \ref{rem:extendDim}, $\Gamma\times \{0\}\subset \R^3$ is also the attractor of a homogeneous IFS, %
and is a $d$-set with the same value of $d$.
\end{example}

\begin{example}[\bf Koch snowflake] \label{ex:KochSnowflake} The Koch snowflake $\Gamma\subset \R^2$, shown in Figure~\ref{fig:2Dexs}(g), is the attractor of a non-homogeneous IFS of 7 contracting similarities satisfying the OSC. It is non-disjoint and has dimension $d=2$.
For details see \cite[\S5.4]{NonDisjointQuad}.
\end{example}
\begin{example}[\bf Sierpinski tetrahedron] \label{ex:SierpinskiTetrahedron} A Sierpinski tetrahedron $\Gamma\subset \R^3$ can be defined, for every $0<\rho\leq 1/2$, as the attractor of the IFS comprising the four contracting similarities
\begin{align}
\label{eq:TetIFS}
s_i(x): = x_i + \rho(x-x_i), \qquad i=1,\ldots,4,
\end{align}
where the $x_i$ are the vertices of a unit tetrahedron, explicitly
\begin{align*} x_1 = (0,0,0)^T&, \; x_2=(1,0,0)^T, \; x_3=(1/2,\sqrt{3}/2,0)^T, \; \\&x_4=(1/2,1/(2\sqrt{2}),\sqrt{5}/(2\sqrt{2}))^T.\end{align*}
$\Gamma$ is shown in Figure~\ref{fig:Tet} for $\rho=3/8$ and $\rho=1/2$.
It satisfies the OSC, has dimension $d=\log 4/\log(1/\rho)$, and is disjoint for $0<\rho<1/2$ but not for $\rho=1/2$.
\end{example}

\subsection{Function spaces}
\label{sec:FunctionSpaces}

\rev{In this section we briefly review some function space definitions and results that will be used in our scattering problem and its IE formulations. 
} 
Our notation follows that of \cite{McLean} and \cite{HausdorffBEM}. 
\rev{Throughout, our function spaces are spaces of complex-valued functions/distributions.}

For $s\in \R$ we let $H^s\Rn$ denote the usual Bessel potential Sobolev space.
For a non-empty open set $\Omega\subset \R^n$, where $C_0^\infty(\Omega)$ is the set of those $C^\infty$ functions that are compactly supported in $\Omega$, we
define
$\tH^s(\Omega):=\overline{C_0^\infty(\Omega)}^{H^s(\R^n)}$, \rev{a closed subspace of $H^s(\R^n)$}. 
For a \rev{non-empty} closed set $E\subset \R^n$ we denote by $H^s_E$ the set of all elements of $H^s\Rn$ whose %
support is contained in $E$, also a closed subspace of $H^s(\R^n)$. 
We recall that, for $s\in \R$, $H^{-s}(\R^n)$ is dual to $H^{s}(\R^n)$, with the duality pairing $\langle \cdot,\cdot\rangle_{H^{-s}(\R^n)\times H^s(\R^n)}$ extending the $L_2(\R^n)$ inner product,\footnote{
\rev{Note that all our distributions and dual spaces are anti-linear rather than linear to suit our complex Hilbert space setting.}} 
and that, with respect to this same pairing, %
$H^{-s}_E$ is dual to $\tH^{s}(E^c)^\perp$,
the orthogonal complement of $\tH^{s}(E^c)$ \rev{in} $H^s(\R^n)$, for any \rev{non-empty closed $E\subsetneqq\R^n$} \cite[Cor.~3.4]{ChaHewMoi:13}.

For \rev{non-empty} open sets $\Omega\subset\R^{n}$
we also work with the classical Sobolev space
$W^1(\Omega)$,
normed by $\|u\|_{W^1(\Omega)}^2=\|u\|_{L_2(\Omega)}^2+\|\nabla u\|_{L_2(\Omega)}^2$, the closed subspace $W^1_0(\Omega)=\overline{C_0^\infty(\Omega)}^{W^1(\Omega)}$,
and their ``local'' versions $W^{1,{\rm loc}}(\Omega)$ and
$W_0^{1,{\rm loc}}(\Omega)$,
defined
as the sets of measurable functions $v$ on $\Omega$
such that $\sigma|_\Omega v$ is in
$W^{1}(\Omega)$
or $W^1_0(\Omega)$,
respectively, for every $\sigma\in C^\infty_0(\R^n)$. Similarly, we define local versions $H^{1,{\rm loc}}(\R^n)$ and $\tH^{1,{\rm loc}}(\Omega)$ of $H^1(\R^n)$ and $\tH^1(\Omega)$.  
We \rev{note} that $H^1(\R^n)=W^{1}(\R^n)$ (and hence $H^{1,{\rm loc}}(\R^n)=W^{1,{\rm loc}}(\R^n)$), 
and that $\tH^{1}(\Omega)=W^{1}_0(\Omega)$ (and hence $\tH^{1,{\rm loc}}(\Omega)=W^{1,{\rm loc}}_0(\Omega)$) for arbitrary non-empty open $\Omega\subset\R^{n}$,
with the latter identification involving the restriction operator, with extension by zero as its inverse.

Finally,
for \rev{compact} $\Gamma\subset\R^n$ let $C^\infty_{0,\Gamma}$ denote the set of functions in $C^\infty_0(\R^n)$ that equal one in a neighbourhood of $\Gamma$.

\section{Scattering problem and integral equation formulations}
\label{sec:BVPsIEs}

Let $\Gamma\subset\R^n$ \rev{($n=2,3$)} be non-empty and compact and let $k>0$. We consider the time-harmonic acoustic scattering of an incident wave $u^i$ by $\Gamma$, a sound-soft obstacle. We assume that
the incident wave $u^i$ is an element of $W^{1,{\rm loc}}(\R^{n})=H^{1,{\rm loc}}(\R^{n})$ satisfying
the Helmholtz equation
\begin{align}
\label{eqn:HE}
\Delta u + k^2 u = 0
\end{align}
in a distributional sense in some neighbourhood of $\Gamma$ (so that $u^i$ is $C^\infty$ in that neighbourhood by elliptic regularity, see, e.g., \cite[Thm~6.3.1.3]{Evans2010}); for instance, $u^i$ might be the plane wave $u^i(x)=\re^{\ri k \vartheta\cdot x}$ for some $\vartheta\in\R^{n}$ with $|\vartheta|=1$).
Where $\Omega:= \Gamma^c =\R^n\setminus \Gamma$, we seek a scattered field $u\in W^{1,{\rm loc}}(\Omega)$ satisfying
\rf{eqn:HE} in a distributional sense
in $\Omega$ (so that $u\in C^\infty(\Omega)$ by elliptic regularity),
the Sommerfeld radiation condition
\begin{align}
\label{eqn:SRC}
\pdone{u(x)}{r} - \ri k u(x) = o(r^{-(n-1)/2}), \qquad r:=|x|\to\infty, \text{ uniformly in } \hat x:=x/|x|,
\end{align}
and the boundary condition $u=-u^i$ on $\partial \Omega = \partial\Gamma$, enforced by requiring that the total field
\begin{align}
\label{eqn:BC_Wonezero}
u^t:= u+u^i\in W^{1,{\rm loc}}_0(\Omega).
\end{align}
\rev{Note that $\partial\Omega=\Gamma$ if and only if $\Gamma$ has empty interior.}

This problem, which we will refer to as our {\em scattering problem}, 
is \rev{uniquely solvable} %
in the case that $\Omega$ is connected (see, e.g.\ \cite[\S3]{WavenumberExplicit}). Figure \ref{fig:2Dexs}, with the  exception of (b), and Figure \ref{fig:Tet} are all examples of such cases. But, to understand the well-posedness of our IE formulation, %
we also want to allow cases (such as Figure \ref{fig:2Dexs}(b)) where $\Omega$ is not connected, in which case $\Omega = \Omega_+\cup\Omega_-$, where $\Omega_\pm$ are disjoint open sets, with $\Omega_+$ the unbounded component of $\Omega$ and $\Omega_-$ a bounded open set. In such cases the above problem decouples into a uniquely-solvable scattering problem for $u|_{\Omega_+}\in  W^{1,{\rm loc}}(\Omega_+)$ and the homogeneous Dirichlet problem that $u^t|_{\Omega_-}\in  W_0^{1}(\Omega_-)$ satisfies \eqref{eqn:HE} in $\Omega_-$. Thus, \rev{if $\Omega$ is not connected,} our scattering problem is uniquely solvable (with $u^t=0$ and $u=-u^i$ in $\Omega_-$) if and only if $k^2$ is not a Dirichlet eigenvalue of $-\Delta$ in $\Omega_-$.
We will frequently assume that $k$ is not one of these exceptional values, making the following assumption.

\begin{ass}%
\label{ass:Uniqueness}
\rev{
The only $v\in W^{1,{\rm loc}}_0(\Omega)$ satisfying 
\eqref{eqn:HE} in $\Omega:=\Gamma^c$ and \eqref{eqn:SRC} is $v=0$. 
}
\end{ass}

\begin{rem} \label{rem:Ass}
We emphasise that  Assumption \ref{ass:Uniqueness} holds for all $k>0$ if $\Omega$ is connected, and that if $\Omega$ is not  connected, in which case $\Omega = \Omega_+\cup\Omega_-$, where $\Omega_\pm$ are the disjoint open sets defined above, then Assumption \ref{ass:Uniqueness} holds if and only if there is no non-trivial $v\in W_0^1(\Omega_-)$ that satisfies \eqref{eqn:HE} in $\Omega_-$, which holds for all $k>0$ outside a countable set whose only accumulation point is infinity.
\end{rem}

\rev{In the case that Assumption \ref{ass:Uniqueness} holds, and where $u$ is the unique solution to the above scattering problem
and $u^t:= u+u^i\in W_0^{1,{\rm loc}}(\Omega)$, it proves convenient to extend $u^t$ by zero from $\Omega$ to $\R^n$ so that (as noted in \S\ref{sec:FunctionSpaces}) $u^t\in \tH^{1,{\rm loc}}(\Omega)\subset H^{1,{\rm loc}}(\R^n)$. We can correspondingly extend the definition of $u$ from $\Omega$ to $\R^n$, by setting $u:=u^t-u^i\in H^{1,{\rm loc}}(\R^n)$ (so that $u=-u^i$ on $\Gamma$ almost everywhere with respect to $n$-dimensional Lebesgue measure). We will assume these extensions hereafter, so that $u$ and $u^t$ are defined (almost everywhere) on $\R^n$ and $u,u^t\in   H^{1,{\rm loc}}(\R^n)$.} 
Alternatively, one can require from the outset that $u\in  H^{1,{\rm loc}}(\R^n)$ and satisfies \eqref{eqn:HE} in $\Omega$ and \eqref{eqn:SRC} and that $u^t=u+u^i\in \widetilde H^{1,{\rm loc}}(\Omega)$, in which case $u|_\Omega$ is the unique solution to the above scattering problem and $u=-u^i$ on $\Gamma$ (almost everywhere with respect to $n$-dimensional Lebesgue measure).

Introducing the orthogonal projection operator
\begin{equation} \label{eq:Porth}
P:H^{1}(\R^n)\to \tH^{1}(\Omega)^\perp,
\end{equation}
we observe that $u^t\in \widetilde H^{1,{\rm loc}}(\Omega)$ if and only if $P(\sigma u^t)=0$ for some, and hence every\footnote{If $\sigma_1,\sigma_2\in C^\infty_{0,\Gamma}$ then $\sigma_1=\sigma_2$ on some open set $G\supset\Gamma$, so that, for $v\in H^{1,\mathrm{loc}}(\R^n)$, $(\sigma_1- \sigma_2)v\in H^{1}_{\R^n\setminus G}\subset
\tH^{1}(\Omega)$, so that $P((\sigma_1- \sigma_2)v)=0$.}, $\sigma\in C^\infty_{0,\Gamma}$, in other words, if and only if
\begin{align}\label{a1bc}
P(\sigma u)&=g,
\end{align}
where
\begin{align}
\label{eqn:gDefScatteringProblem}
g:=-P (\sigma u^i).
\end{align}
Thus \eqref{a1bc} is an alternative formulation of the boundary condition that $u=-u^i$ on $\Gamma$, equivalent to the requirement that $u^t\in \widetilde H^{1,{\rm loc}}(\Omega)$.

\rev{To summarise, our scattering problem can be stated as follows: find $u\in  H^{1,{\rm loc}}(\R^n)$ satisfying \eqref{eqn:HE} in $\Omega$, \eqref{eqn:SRC}, and 
\eqref{a1bc}.  
We now reformulate this problem
as an 
integral equation.} 

\subsection{\rev{The integral equation on general compact sets}} \label{sec:IECompact}

In what follows, $\cA$ will
denote the standard acoustic Newton potential operator, defined for compactly supported $\phi \in L_{2}(\R^n)=H^{0}(\R^n)$
by
\begin{align}
\label{eq:NPPrep}
\cA\phi(\bx)=\int_{\R^n}\Phi(\bx,\by)\phi(\by)\,\rd \by, \qquad \bx\in \R^n,
\end{align}
where 
\rev{$\Phi(\bx,\by):=\re^{\ri k |\bx-\by|}/(4\pi |\bx-\by|)$ ($n=3$), $\Phi(\bx,\by):=\frac\ri4H^{(1)}_0(k|\bx-\by|)$ ($n=2$),  is the standard fundamental solution of the Helmholtz equation, and $H_0^{(1)}$ is the Hankel function of the first kind of order zero (e.g., \cite[Eqn.~(\href{https://dlmf.nist.gov/10.4.E3}{10.4.3})]{DLMF}).
}
It is standard (see e.g.\ \cite[Thm 3.1.2]{sauter-schwab11}) that, for $s\in \R$, in particular for $s=0$,
$\cA$ is continuous as a mapping
\begin{align}
\label{eq:NewtonMapping}
\cA: H^{s-1}_{\rm comp}(\R^n)\to H^{s+1,{\rm loc}}(\R^n),%
\end{align}
where, for $s\in \R$, $H^{s}_{\rm comp}(\R^n)$ is the space of compactly supported elements of $H^{s}(\R^n)$. \rev{Further}  
(e.g.\ \cite[Thm 3.1.4]{sauter-schwab11}),
\begin{align}
\label{eq:NewtonHelmholtz}
(\Delta+k^2)\cA\phi = \cA(\Delta+k^2)\phi = -\phi, \qquad \phi \in H^{-1}_{\rm comp}(\R^n).
\end{align}

Viewing $\cA$ as an operator $\cA:H^{-1}_{\Gamma}\to H^{1,{\rm loc}}(\R^n)=W^{1,{\rm loc}}(\R^n)$,
we have that
\begin{equation} \label{eq:NPdef}
\cA\phi(x) = \langle(\sigma\Phi(x,\cdot)), \overline{\phi}\rangle_{H^{1}(\R^n)\times H^{-1}(\R^n)}, \quad x\in \Omega, %
\end{equation}
where $\overline{\phi}$ denotes the complex conjugate of $\phi$ and $\sigma$ is any element of $C^\infty_{0,\Gamma}$ with $x\not\in \supp{\sigma}$.
We define the operator $A:H^{-1}_{\Gamma}\to \tH^{1}(\Omega)^\perp = (H^{-1}_{\Gamma})^*$ 
\rev{(the latter equality holding by \cite[Cor.~3.4]{ChaHewMoi:13})} by
\begin{equation} \label{eq:ADef}
A\phi:=  P(\sigma\cA\phi), \qquad \phi\in H^{-1}_{\Gamma},
\end{equation}
with $\sigma \in C^\infty_{0,\Gamma}$ arbitrary. 
We also define the associated sesquilinear form $a(\cdot,\cdot)$  on $H^{-1}_\Gamma\times  H^{-1}_\Gamma$ by
\begin{align}
\label{eqn:Sesqui}
a(\phi,\psi):=
\langle A\phi,\psi\rangle_{H^{1}(\R^n)\times H^{-1}(\R^n)},\qquad \phi, \psi\in H^{-1}_{\Gamma}.
\end{align}
This form is compactly perturbed coercive, meaning that the 
operator $A:H^{-1}_{\Gamma}\to \tH^{1}(\Omega)^\perp$ is a compact perturbation of a coercive operator (see, e.g., \cite[\S2.2]{BEMfract} for detailed definitions and discussion). The following lemma 
is 
a generalisation of \cite[Prop.~8.7, 8.8]{ClHi:13}, and our proof is similar.

\begin{lem} \label{lem:coer}
\rev{Let $\Gamma\subset\R^n$ be compact.} The sesquilinear form $a(\cdot,\cdot)$ is continuous and compactly perturbed coercive on $H^{-1}_\Gamma\times  H^{-1}_\Gamma$, i.e., for some constants $C_a, \alpha>0$, and some compact sesquilinear form $\tilde{a}(\cdot,\cdot)$,
\begin{equation} \label{eq:ContCoer}
|a(\phi,\psi)| \leq C_a\|\phi\|_{H^{-1}_\Gamma}\, \|\psi\|_{H^{-1}_\Gamma}, \quad |a(\phi,\phi)-\tilde{a}(\phi,\phi)|\geq \alpha \|\phi\|_{H^{-1}_\Gamma}^2, \quad \phi,\psi\in H^{-1}_\Gamma.
\end{equation}
\end{lem}
\begin{proof}
Continuity follows immediately from \eqref{eq:NewtonMapping}.
Let us temporarily introduce the notations $\cA_\kappa$ and $A_\kappa$ to denote the operators $\cA$ and $A$ with $k$ replaced by some complex wavenumber $\kappa$.
To prove that $a(\cdot,\cdot)$ is compactly perturbed coercive, we split the associated operator $A=A_k$ as $A_k=A_{\ri}+(A_k-A_{\ri})$, where $A_\ri$ is the operator with wavenumber $\kappa=\ri$.
It is easy to check that for $\phi\in C^\infty_0(\R^n)$ the Fourier transform of $\cA_\ri\phi$ is given by $\hat{\phi}(\xi)/(|\xi|^2+1)$, which gives that
\[ \langle A_\ri\phi,\phi\rangle_{H^{1}(\R^n)\times H^{-1}(\R^n)} = \int_{\R^n} \frac{|\hat\phi(\xi)|^2}{|\xi|^2+1}\,\rd \xi = \|\phi\|^2_{H^{-1}(\R^n)}, \qquad \phi\in H^{-1}_{\Gamma}, \]
so that $A_\ri:H^{-1}_{\Gamma}\to \tH^{1}(\Omega)^\perp$ is coercive %
with coercivity constant $\alpha=1$.
Further, for all $\sigma \in C^\infty_{0,\Gamma}$,
$A_k-A_\ri = P\sigma (\cA_k-\cA_\ri)$ and, arguing as in \cite[Rem.~3.1.3]{sauter-schwab11} and \cite[Prop.~8.8]{ClHi:13}, $(\cA_k-\cA_\ri):H_\Gamma^{-1}\to H^{3,{\rm loc}}(\R^n)$ is continuous, so that $\sigma(\cA_k-\cA_\ri):H_\Gamma^{-1}\to H^{1}(\R^n)$ is compact, which implies that $A_k-A_\ri$ is compact.
\end{proof}

\rev{Using Lemma \ref{lem:coer} we can prove well-posedness of the scattering problem for arbitrary compact $\Gamma$, by reformulating it as a well-posed integral equation (IE), which we do in the following theorem. 
Our description of \eqref{eq:IE} as an IE will be justified in \S\ref{sec:IEdSet}, when we discuss conditions under which the operator $A$ 
can be interpreted as an integral operator on $\Gamma$. }

\begin{thm}%
\label{thm:equivalence}
Let $\Gamma\subset\R^n$ be compact, and suppose that Assumption \ref{ass:Uniqueness} %
holds. Then $A:H^{-1}_\Gamma\to \tH^1(\Omega)^\perp$ is invertible. Further, for every $g\in \tH^{1}(\Omega)^\perp$ the problem defined by \eqref{eqn:HE} in $\Omega$, \eqref{eqn:SRC}, and \eqref{a1bc} has a unique solution $u\in H^{1,{\rm loc}}(\R^n)$ given by
\begin{align}
\label{eq:Rep}
u = \cA \phi,
\end{align}
where
$\phi\in H^{-1}_\Gamma$ is the unique solution of the %
IE
\begin{align}
\label{eq:IE}
A\phi = g,
\end{align}
\rev{which can be written equivalently in variational form as %
\begin{align}
\label{eqn:VariationalCts}
a(\phi,\psi)=\langle g,\psi\rangle_{H^{1}(\R^n)\times H^{-1}(\R^n)}, \qquad \forall \psi\in H^{-1}_{\Gamma}.
\end{align}
}
If Assumption \ref{ass:Uniqueness} does not hold then $A\phi=0$ has a non-trivial solution $\phi\in H_\Gamma^{-1}$.
\end{thm}
\begin{proof}

If $\phi\in H^{-1}_\Gamma$ satisfies \eqref{eq:IE} then $u$ given by \eqref{eq:Rep} solves the scattering problem, by \eqref{eq:NewtonHelmholtz} and \eqref{eq:ADef}, and the fact that the acoustic Newton potential \eqref{eq:NPPrep} satisfies 
\eqref{eqn:SRC}. 
Thus if \eqref{eq:IE} has a solution then the scattering problem has a solution, and this solution is unique if Assumption \ref{ass:Uniqueness} holds.
By Lemma \ref{lem:coer}, $A$ is continuous and compactly perturbed coercive, and hence Fredholm of index zero by Lax-Milgram. Thus, to prove $A:H^{-1}_\Gamma\to \tH^1(\Omega)^\perp$ is invertible, so that \eqref{eq:IE} has a unique solution, it suffices to prove that $A$ is injective. For this, 
suppose that $\phi\in H^{-1}_\Gamma$ and $A\phi = 0$. Then $\cA\phi$ satisfies the homogeneous scattering problem, so, if Assumption \ref{ass:Uniqueness} holds, we have $\cA\phi = 0$ in $\Omega$. But, for $\sigma\in C^\infty_{0,\Gamma}$, we also have $P(\sigma\cA\phi) = A\phi = 0$, so $\sigma\cA\phi\in \tH^1(\Omega)$. Thus, \rev{by \cite[eqn.~(17)]{ChaHewMoi:13},} 
$\cA\phi=\sigma\cA\phi=0$ in $\Omega^c=\Gamma$, almost everywhere with respect to $n$-dimensional Lebesgue measure. Hence %
$\cA\phi = 0$, and by \eqref{eq:NewtonHelmholtz} we conclude that $\phi = 0$, 
proving injectivity, and hence invertibility of $A$.

If Assumption \ref{ass:Uniqueness} does not hold then, by Remark \ref{rem:Ass}, where $\Omega_-$ is as in that remark, there exists a non-zero $v\in \tH^1(\Omega_-)\subset \tH^1(\Omega)$ such that $\Delta v + k^2 v=0$ in $\Omega_-$. By \eqref{eq:NewtonHelmholtz} we have $v=\cA\phi$, where $\phi:=-(\Delta+k^2)v\in H^{-1}_{\partial \Omega_-}\subset H^{-1}_\Gamma$. Further, for $\sigma\in C^\infty_{0,\Gamma}$, $P(\sigma v)=0$ since $\sigma v\in \tH^1(\Omega)$, so $A\phi=P(\sigma\cA\phi)=0$, and $\phi\neq 0$ since $\cA\phi=v\neq 0$. 
\end{proof}

\begin{rem}[\bf The role of the capacity of $\Gamma$] \label{rem:cap}
We make the trivial observation that, if $H_\Gamma^{-1}=\{0\}$, then the only solution to \eqref{eq:IE} is $\phi=0$, so that the scattered field $u=\cA\phi = 0$; i.e.\ the incident field does not interact with $\Gamma$. Further, 
$H^{-1}_\Gamma\neq\{0\}$ 
if and only if %
$\Gamma$ has positive $H^1$ capacity  (see, e.g.\ \cite[Thm 13.2.2]{Maz'ya}, and for a collection of related results and generalisations, see \cite{HewMoi:15,ChaHewMoi:13}). This holds if $\dimH(\Gamma)>n-2$ \cite[Thm 2.12]{HewMoi:15}, and is equivalent to $\dimH(\Gamma)>n-2$ if $\Gamma$ is a $d$-set \cite[Thm 2.17]{HewMoi:15}.
\rev{Moreover, by \cite[Thm.~3.12]{ChaHewMoi:13}, $H^{-1}_\Gamma=\{0\}$ if and only if $\tH^1(\Omega)^\perp=\{0\}$, so if $H^{-1}_\Gamma=\{0\}$ then the datum $g$ of the IE \eqref{eq:IE} is zero (recall the definition of $g$ in \eqref{eq:Porth}, \eqref{eqn:gDefScatteringProblem}).}
\end{rem}

The following proposition concerns the support of the IE solution, and allows us to determine when the scattered fields and IE solutions for different scatterers $\Gamma$ coincide. 

\begin{prop}
\label{prop:components2}
Let $\Gamma\subset\R^n$ be compact \rev{with non-empty interior $\Gamma^\circ$} and let $\Omega:=\Gamma^c$ be connected. Let
$u=\cA\phi$ be the unique solution of the scattering problem
for $\Gamma$, 
where $\phi\in H_\Gamma^{-1}$ is the unique solution of \eqref{eq:IE}, with $g$ given by \eqref{eqn:gDefScatteringProblem}. Then $\phi\in H^{-1}_{\partial \Gamma}$. Suppose further that $\Gamma_\dag$ is compact, with $\partial \Gamma \subset \Gamma_\dag\subset \Gamma$, and that Assumption \ref{ass:Uniqueness} is 
satisfied  by $\Gamma_\dag$, and let 
$u_\dag=\cA\phi_\dag$ 
be the unique solution of the scattering problem for $\Gamma_\dag$, 
where $\phi_\dag\in H_{\Gamma_\dag}^{-1}$ is the unique solution of the IE for $\Gamma_\dag$. Then
$u_\dag=u$ and $\phi_\dag=\phi\in H^{-1}_{\partial \Gamma}$.
\end{prop}
\begin{proof}
Since $\mathcal{A}\phi+u^i=u+u^i=u^t\in \tH^{1,{\rm loc}}(\Omega)$, 
\rev{it holds in $\Gamma^\circ$ that}
\begin{equation} \label{eq:phizero}
0 = (\Delta+k^2)u^t= (\Delta +k^2)\mathcal{A}\phi +  (\Delta +k^2)u^i=-\phi, %
\end{equation}
by \eqref{eq:NewtonHelmholtz}, and since $u^i$ satisfies \eqref{eqn:HE} in a neighbourhood of $\Gamma$. Thus $\phi\in H_{\partial \Gamma}^{-1}$.
Further, $u^t\in \tH^{1,{\rm loc}}(\Omega)$ implies that $u^t=0$, so $u=-u^i$, in $\Gamma^\circ$. Since $\Omega\subset \Omega_\dag\subset \Omega \cup \Gamma^\circ$, where $\Omega_\dag:=\Gamma_\dag^c$, and $u^i$ satisfies \eqref{eqn:HE} in a neighbourhood of $\Gamma$, it follows that $u^t\in \tH^{1,{\rm loc}}(\Omega_\dag)$ and that $u$ satisfies \eqref{eqn:HE} in $\Omega_\dag$. Since $u$ also satisfies \eqref{eqn:SRC} it follows from Assumption \ref{ass:Uniqueness} for $\Gamma_\dag$ that $u_\dag=u$. Since $u-u_\dag = \mathcal{A}(\phi-\phi_\dag)$, it follows from \eqref{eq:NewtonHelmholtz} that $\phi=\phi_\dag$. 
\end{proof}

\begin{rem}
\label{rem:support}
\rev{Proposition \ref{prop:components2} implies that if, for a given $k>0$, Assumption \ref{ass:Uniqueness} holds for a scatterer $\Gamma$ for which $\Omega:=\Gamma^c$ is not connected, then the scattered field and the IE solution $\phi$ for 
the scatterer 
$\Gamma$ coincide with those for the scatterer $\Omega_+^c$, where $\Omega_+$ is the unbounded component of $\Gamma^c$, and $\phi$ is supported in $\partial\Omega_+$.}
\end{rem}

\begin{rem}[\bf Alternative IEs for the same scattering problem] \label{rem:Alt}
Consider the case where $\Omega:=\Gamma^c$ is connected and $\Gamma^\circ$ is non-empty (e.g., as in Figure \ref{fig:2Dexs}(a) and (g)). By Proposition \ref{prop:components2} and Remark \ref{rem:support}, to solve the scattering problem for $\Gamma$ we can solve the IE on $\Gamma_\dag$, for any compact $\Gamma_\dag$ with $\partial \Gamma \subset \Gamma_\dag \subset \Gamma$ provided $k^2$ is not a Dirichlet eigenvalue of $-\Delta$ in $\Omega_-:= \Gamma \setminus \Gamma_\dag$, in particular if $0<k<k_0$ where $k_0^2$ is the smallest such eigenvalue. Recall from Theorem \ref{thm:equivalence} that \rev{satisfying} the IE on $\Gamma_\dag$ is equivalent to requiring that $u^t=\cA\phi+u^i\in \tH^1(\Gamma_\dag^c)$, i.e.\ to enforcing $\cA\phi=-u^i$ on $\Gamma_\dag$.

Since $\phi\in H_{\partial \Gamma}^{-1}$, the  choice $\Gamma_\dag=\partial \Gamma$ is natural; with this choice the IE enforces $\cA\phi=-u^i$ on $\partial \Gamma$. When $\Omega$ is a Lipschitz open set this corresponds, as we discuss in Remark \ref{rem:2nd} below, to the standard single-layer-potential boundary IE (BIE) formulation. But it may be attractive to choose a larger $\Gamma_\dag$, so that $\cA\phi = - u^i$ is enforced not just on $\partial \Gamma$ but also at points in $\Gamma^\circ$. This reduces the size of $\Omega_-$ and so increases $k_0$, 
and hence the interval $(0,k_0)$ in which the IE is uniquely solvable. (This is the rationale behind the CHIEF method and its variants for removing irregular frequencies of BIEs, e.g., \cite{Schenck:68}, \cite{WuSe:91}.) For the largest choice,  $\Gamma_\dag=\Gamma$, $\cA\phi=-u^i$ is enforced on the whole of $\Gamma$, and the IE is uniquely solvable for all $k>0$, but at the cost in  computation of discretising the whole of $\Gamma$ rather than $\partial \Gamma$ or some intermediate set. 

We explore this further in \S\ref{sec:Numerics}, where we compare computations for the choices $\Gamma_\dag=\Gamma$ and $\Gamma_\dag=\partial\Gamma$ for the particular example of the Koch snowflake (Figure \ref{fig:2Dexs}(g)) - see the discussion around Figure \ref{fig:KochSnowflake} below.
\end{rem}

\rev{The variational formulation}  \eqref{eqn:VariationalCts} 
will be the starting point for our Galerkin discretisation in \S\ref{sec:HausdorffIEM}.
Having computed $\phi$ by solving a Galerkin discretisation of \eqref{eqn:VariationalCts}, we will evaluate $u(x)$ at points $x\in \Omega$ using the formulas \eqref{eq:Rep}/\eqref{eq:NPdef}. We will also compute the far-field pattern $u^\infty\in C^\infty(\IS^{n-1})$, which satisfies
(see, e.g., \cite[Eqn.~(2.23)]{ChGrLaSp:11}, \cite[p.~294]{McLean})
$$
u(x) = \frac{\re^{\ri k|x|}}{|x|^{(n-1)/2}}\left(u^\infty(\hat x)+ O(|x|^{-1})\right), \quad \mbox{as} \quad |x|\to \infty,
$$
uniformly in $\hat x:= x/|x|$. Explicitly (\cite[Eqn.~(2.23)]{ChGrLaSp:11}, \cite[p.~294]{McLean}),
\begin{equation} \label{eq:ffpattern}
u^\infty(\hat x) = \langle \sigma \Phi^\infty(\hat x,\cdot), \overline{\phi}\rangle_{H^{1}(\R^n)\times H^{-1}(\R^n)}, \quad \hat x\in \IS^{n-1},
\end{equation}
where $\sigma$ is any element of $C^\infty_{0,\Gamma}$ and $\Phi^\infty(\cdot,y)$ is the far-field pattern of $\Phi(\cdot,y)$, for $y\in \R^{n}$, viz.
\begin{equation} \label{eq:FFps}
\Phi^\infty(\hat x,y):= \frac{\ri k^{(n-3)/2}}{2(2\pi \ri)^{(n-1)/2}}\, \exp(-\ri k\hat x\cdot y), \quad \hat x\in \IS^{n-1}, \; y\in \R^{n}.
\end{equation}

\subsection{The \rev{integral equation} on \texorpdfstring{$d$}{d}-sets in trace spaces} \label{sec:IEdSet}

\rev{So far, our analysis has been for general compact scatterers $\Gamma\subset\R^n$. 
We now assume additionally that $\Gamma$ is a $d$-set (in the sense of \eqref{eq:dset}), and that $n-2<d\leq n$, so we have a non-trivial scattered field (see Remark \ref{rem:cap}). 
In this case 
one can  
view the operator $A$ as an integral operator with respect to Hausdorff measure $\cH^d$, by reinterpreting $A$ as a map between certain ``trace spaces'' on $\Gamma$. This will allow us to relate our IE \eqref{eq:IE} to previously studied IE formulations in certain special cases (see Remark \ref{rem:2nd}), and will pave the way for the discretization we consider in \S\ref{sec:HausdorffIEM}. 
We begin by briefly recalling the definition of trace spaces on $d$-sets, %
and the relationship between them and function spaces on $\R^n$. 
For a more detailed explanation 
see \cite[\S2.4]{HausdorffBEM}. 
Let $\Gamma\subset\R^n$ be a $d$-set for some \rev{$n-2<d\leq n$}.
\footnote{\rev{While our focus here is on $n=2,3$ and $n-2<d\leq n$, we note that the definitions and results from the current paragraph onwards, up to and including Lemma \ref{lem:density}, all extend to general $n\in\N$ and $0<d\leq n$. For details see, e.g., \cite[\S2.4]{HausdorffBEM}.}} 
We denote by $\IL_2(\Gamma)$ the Hilbert space of %
functions on $\Gamma$ that are \rev{measurable and} square integrable with respect to
$\cH^d|_\Gamma$, normed by
$\|f\|_{\IL_2(\Gamma)} := (\int_\Gamma |f(x)|^2\,\rd \cH^d(x))^{1/2}$, and by $\IL_\infty(\Gamma)$ the Banach space of 
functions on $\Gamma$ 
that are \rev{measurable and} essentially bounded with respect to $\mathcal{H}^d|_\Gamma$, normed by $\|f\|_{\IL_\infty(\Gamma)} := {\rm{ess}}\,\sup_{x\in \Gamma} |f(x)|$.

Let $\tr:C_0^\infty(\R^n)\to \IL_2(\Gamma)$ be the trace (or restriction) operator, with dense range, defined by $\tr \varphi=\varphi|_\Gamma\in \IL_2(\Gamma)$, for $\varphi\in C_0^\infty(\R^n)$.
For $s>\frac{n-d}{2}$, this extends to a continuous linear operator
\begin{align}
\label{def:tr}
\tr:H^{s}(\R^n) \to \IL_2(\Gamma),
\end{align}
also with dense range (see \cite[Thm~18.6]{Triebel97FracSpec} and \cite[\S2.4]{HausdorffBEM}).
Setting
\begin{equation} \label{eq:st}
t:=s-\frac{n-d}{2} >0,
\end{equation}
we define the trace space $\IH^{t}(\Gamma):=\tr(H^{s}(\R^n))\subset \IL_2(\Gamma)$, and equip it with the quotient norm
\[ \|f\|_{\IH^t(\Gamma)} := \inf_{\substack{\varphi\in H^s\Rn\\ \tr \varphi=f}}\|\varphi\|_{H^s\Rn}.\]
This makes $\IH^t(\Gamma)$ a Hilbert space unitarily isomorphic to the quotient space $H^{s}(\R^n)/\ker(\tr)$. 
For $t>0$ we denote by $\IH^{-t}(\Gamma)$ the dual space $(\IH^{t}(\Gamma))^*$. Identifying $\IL_2(\Gamma)$ with its dual in the standard way, and with $\IH^0(\Gamma):=\IL_2(\Gamma)$, we have that
$\IH^{t^\prime}(\Gamma)$ is continuously embedded in $\IH^{t}(\Gamma)$ with dense image for any $t,t^\prime\in\R$ with $t^\prime>t$,
and if
$g\in\IH^t(\Gamma)$ for some $t\geq 0$ and $f\in\IL_2(\Gamma)$ then
\begin{align}
\label{eq:L2dualequiv}
\langle f,g\rangle_{\IH^{-t}(\Gamma) \times \IH^{t}(\Gamma)}=(f,g)_{\IL_2(\Gamma)}.
\end{align}

Assuming \eqref{eq:st}, $\tr:H^{s}(\R^n) \to \IH^t(\Gamma)$ is a continuous linear surjection that
has unit norm and is a unitary isomorphism from $\ker(\tr)^\perp$ to $\IH^t(\Gamma)$. 
Furthermore, $\tH^s(\Gamma^c)\subset \ker(\tr)$. 
Accordingly, \rev{again assuming \eqref{eq:st},} its adjoint
\begin{align}
\label{def:trstar}
\tr^*: \IH^{-t}(\Gamma) \to H^{-s}(\R^n),
\end{align}
is a continuous linear isometry with range contained in $H^{-s}_\Gamma=(\tH^s(\Gamma^c)^\perp)^*$, satisfying
\begin{equation} \label{eq:tracedual}
 \langle \varphi,\tr^* f \rangle_{H^{s}(\R^n) \times H^{-s}(\R^n)} = \langle \tr \varphi,f\rangle_{\IH^{t}(\Gamma)\times \IH^{-t}(\Gamma)},
\qquad f\in\IH^{-t}(\Gamma),\,\varphi\in H^s(\R^n).
\end{equation}
In particular, when $f\in \IL_2(\Gamma)$ we have that
\begin{align}
\label{eq:L2dualrep}
\langle \varphi,\tr^* f \rangle_{H^{s}(\R^n) \times H^{-s}(\R^n)} =(\tr \varphi,f)_{\IL_2(\Gamma)}.
\end{align}

The following theorem is a slight 
\rev{modification} 
of results presented in \cite{caetano2019density} and \cite{HausdorffBEM}.
}
\begin{thm}
\label{thm:density}
\rev{
Let $\Gamma\subset \R^n$ be a $d$-set for some \rev{$n-2<d\leq n$}. If $\frac{n-d}{2}<s<\frac{n-d}{2}+1$, so that $t:=s-\frac{n-d}2\in(0,1)$, then for $\tr:H^s(\R^n)\to \IH^t(\Gamma)$ it holds that $\ker(\tr)=\tH^{s}(\Gamma^c)$. 

\noindent 
Hence $\tr|_{\tH^{s}(\Gamma^c)^\perp}:\tH^{s}(\Gamma^c)^\perp \to \IH^{t}(\Gamma)$ is a unitary isomorphism, 
the range of $\tr^*$ is equal to $H^{-s}_\Gamma$, the map
$\tr^*:\IH^{-t}\GG\to H^{-s}_\Gamma$ is a unitary isomorphism, and $\tr^*(\IL_2(\Gamma))$ is dense in $H^{-s}_\Gamma$. 

\noindent 
If $d=n$ (in which case $t=s$) then the above statements hold also for the limiting \rev{case} $s=t=0$.}
\end{thm}
\begin{proof}
\rev{
For the statements for $\frac{n-d}{2}<s<\frac{n-d}{2}+1$ see \cite[Prop.~6.7, Thm~6.13]{caetano2019density} and \cite[Thm 2.7]{HausdorffBEM}). For the limiting case mentioned, we note that when $d=n$ we have $\IL_2(\Gamma)=L_2(\Gamma)$, and the trace map $\tr:H^s(\R^n)\to \IL_2(\Gamma)$ is continuous for all $s\geq 0$ and is given simply by $\tr\phi = \phi|_\Gamma$, so that $\ker(\tr)$ is the set of functions in $H^s(\R^d)$ that vanish almost everywhere on $\Gamma$ with respect to Lebesgue measure. For $s=0$, that $\ker(\tr)=\tH^{0}(\Gamma^c)$ is then a consequence of \cite[Lem.~3.16]{ChaHewMoi:13}.} 
\end{proof}

\begin{rem}[\bf Connection to known cases: I] \label{rem:conI}
The spaces $\IH^t(\Gamma)$ introduced above can be related to well-known trace spaces in special cases. For instance:
\begin{enumerate}
\item[(a)] \label{item:one}
If $\Gamma$ is the closure of a bounded Lipschitz open set $D$ (e.g.\ Figure \ref{fig:2Dexs}(a)), then \rev{$\Gamma$ is a $d$-set with $d=n$, so that $t=s$, and $\IH^t(\Gamma)$ coincides with the restriction space $H^t(D)=\{U|_\Gamma:U\in H^t(\R^n)\}$ for $t\geq 0$  
(this follows from the fact that $\tr$ coincides with the restriction operator $|_\Gamma$ in this case, as noted in the proof of Theorem \ref{thm:density} above).}   
Hence, for $t<0$, $\IH^t(\Gamma)$ is unitarily isomorphic to 
the space $\tH^{t}(D)=(H^{-t}(D))^*$ \rev{(see, e.g.,\ \cite[Thm 3.30]{McLean})}.
\item[(b)]
If $\Gamma$ is the boundary of a bounded Lipschitz open set (e.g.\ Figure \ref{fig:2Dexs}(b)), then \rev{$\Gamma$ is a $d$-set with} $d=n-1$ and $\IH^t(\Gamma)$ coincides, for $|t|\leq 1/2$, with the boundary Sobolev space $H^t(\Gamma)$, as defined, e.g., in \cite[pp.~98--99]{McLean}; see \cite[Rem.~6.5]{caetano2019density} for details.

\item[(c)]  
A further example of a family of $d$-sets with $d=n-1$ is provided by the ``multi-screens'' defined in Definition 2.3 of \cite{ClHi:13}; these are finite unions of Lipschitz subsets of the boundaries of bounded Lipschitz open sets, a specific example being given in Figure \ref{fig:2Dexs}\rev{(d)}. %
In \cite{ClHi:13}, trace spaces on multi-screens are defined as quotient spaces. %
Specifically, in the parlance of multi-screen theory, the space $H^{1/2}([\Gamma]) := H^1(\mathbb{R}^n)/\widetilde{H}^1(\Gamma^c)$
is referred to as the ``Dirichlet single-trace space'' (see \cite[Defn 6.1]{ClHi:13}).  
It is unitarily isomorphic to the
space $\mathbb{H}^{1/2}(\Gamma)$.
The dual space $\tH^{-1/2}([\Gamma]) := (H^{1/2}([\Gamma]))^*$ is
referred to as the ``Neumann jump space'' (see \cite[Defn 6.4]{ClHi:13}). It is unitarily isomorphic to $\mathbb{H}^{-1/2}(\Gamma)= (\mathbb{H}^{1/2}(\Gamma))^*$, 
which, \rev{by Theorem} \ref{thm:density}, is unitarily isomorphic to $H^{-1}_{\Gamma}$.
\end{enumerate}
\end{rem}

\rev{In \cite{HausdorffBEM} we studied scattering by planar screens $\Gamma \subset \Gamma_{\infty}:=\mathbb{R}^{n-1}\times\{0\}$ (see Figure \ref{fig:2Dexs}(c) and (e)). There we defined trace spaces on $\Gamma$ by a two-step process, first taking a trace onto the hyperplane $\Gamma_\infty$, then applying the above trace results in $\R^{n-1}$. 
The next result, which is stated without proof since it follows trivially from standard trace mapping properties (see e.g.~\cite[Lem.~3.35]{McLean}), allows us to avoid this complication, and instead treat a planar screen as any other compact set.} 

\begin{lem}[\bf Planar screens]\label{ConnectionPlanarScreens}
\rev{Suppose that $\Gamma = \widetilde{\Gamma}\times\{0\}\subset
  \Gamma_{\infty}:=\mathbb{R}^{n-1}\times\{0\}$ where
  $\widetilde{\Gamma}\subset \mathbb{R}^{n-1}$ is a $d$-set in $\mathbb{R}^{n-1}$ for some \rev{$n-2<d\leq n-1$}. 
\rev{Then $\Gamma$ is a $d$-set in $\R^n$.} 
Further, noting from \eqref{eq:st} that $t = s-(n-d)/2 = (s-1/2) - (n-1-d)/2$, \ for $s>(n-d)/2$ we have \rev{one continuous trace operator} ${\rm
    tr}_{\widetilde\Gamma}:H^{s-1/2}(\R^{n-1})\to \IL_2(\widetilde
  \Gamma)$ with range $\IH^t(\widetilde\Gamma)$, and another \rev{continuous trace operator}
  \mbox{$\tr:H^{s}(\R^{n})\to \IL_2(\Gamma)$} with range $\IH^t(\Gamma)$.
The spaces  $\IH^t(\Gamma)$ and $\IH^t(\widetilde\Gamma)$ coincide up to identification by
  the map $f\mapsto f(\cdot,0)$, and the traces satisfy
  $\tr = {\rm tr}_{\widetilde\Gamma}\circ\gamma$
  where  $\gamma:H^{s}(\R^{n})\to H^{s-1/2}(\R^{n-1})$ is the standard (surjective) trace 
  operator.}

\end{lem}

\rev{Of central importance is the case $s=1$, as our scattering problem is posed in $H^{1,{\rm loc}}(\R^n)$. Hence we give the value of $t$ given by \eqref{eq:st} for $s=1$ its own notation, defining (for $n-2<d\leq n$)
\begin{align}
\label{eq:tdDef}
t_d:= 1-\frac{n-d}{2} \in(0,1].
\end{align} 
The case $s=1$ ($t=t_d$) is covered by Theorem \ref{thm:density} for $n-2<d<n$, but not for $d=n$. It is not known to us whether $\ker(\tr)=\tH^1(\Gamma^c)$ for general $n$-sets $\Gamma$, although we know it to hold in many cases (see Remark \ref{rem:TildeCirc}). For convenience we introduce this as an assumption, with which assumption Corollary \ref{cor:td} below is an immediate consequence of Theorem \ref{thm:density}. }

\rev{
\begin{ass}
\label{ass:TildeCirc}
$\Gamma\subset\R^n$ is a 
$d$-set with either (i) $n-2<d<n$ or (ii) $d=n$ and $\ker(\tr)=\tH^1(\Gamma^c)$.
\end{ass}

\begin{cor}
\label{cor:td}
Suppose that Assumption \ref{ass:TildeCirc} holds. Then, 
\rev{with $t_d$ defined by \eqref{eq:tdDef},} 
$\tr:\tH^{1}(\Gamma^c)^\perp \to \IH^{t_d}(\Gamma)$ and $\tr^*:\IH^{-t_d}(\Gamma)\to H^{-1}_\Gamma$ are unitary isomorphisms and $\tr^*(\IL_2(\Gamma))$ is dense in $H^{-1}_\Gamma$. 
\end{cor}

\begin{rem}
\label{rem:TildeCirc}
\rev{In the case $d=n$,} a sufficient condition for $\ker(\tr)=\tH^1(\Gamma^c)$ 
is that $\tH^{1}(\Gamma^c)= H^1_{\overline{\Gamma^c}}$, since $\tH^{1}(\Gamma^c)\subset \ker(\tr) \subset H^1_{\overline{\Gamma^c}}$ (see, e.g.,\ \cite[Eqn \rev{(17)}]{ChaHewMoi:13}). Hence Assumption \ref{ass:TildeCirc} holds in the following cases: 
\begin{itemize}
\item [(i)] 
$\Gamma^c$ is $C^0$ \cite[Thm.~3.29]{McLean}, so in particular if $\Gamma$ is the closure of a Lipschitz open set; more generally, if 
$\Gamma^c$ is $C^0$ except at points in a closed, countable subset $P$ of $\partial(\Gamma^c)$ with at most finitely many limit points \cite[Thm.~3.24]{ChaHewMoi:13}; see, e.g., the examples in \cite[Fig.~4]{ChaHewMoi:13};
\item [(iii)] $\Gamma$ is an $n$-set OSC-IFS attractor (see \cite{CaChGiHe23}), e.g.\ the Koch snowflake in Figure \ref{fig:2Dexs}(g);
\item [(iv)] $\Gamma$ is the closure of one of the classical snowflake domains in \cite[\S5.1]{caetano2019density}.
\end{itemize}

\end{rem}
}

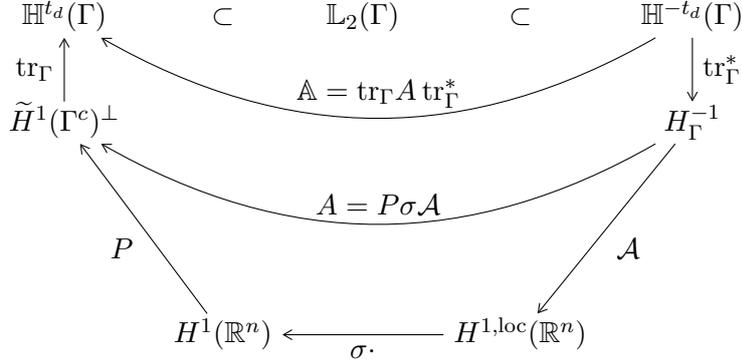
\begin{figure}[t]
\centering
\begin{tikzpicture}
\matrix[matrix of math nodes,
column sep={13pt},
row sep={40pt,between origins},
text height=1.5ex, text depth=0.25ex] (s)
{
|[name=Ht]| \IH^{t_d}\GG &
|[name=]| \subset &
|[name=L2]| \IL_2\GG &
|[name=]| \subset &
|[name=Hmt]| \IH^{-t_d}\GG
\\
|[name=Hscp]| \tH^1(\Gamma^c)^\perp&
& & &
|[name=HmsG]| H^{-1}_\Gamma&
\\
&&&&
\\
&
|[name=W]| H^{1}(\R^{n})
&&
|[name=Wloc]| H^{1,{\rm loc}}(\R^{n})
\\
};
\draw[->,>=angle 60] %
(Hscp) edge node[auto]{\(\tr\)} (Ht)
(Hmt) edge node[auto]{\(\tr^*\)} (HmsG)
(HmsG) edge node[auto]{\(\cA\)} (Wloc)
(Wloc) edge node[auto]{\(\sigma\cdot\)} (W)
(W) edge node[auto]{\(P\)} (Hscp)
(HmsG) edge[bend left=30] node[auto,swap]{\( A = P\sigma\cA\)} (Hscp)
(Hmt) edge[bend left=30] node[auto,swap]{\( \IA = \tr A\, \tr^*\)} (Ht)
;
\end{tikzpicture}
\caption{Schema of relevant function spaces and operators for $s=1$, $t=t_d:=1-\frac{n-d}2\rev{\in(0,1]}$.
The operators $\tr$ and $\tr^*$ are isometries, indeed unitary isomorphisms 
\rev{if Assumption \ref{ass:TildeCirc} holds}.
}
\label{fig:AndreaScheme}
\end{figure}

\rev{
Returning to %
our integral equation, 
we can now ``lift'' the potential $\cA:H^{-1}_\Gamma\to H^{1,{\rm loc}}(\R^n)$ and the operator $A:H^{-1}_\Gamma\to\tH^{1}(\Omega)^\perp$ to continuous maps on the trace spaces via the compositions
\begin{align}
\notag
\cA\tr^*:\IH^{-t_d}(\Gamma) \to H^{1,{\rm loc}}(\R^n)
\qquad \text{and}
\\
\label{eqn:IADef}
\IA:=\tr\, A\, \tr^*:\IH^{-t_d}\GG\to\IH^{t_d}\GG.
\end{align}
The next lemma collects some basic results about \rev{these compositions}. 
A schematic showing the relationships between the main function spaces and operators involved is given in Figure \ref{fig:AndreaScheme}.}

\rev{
\begin{lem}
\label{lem:IA}
Let $\Gamma$ be a compact $d$-set with $n-2<d\leq n$, and define $\IA$ as in \eqref{eqn:IADef}. Then:
\begin{itemize}
\item[(i)]
For arbitrary $\sigma \in C^\infty_{0,\Gamma}$,
\begin{equation} \label{eq:Arep}
\IA\bbx = \tr (\sigma\cA\tr^*\bbx), \quad \bbx \in \IH^{-t_d}(\Gamma);
\end{equation}
\item[(ii)]
The sesquilinear form on $\IH^{-t_d}\GG\times \IH^{-t_d}\GG$ associated with the operator $\IA$ satisfies
\begin{equation} \label{eqn:SesquiA}
\langle \IA \bbx,\bby\rangle_{\IH^{t_d}\GG\times\IH^{-t_d}\GG} = a(\tr^*\bbx,\tr^*\bby), \quad \bbx,\bby\in \IH^{-t_d}(\Gamma),
\end{equation}
with $a(\cdot,\cdot)$ as in \eqref{eqn:Sesqui}, 
and 
is continuous and compactly perturbed coercive;
\item[(iii)]
If 
Assumption \ref{ass:TildeCirc} holds, 
then 
problem \eqref{eq:IE}/\rf{eqn:VariationalCts} can be equivalently stated as: given $g\in \tH^{1}(\Omega)^\perp$, find $\bbx\in \IH^{-t_d}(\Gamma)$ such that
\begin{align}
\label{eq:IEnew}
\IA \Psi = \tr g,
\end{align}
or, equivalently,
\begin{equation} \label{eq:varformnew}
\langle \IA \bbx,\bby\rangle_{\IH^{t_d}\GG\times\IH^{-t_d}\GG}=\langle \tr g,\bby\rangle_{\IH^{t_d}\GG\times \IH^{-t_d}\GG},
\qquad \bby\in \IH^{-t_d}\GG,
\end{equation}
with the solutions of \eqref{eq:IE}/\rf{eqn:VariationalCts} and \eqref{eq:IEnew}/\eqref{eq:varformnew} related by $\phi = \tr^*\Psi$. 
If Assumption \ref{ass:Uniqueness} also holds, then $A:H_\Gamma^{-1}\to \tH^1(\Omega)^\perp$ and $\IA:\IH^{-t_d}(\Gamma)\to \IH^{t_d}(\Gamma)$ are both invertible.
\end{itemize}
\end{lem}
\begin{proof}
(i) Equation \eqref{eq:Arep} follows from \eqref{eq:Porth}, \eqref{eq:ADef}, and the fact that $\tH^{1}(\Omega)\subset \ker(\tr)$,  %
which implies that $\tr P\phi=\tr \phi$ for $\phi\in H^{1}(\R^n)$. 

(ii) Equation \eqref{eqn:SesquiA} follows from \eqref{eqn:Sesqui} and \eqref{eq:tracedual}, and the rest of (ii) follows from \eqref{eqn:SesquiA}, Lemma \ref{lem:coer}, 
and that $\tr^*:\IH^{-t_d}(\Gamma)\to H^{-1}_\Gamma$ is an isometry.%

(iii) The first statement follows from Corollary \ref{cor:td}, and the second from Theorem \ref{thm:equivalence}.
\end{proof}
}

\rev{Crucial for the practical implementation of the Hausdorff IE method described in \S\ref{sec:HausdorffIEM} is the fact that both $\cA\tr^*$ and $\IA$ have integral representations with respect to Hausdorff measure.

\begin{thm}\label{lem:ISasHauss}
Let $\Gamma$ be a compact $d$-set with $n-2<d\leq n$. %
Then:
\begin{itemize}
\item[(i)]
For $\Psi\in\IL_2(\Gamma)$, 
\begin{align}
\label{eq:SLPrepHausdorff}
\cA\tr^*\Psi(\bx)=\int_{\Gamma}\Phi(\bx,\by)\Psi(\by)\,\rd \cH^d(\by), \qquad \bx\in \Omega;
\end{align}
\item[(ii)]
For $\Psi\in \IL_\infty(\Gamma)$, the right-hand side of \eqref{eq:SLPrepHausdorff} is well-defined (as a Lebesgue integral with respect to $\mathcal{H}^d$ measure) for all $x\in\R^n$, and is a continuous function on $\R^n$. Further, \eqref{eq:SLPrepHausdorff} holds for almost all $x\in \R^n$ with respect to $n$-dimensional Lebesgue measure, so that
$\cA\tr^*\Psi\in C(\R^n)$.
\item[(iii)]
For $\Psi\in \IL_\infty(\Gamma)$,
\begin{equation} \label{eq:ISasHauss}
\IA\bbx(x) = \int_{\Gamma} \Phi(x,y) \bbx(y) \rd \cH^d(y),
\qquad \text{ for }\cH^d\text{-a.e.\ }x\in\Gamma.
\end{equation}
\end{itemize}
\end{thm}

\begin{proof}
(i) This is an immediate consequence of \rf{eq:NPdef} and \rf{eq:L2dualrep}.

(ii) For $\Psi\in \IL_\infty(\Gamma)$, that the right-hand-side of \eqref{eq:SLPrepHausdorff} is well-defined for all $x\in \R^n$ and defines a continuous function on $\R^n$ follows as in the proof of \cite[Prop.~4.5]{HausdorffBEM}, using the estimates for convolution integrals with respect to $\mathcal{H}^d$ measure on $d$-sets in  \cite[Rem.~2.2]{HausdorffBEM} \rev{(cf.\ \cite[Lemma 2.18]{kral1980integral})}. For $d<n$, in which case $m(\partial \Gamma)=\mathcal{H}^n(\Gamma)=0$, we have %
from (i) 
that \eqref{eq:SLPrepHausdorff} holds for almost all $x\in \R^n$. For $d=n$, when $\Psi\in \IL_\infty(\Gamma)\subset L_{2}(\R^n)$, that \eqref{eq:SLPrepHausdorff} holds for \rev{almost} all $x\in \R^n$ is just a special case of \eqref{eq:NPPrep}.%

(iii) Suppose $\Psi\in \IL_\infty(\Gamma)$. For $x\in \R^n$ let $G(x)$ denote the right-hand side of \eqref{eq:SLPrepHausdorff}. By 
part (ii) 
and \eqref{eq:NewtonMapping},  $\sigma G\in H^1(\R^n)$ and is continuous, for $\sigma\in C^\infty_{0,\Gamma}$, so that $\tr(\sigma G)=(\sigma G)|_\Gamma=G|_\Gamma$. Since $\mathcal{A}\tr^*\Psi(x)=G(x)$, for a.e.\ $x\in \R^n$ with respect to $n$-dimensional Lebesgue measure by 
part (ii), 
it follows by \eqref{eq:Arep} that $\IA\Psi=G|_\Gamma$ in $\IL_2(\Gamma)$, giving the claimed result.
\end{proof}
}

Our next remark builds on the characterisations of the spaces $\IH^t(\Gamma)$ in Remark \ref{rem:conI}.
\begin{rem}[\bf Connection to known cases: II]  \label{rem:2nd}
\rev{The trace space formulation \eqref{eq:IEnew} of our IE} is familiar in a number of cases, in each of which $\Gamma$ is a $d$-set by Remark \ref{rem:conI} \rev{or Lemma \ref{ConnectionPlanarScreens}.}
\begin{enumerate}
\item[(a)] If $\Gamma$ is the boundary of a bounded Lipschitz open set (see, \rev{e.g.,} Figure \ref{fig:2Dexs}(b) \rev{and Remark \ref{rem:conI}(b)}), then $d=n-1$, \rev{$t_d=1/2$,} and $\cH^d$ coincides with surface measure on $\Gamma$ \cite[Theorem 3.8]{EvansGariepy}. Thus the expression \eqref{eq:SLPrepHausdorff} for $\cA\tr^*\Psi$ coincides with the definition (e.g., \cite[Eqn.~(2.19)]{ChGrLaSp:11}) of the standard single-layer potential with density $\Psi$. 
The representation \eqref{eq:ISasHauss} for $\IA$ %
coincides with the definition \cite[Eqn.~(2.32)]{ChGrLaSp:11} of the single-layer \rev{boundary integral} %
operator $S:H^{-1/2}(\Gamma)\to H^{1/2}(\Gamma)$\rev{, viz.\ $S\phi(x)=\int_\Gamma\Phi(x,y)\phi(x)\rd s(y)$}, 
\rev{and \eqref{eq:Arep} coincides with \cite[Eqn (3.6), Def.~3.15]{sauter-schwab11}.}
The IE \eqref{eq:IEnew} coincides with \cite[Eqn.~(2.63)]{ChGrLaSp:11}. 

\item[(b)] In the case where $\Gamma$ is a multi-screen \rev{(see, \rev{e.g.,} Figure \ref{fig:2Dexs}(d) and Remark \ref{rem:conI}(c)), $d=n-1$, $t_d = 1/2$,} the representation
\eqref{eq:SLPrepHausdorff} for $\mathcal{A}\tr^*$ coincides with
the definition of the single-layer potential given in \cite[Eqn.~(8.2)]{ClHi:13} and the representation \eqref{eq:ISasHauss} for $\IA$ coincides with the
first boundary integral operator from Proposition 8.8 in
\cite{ClHi:13} \rev{(with the same explicit surface integral form as in part (a) above)}. The mapping properties and coercivity up to a
compact perturbation derived in Lemma \ref{lem:IA}(ii) %
generalize the first inequality of \cite[Prop.~8.8]{ClHi:13}.

\item[(c)] If the $d$-set $\Gamma$ is a compact subset of $\Gamma_\infty:=\R^{n-1}\times {0}$  (\rev{see Lemma \ref{ConnectionPlanarScreens}}),
so $\Gamma$ is a planar screen (examples are Figure \ref{fig:2Dexs}(c) and (e)), the expression \eqref{eq:SLPrepHausdorff} for $\cA\tr^*\Psi$ coincides with  \cite[Eqn.~(43)]{HausdorffBEM} and the representation \eqref{eq:ISasHauss} for $\IA$ coincides with \cite[Eqn.~(49)]{HausdorffBEM}.
\end{enumerate}
\end{rem}

The definition and mapping properties of $\tr$ and $\tr^*$, noted in 
\eqref{def:tr} and \eqref{def:trstar},
combined with the representation \eqref{eq:Arep}, enable us to extend the domain of $\IA$ to $\IH^{-t}(\Gamma)$ for $t_d<t<2t_d$ or restrict it to $\IH^{-t}(\Gamma)$ for $0<t<t_d$ as stated in the following result (cf.~\cite[Prop.~4.7]{HausdorffBEM}).

\begin{prop}
\label{lem:cont}
Let $\Gamma$ be a compact $d$-set with $n-2<d \leq n$, and let $|t|<t_d$. Then $\IA: \IH^{t-t_d}(\Gamma)\to \IH^{t+t_d}(\Gamma)$ 
and is continuous. When $d=n$ this holds also for $t=\pm t_d = \pm 1$. %
\end{prop}
\begin{proof}
This follows from \eqref{eq:Arep} and the mapping properties \eqref{eq:NewtonMapping} of $\mathcal{A}$, recalling from \eqref{def:tr} and \eqref{def:trstar}
that, for $s>(n-d)/2$,  the mappings $\tr:H^s(\R^n)\to \IH^t(\Gamma)$ and $\tr^*:\IH^{-t}(\Gamma)\to H_\Gamma^{-s}$ are continuous, where $t=s-(n-d)/2>0$, and that, when $d=n$ and $t=s$, these mappings are continuous also for $s=0$ \rev{(see the proof of Theorem \ref{thm:density})}. 
\end{proof}

\rev{Under certain assumptions, 
$\IA: \IH^{t-t_d}(\Gamma)\to \IH^{t+t_d}(\Gamma)$ is invertible for a range of $t$ around $0$.} 
\begin{prop}
\label{prop:epsilon}
\rev{Let $\Gamma$ be an \rev{OSC-}IFS attractor with  dimension 
$d:=\dimH(\Gamma)$  
such that either \mbox{(a) $\Gamma$ is disjoint} with $n-2<d<n$, (b) $d=n$, or (c) $d=n-1$ and $\Gamma\subset\Gamma_\infty:=\R^{n-1}\times\{0\}$.
If Assumption \ref{ass:Uniqueness} holds,} there exists $0<\epsilon\leq t_d$ such that $\IA: \IH^{t-t_d}(\Gamma)\to \IH^{t+t_d}(\Gamma)$ is invertible for $|t|<\epsilon$.
\end{prop}
\begin{proof}
The claimed invertibility of $\IA: \IH^{t-t_d}(\Gamma)\to \IH^{t+t_d}(\Gamma)$ for a range of $t$ in a neighbourhood of $t=0$ follows by applying a result on interpolation of invertibility of operators (\!{\cite[Prop.~4.7]{mitrea1999boundary}, which quotes \cite{vsneiberg1974spectral}}),
recalling that (i) $\IA: \IH^{t-t_d}(\Gamma)\to \IH^{t+t_d}(\Gamma)$ is bounded for $|t|< t_d$ (Proposition \ref{lem:cont}); (ii) $\IA: \IH^{-t_d}(\Gamma)\to \IH^{t_d}(\Gamma)$ is invertible, as noted below \eqref{eq:varformnew}; and (iii) in the case that $\Gamma$ is disjoint and $d<n$, $\{\IH^t(\Gamma)\}_{|t|<1}$ is an interpolation scale \cite[Cor.~3.3]{HausdorffBEM}; (iv) in the case that $d=n$, 
$\{\IH^t(\Gamma)\}_{t\geq 0}$ and $\{\IH^t(\Gamma)\}_{t\leq 0}$ are interpolation scales \cite{CaChGiHe23}; \rev{(v) in the case $d=n-1$ and $\Gamma\subset\Gamma_\infty:=\R^{n-1}\times\{0\}$ the statement of (iv) is also true, by (iv) applied in $\R^{n-1}$ and Lemma \ref{ConnectionPlanarScreens}.}
\end{proof}

\rev{Invertibility of $\IA: \IH^{t-t_d}(\Gamma)\to \IH^{t+t_d}(\Gamma)$ for some $t>0$ implies a regularity result about the solution 
$\phi$ of the IE \eqref{eq:IE},
provided that the datum $g$ is sufficiently smooth.
}

\begin{rem}[\bf Solution regularity in the $H_\Gamma^s$ scale] \label{rem:regularity}
If there exists $0<t<t_d$ such that $\IA: \IH^{t-t_d}(\Gamma)\to \IH^{t+t_d}(\Gamma)$ is invertible and $\tr g\in \IH^{t+t_d}(\Gamma)$, and if
\rev{Assumption \ref{ass:TildeCirc} holds,}   
then, by \eqref{eqn:IADef} and the mapping properties of $\tr^*$ recalled in \rev{Theorem} \ref{thm:density}, 
the solution $\phi=\tr^*\Psi$ of the IE \rf{eq:IE} satisfies
\begin{align}
\label{eq:phiReg}
\phi\in H^{-1+t}_\Gamma, \quad \mbox{with} \quad \|\phi\|_{H^{-1+t}_\Gamma} \leq C\|\tr g\|_{\IH^{t+t_d}(\Gamma)},
\end{align}
for some constant $C>0$ independent of $\phi$ and $g$.
In the case of scattering of an incident wave $u^i$, in which $g$ is given by \eqref{eqn:gDefScatteringProblem}, we have that $\tr g=-u^i|_\Gamma \in \IH^{t+t_d}(\Gamma)$ for all $0<t<t_d$, since $u^i$ is $C^\infty$ in a neighbourhood of $\Gamma$.
Hence, in this case, 
if the conditions of Proposition \ref{prop:epsilon} hold, then 
\eqref{eq:phiReg} holds for $0< t<\epsilon$, where $\epsilon$ is as in Proposition \ref{prop:epsilon}. 
\end{rem}

\rev{
Given $\Gamma$, determining the largest value of $t$ for which $\IA: \IH^{t-t_d}(\Gamma)\to \IH^{t+t_d}(\Gamma)$ is invertible is an open problem. However, so that we have a theoretical prediction against which to compare our numerical results in \S\ref{sec:Numerics}, we consider the following hypothesis.

\rev{
\begin{hyp}\label{conj}
$\IA: \IH^{t-t_d}(\Gamma)\to \IH^{t+t_d}(\Gamma)$ is invertible for $0\leq t<t_{d'}$, where $t_{d'}:=1-(n-d')/2$, with 
$d':=\dimH(\partial\Omega_+)$ and $\Omega_+$ the unbounded component of $\Omega=\Gamma^c$.
\end{hyp}
}

To give some context for Hypothesis \ref{conj}, we note that, to match Proposition \ref{lem:cont}, a naive hypothesis might be that 
$\IA: \IH^{t-t_d}(\Gamma)\to \IH^{t+t_d}(\Gamma)$ is invertible for all $0\leq t<t_{d}$ (cf.~\cite[Conj.~4.8]{HausdorffBEM} in the planar screen case). 
However, such a hypothesis fails in cases where $d'<d$, such as Figures \ref{fig:2Dexs}(a) (where $d'=1<2=d$) and \ref{fig:2Dexs}(g) (where $d'=\log(4)/\log(3)<2=d$).} 
\rev{Indeed, if this naive hypothesis were to hold,  then, for a scattering problem with $u^i|_\Gamma\neq 0$, arguing as in Remark \ref{rem:regularity} it would follow that $0\neq\phi\in H^s_\Gamma$ for every $s<-(n-d)/2$. 
Furthermore, by Remark \ref{rem:support} we would have that $\supp\phi\subset\partial\Omega_+$, from which it would follow that $0\neq\phi\in H^s_{\partial\Omega_+}$ for every $s<-(n-d)/2$. 
But if $d'<d$ this is impossible because 
$H^s_{\partial\Omega_+}=\{0\}$ for $-(n-d')/2<s<-(n-d)/2$, in fact for $-(n-d')/2\leq s<-(n-d)/2$ if $\partial\Omega_+$ is a $d'$-set \cite[Thms~2.12 \& 2.17]{HewMoi:15}. 
Therefore, Hypothesis \ref{conj} is the strongest hypothesis that is consistent with Remark \ref{rem:support}.}  
\rev{In \S\ref{sec:Numerics} we report numerical results which suggest that Hypothesis \ref{conj} may hold in certain cases, but not in general.}

\section{The Hausdorff-measure IEM}
\label{sec:HausdorffIEM}
We now define and analyse our Hausdorff-measure Galerkin IEM. %
To begin with, let us assume that $\Gamma$ is a compact $d$-set for some $n-2<d\leq n$.
Given $N\in \N$ let $\{T_j\}_{j=1}^N$ be a \emph{mesh} of $\Gamma$, a collection of $\cH^d$-measurable subsets of $\Gamma$ (the \emph{elements}) such that
\[\Gamma = \bigcup_{j=1}^N T_j, \quad \cH^d(T_j)>0 \text{ for }j=1,\ldots,N, \quad\text{ and }\cH^d(T_j\cap T_{j'})=0 \text{ for }j\neq j',\]
and set $h:=\max_{j=1,\ldots,N}\diam(T_j)$. Define  the $N$-dimensional space of piecewise constants
\begin{equation}\label{eq:IVnDef}
\IV_N:=\{f\in\IL_2(\Gamma):f|_{T_j}=c_j \text{ for some }c_j\in \C, \, j=1,\ldots,N\}\subset\IL_2(\Gamma)\end{equation}
and set
\begin{align}
\label{eq:VnDef}
V_N := \tr^*(\IV_N)\subset H^{-1}_\Gamma.
\end{align}
\rev{
Under appropriate assumptions, the spaces $V_N$ are dense in $H^{-1}_\Gamma$ as $N\to\infty$. 
\begin{lem}
\label{lem:density}
Suppose that Assumption \ref{ass:TildeCirc} holds, and that $h\to 0$ as $N\to\infty$. Then 
\begin{align}
\label{eq:VNConv}
\inf_{\psi_N\in V_N}\|\psi-\psi_N\|_{H^{-1}(\R^n)}\to 0 \quad \mbox{as} \quad N\to \infty, \quad \mbox{for all} \quad \psi\in H_\Gamma^{-1}. 
\end{align}
\end{lem}
\begin{proof}
Suppose that $h\to 0$ as $N\to\infty$. It is easy to see (see the proof of \cite[Thm 5.1]{HausdorffBEM}) that $\inf_{f_N\in \IV_N}\|f-f_N\|_{\IL_2(\Gamma)}\to 0$ as $N\to \infty$ for every $f\in \IL_2(\Gamma)$,
and then \eqref{eq:VNConv} follows by the density of $\tr^*(\IL_2(\Gamma))$ in $H_\Gamma^{-1}$, which holds under Assumption \ref{ass:TildeCirc} by Corollary \ref{cor:td}. 
\end{proof}
}

Our method for solving the IE \rf{eq:IE}
uses $V_N$ as the approximation space in a Galerkin method, based on %
\eqref{eqn:VariationalCts}, with $a$ defined by \rf{eqn:Sesqui}.
Given $g\in (\tH^{1}(\Gamma^c))^\perp$ we seek $\phi_N\in V_N$ such that
\begin{align}
\label{eqn:Variational}
a(\phi_N,\psi_N)
=\langle g,\psi_N\rangle_{H^{1}(\R^n)\times H^{-1}(\R^n)}, \qquad \forall \psi_N\in V_N.
\end{align}
Let $\{f^i\}_{i=1}^N$ be a basis for $\IV_N$, and let $\{e^i=\tr^*f^i\}_{i=1}^N$ be the corresponding basis for $V_N$. Then, writing $\phi_N=\sum_{j=1}^N c_j e^j$, \eqref{eqn:Variational} implies that
$\vec{c}=(c_1,\ldots,c_N)^T\in\C^N$ satisfies the system
\begin{equation} \label{eq:cvec}
\underline{\underline{A}} \vec{c} = \vec{b},
\end{equation}
where, by \eqref{eqn:SesquiA}, \eqref{eq:L2dualequiv}, and \eqref{eq:ISasHauss}, the matrix $\underline{\underline{A}} \in \C^{N\times N}$ has $(i,j)$-entry given by %
\begin{align}
A_{ij}&=a(e^j,e^i)
 = \langle \IA f^j,f^i \rangle_{\IH^{t_d}(\Gamma)\times \IH^{-t_d}(\Gamma)}
\nonumber\\
&=\int_\Gamma \int_\Gamma \Phi(x,y) f^j(y) \overline{f^i(x)}\, \rd\cH^d(y)\rd\cH^d(x),\label{eq:Galerkin}
\end{align}
and, by \eqref{eq:L2dualrep}, the vector $\vec{b}\in\C^N$ has $i$th entry given by
\begin{align}
b_{i}=\langle g,e^i \rangle_{H^{1}(\R^n)\times H^{-1}(\R^n)}
= \int_\Gamma \tr g(x) \overline{f^i(x)}\, \rd\cH^d(x),
\label{eq:GalerkinRHS}
\end{align}
with $\tr g(x)=-u^i(x)$, $x\in \Gamma$, for the scattering problem with $g$ given by \eqref{eqn:gDefScatteringProblem}. 

\rev{
\begin{rem}[Connection to known cases: III]
Building on Remark~\ref{rem:2nd}, 
if $\Gamma$ is the boundary of a bounded Lipschitz open set, or a multi-screen, then the above Galerkin method is simply a classical piecewise-constant boundary element method for the single-layer equation $S\phi=g$. 
If $\Gamma$ is a planar screen in the sense of Lemma \ref{ConnectionPlanarScreens} then the method is identical to that proposed in \cite[\S5]{HausdorffBEM} and the linear system \eqref{eq:cvec} is identical to \cite[Eqn.~(55)]{HausdorffBEM}.
\end{rem}
}
Once we have computed $\phi_N$ by solving \eqref{eq:cvec} we will compute approximations to 
$u(x)$ and $u^\infty(x)$, given by \eqref{eq:Rep}/\eqref{eq:NPdef} and \eqref{eq:ffpattern}, respectively.  
Each expression takes 
the form
$J(\phi)$, where
\begin{equation} \label{eq:Jdef}
J(\psi) := \langle \varphi, \overline{\psi}\rangle_{H^{1}(\R^n)\times H^{-1}(\R^n)}, \quad \psi\in H^{-1}_\Gamma,
\end{equation}
for %
some $\varphi\in (\tH^{1}(\Omega))^\perp$. Explicitly, 
\begin{equation} \label{eq:varphi}
\varphi = P\left(\sigma v\right),
\end{equation}
where $\sigma$ is any element of $C^\infty_{0,\Gamma}$ (with $x$ not in the support of $\sigma$ in the case that $J(\phi)=u(x)$) and $v = \Phi(x,\cdot)$ in the case that $J(\phi)=u(x)$, $v=\Phi^\infty(\hat x,\cdot)$ in the case that $J(\phi)=u^\infty(\hat x)$; note that each $v$ is $C^\infty$ in a neighbourhood of $\Gamma$. In each case we approximate $J(\phi)$ by $J(\phi_N)$ which, recalling \eqref{eq:L2dualrep}, is given explicitly by (cf.~\cite[Eqn.~(62)]{HausdorffBEM})
\begin{align} %
J(\phi_N) = \langle \varphi, \overline{\phi_N}\rangle_{H^{1}(\R^n)\times H^{-1}(\R^n)} = \vec{c}^T\vec{\varphi},
\label{eq:JphiN}
\end{align}
where $\vec{\varphi}$ has $j$th entry given by
\begin{align}
\label{eq:varphivecdefn}
\vec{\varphi}_j = \int_\Gamma \tr \varphi(x) \,f^j(x)\, \rd\cH^d(x),
\end{align}
and $\tr \varphi(x) = v(x)$, $x\in \Gamma$, for $\varphi$ given by \eqref{eq:varphi}.
The following is a basic convergence result.

\begin{thm}
\label{thm:Convergence}
Let $\Gamma$ be a compact $d$-set for some $n-2<d\leq n$,  
\rev{and suppose that Assumptions \ref{ass:Uniqueness} and \ref{ass:TildeCirc} hold.}  
Suppose also that $h\to 0$ as $N\to \infty$.
Then for sufficiently large $N\in\N$ the variational problem \rf{eqn:Variational}
has a unique solution $\phi_N\in V_N$ that is quasi-optimal in the sense that, for some constant $C>0$ independent of $\phi$ and $N$,
\begin{align}
\label{eq:Quasiopt}
\|\phi - \phi_N\|_{H^{-1}(\R^n)} \leq C
\inf_{\psi_N\in V_N}\|\phi - \psi_N\|_{H^{-1}(\R^n)},
\end{align}
where $\phi\in H^{-1}_\Gamma$ denotes the solution of \rf{eq:IE}. Furthermore, $\|\phi - \phi_N\|_{H^{-1}(\R^n)} \to 0$ as $N\to \infty$, and, where $J(\cdot)$ is given by \eqref{eq:Jdef} for some $\varphi \in  (\tH^{1}(\Omega))^\perp$, $J(\phi_N)\to J(\phi)$ as $N\to\infty$.
\end{thm}
\begin{proof}
The sesquilinear form $a(\cdot,\cdot)$ is compactly perturbed coercive (Lemma \ref{lem:coer}), and invertible if Assumption \ref{ass:Uniqueness} holds (Theorem \ref{thm:equivalence}), so the quasi-optimality \eqref{eq:Quasiopt} holds for all sufficiently large $N$ by \eqref{eq:VNConv} and standard Galerkin method theory \cite[\S4.2.3]{sauter-schwab11}. The remaining results follow by %
\rev{Lemma \ref{lem:density}}  
and the continuity of the linear functional $J(\cdot)$.
\end{proof}

\subsection{Galerkin error estimates}
\label{sec:GalerkinBounds}

\rev{If the exact solution $\phi$ possesses sufficient regularity and the spaces $V_N$ have appropriate approximability properties, then Theorem \ref{thm:Convergence} can be used to derive Galerkin error estimates, and superconvergence estimates for functionals. We record this fact in the following theorem.} 

\rev{
\begin{thm}
\label{thm:IEMConvergence}
Let the assumptions of Theorem \ref{thm:Convergence} hold. Suppose additionally that $\phi\in H^s_\Gamma$ for some $-1<s<-(n-d)/2$, and that 
\begin{align}
\label{eq:Hsboundh}
\inf_{\psi_h\in V_N}\|\psi-\psi_h\|_{H^{-1}_\Gamma}\leq c  h^{1+s}\|\psi\|_{H^{s}_\Gamma}, \qquad 0<h\leq \diam(\Gamma), \quad \psi\in H^{s}_\Gamma.
\end{align}
Then, for some constant $c>0$ independent of $h$ and $\phi$,
\begin{align}
\label{eq:GalerkinBound}
\|\phi-\phi_N\|_{H^{-1}_\Gamma}\leq c  h^{1+s}\|\phi\|_{H^{s}_\Gamma},
\end{align}
for all sufficiently large $N$. 
\rev{Furthermore, let $J(\cdot)$ be given by  \eqref{eq:Jdef} for some $\varphi\in (\tH^{1}(\Omega))^\perp$, and denote by $\zeta\in H^{-1}_\Gamma$ the solution $\phi$ of \rf{eqn:VariationalCts}, in the case that $g$ is replaced by $\varphi$. Suppose that $\zeta \in H^s_\Gamma$. 
Then }
\begin{align}
\label{eq:GalerkinBoundJ}
|J(\phi)-J(\phi_N)|\leq c h^{2(1+s)}\|\phi\|_{H^{s}_\Gamma}\|\zeta\|_{H^{s}_\Gamma},%
\end{align}
for some constant $c>0$ independent of $h$, $\phi$, and $\zeta$, for all sufficiently large $N$.
\end{thm}
\begin{proof}
\rev{The bound \eqref{eq:GalerkinBound} follows from \eqref{eq:Hsboundh} and 
\eqref{eq:Quasiopt}.  
The bound \eqref{eq:GalerkinBoundJ} follows from  \eqref{eq:GalerkinBound} by a standard Aubin-Nitsche trick argument, 
as in the proof of 
\cite[Thm 5.6]{HausdorffBEM}.}
\end{proof}
}

\rev{In the case where $\Gamma$ is the attractor of an OSC-IFS there is a natural way to build quasi-uniform meshes on $\Gamma$. Furthermore, under certain assumptions, we prove in Theorem \ref{thm:ApproxDisjoint} that the conditions of Theorem \ref{thm:IEMConvergence} are satisfied, so that the error bounds \eqref{eq:GalerkinBound} and \eqref{eq:GalerkinBoundJ} hold. 
}

\begin{figure}
\centering
\subfl{$\ell=1$ decomposition}{\includegraphics[width=.48\linewidth]{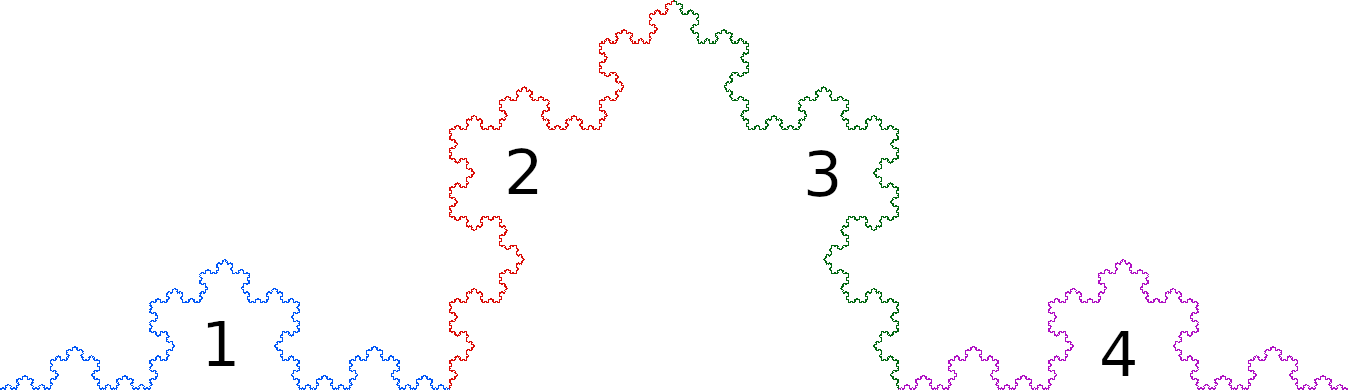}
}
\hfill%
\subfl{$\ell=2$ decomposition}{\includegraphics[width=.51\linewidth]{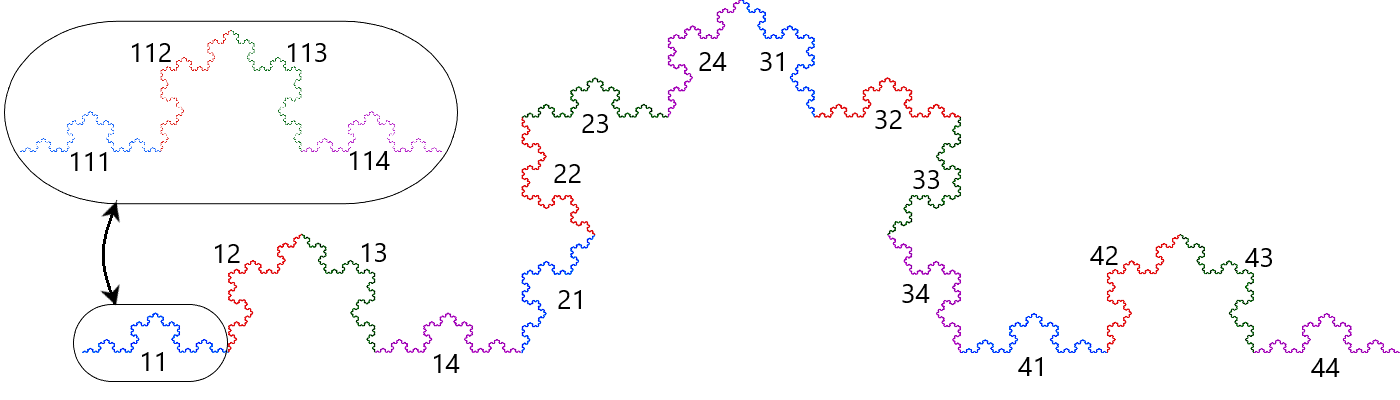}%
}
\caption{Level 1 (a) and level 2 (b) decompositions of the Koch curve. To make the labelling more compact, in (a) the labels ``$1$'',\ldots,``$4$'' indicate 
the subsets 
$\Gamma_1,\ldots,\Gamma_4$, and in (b) 
the labels 
``$ij$'' and ``$ijk$'' indicate 
$\Gamma_{(i,j)}$
and
$\Gamma_{(i,j,k)}$.
In (b) the insert shows the level 3 decomposition of
$\Gamma_{(1,1)}$.
}
\label{fig:KochCurveDecomposition}
\end{figure}

Let $\Gamma \subset \R^n$ be 
the attractor of an \rev{OSC-}IFS $\{s_1,\ldots,s_M\}$. %
Following \cite{Jonsson98},
for $\ell\in \N$ we define the set of multi-indices $I_\ell:=\{1,\ldots,M\}^\ell\!=\{\bm = (m_1,m_2,\ldots,m_\ell)$: $\,1\leq m_j\leq M , \, j=1,2,\ldots,\ell\}$, and for $E\subset\R^n$ and $\bm\in I_\ell$ we define
$E_{\bm}=s_{m_1}\circ s_{m_2}\circ \ldots \circ s_{m_\ell}(E)$.
We also set $I_0:=\{0\}$ and adopt the convention that $E_0:=E$.
This notation extends that of \rf{eq:GammamDef} where the sets $\Gamma_1,\ldots,\Gamma_M$ were introduced, corresponding to $E=\Gamma$ and $\ell=1$ here.
We illustrate this 
for the Koch curve (Example \ref{ex:Koch}, Figure \ref{fig:2Dexs}(e)) in Figure \ref{fig:KochCurveDecomposition}. Illustrations for other examples are given in \cite[Figs~1-6]{NonDisjointQuad}.

Let $0<h\leq \diam(\Gamma)$. Define the %
index set $L_h$ by $L_h:=\{0\}$ for $h=\diam(\Gamma)$, and by
\begin{align}
\label{eq:LhDef}
L_{h} := \bigg\{\bm \in \bigcup_{\ell=1}^\infty I_\ell : \diam(\Gamma_{\bm})\leq h \text{ and } \diam(\Gamma_{\bm_-})>h\bigg\},
\end{align}
for $h<\diam(\Gamma)$, where, for $\bm = (m_1,\ldots,m_\ell)$, $\bm_-:= (m_1,...,m_{\ell-1})$ if $\ell\in \N$ with $\ell\geq 2$, and $\bm_-:= 0$ if $\ell=1$. 
\rev{Then $\{T_j\}_{j=1}^N:=\{\Gamma_{\bm}\}_{\bm\in L_h}$ defines a quasi-uniform mesh of $\Gamma$. }
We define 
\rev{the spaces of piecewise-constant functions
\begin{align} %
\label{eq:YHDef}
\IY_h := {\rm span}\big(\{\chi_{\bm}\}_{\bm\in L_{h}}\big) 
\qquad \text{and} \qquad 
Y_h := \tr^*(\IY_h) \subset H^{-1}_\Gamma,
\end{align}
}
where $\{\chi_{\bm}\}_{\bm\in L_{h}}$ is the canonical $\IL_2(\Gamma)$-orthonormal basis for $\IY_h$ given by
\begin{align}
\label{eq:VkBasisDefn}
\chi_{\bm}(x):=
\begin{cases}
\frac{1}{\cH^d(\Gamma_{\bm})^{1/2}}, & x\in \Gamma_{\bm},\\
0,& \text{otherwise},
\end{cases}
\end{align}
The following theorem is then a consequence of results in \cite{HausdorffBEM} and \cite{CaChGiHe23}. 

\rev{
\begin{thm}\label{thm:ApproxDisjoint}
Let $\Gamma$ be an OSC-IFS attractor with dimension $d:=\dimH(\Gamma)$, such that either (a) $\Gamma$ is disjoint, (b) $d=n$, or (c) $d=n-1$ and $\Gamma\subset \Gamma_\infty:=\R^{n-1}\times \{0\}$. Let Assumption \ref{ass:Uniqueness} hold, and set $V_N=Y_h$. %
Then \eqref{eq:Hsboundh} holds for all $-1<s<-(n-d)/2$. 

Furthermore, for the scattering problem with $g$ defined by \eqref{eqn:gDefScatteringProblem} there exists $-1<s<-(n-d)/2$ 
\rev{such that 
\eqref{eq:GalerkinBound} holds and \eqref{eq:GalerkinBoundJ} 
holds 
}
for both $J(\phi)=u(x)$ and $J(\phi)=u^\infty(\hat{x})$. 
\end{thm}
}

\begin{proof}
For case (a), where $\Gamma$ is a disjoint IFS attractor, it was proved in \cite[Prop.~5.2]{HausdorffBEM} that, for every $0<t<1$ and every $0<t'<t$, there exists a constant $c>0$ such that 
\begin{align}
\label{eq:Htbound}
\inf_{\Psi_h\in \IY_h}\| f-\Psi_h\|_{\IH^{-t}(\Gamma)}\leq c\,h^{t-t'}\| f\|_{\IH^{-t'}(\Gamma)}, \quad 0<h\leq \diam(\Gamma), \quad f\in \IH^{-t'}(\Gamma).
\end{align}
In \cite[Prop.~5.2]{HausdorffBEM} this result was actually only stated for $n-1<d<n$, but the argument of \cite[Prop.~5.2]{HausdorffBEM} holds in fact for all $0<d<n$ (so, in particular for $n-2<d<n$), because the results from \cite{Jonsson98} on which it is based hold for all $0<d<n$. 
This latter statement requires some explanation. A key step in the argument of \cite[Prop.~5.2]{HausdorffBEM} was the use of results from \cite{Jonsson98} to prove that $\|\cdot\|_{\IH^t(\Gamma)}$ is equivalent to a norm defined in terms of coefficient decay in a wavelet expansion; see \cite[Thm.~3.1 and Cor.~3.3(iii)]{HausdorffBEM}. The relevant results in \cite{Jonsson98} (Theorems 1 and 2) are stated under the additional assumption that $\Gamma$ is not contained in an $(n-1)$-dimensional hyperplane.
However, this additional assumption is made in \cite{Jonsson98} solely to ensure that Markov's inequality \cite[Eqn (4.1)]{Jonsson98} is satisfied for whatever class of polynomials is being used in the wavelet expansion. As a result, this additional assumption is superfluous for us because we consider only piecewise-constant approximations and $0<t<1$, while \cite{Jonsson98} considers also higher order polynomials and larger $t$, and for constant functions Markov's inequality \cite[Eqn (4.1)]{Jonsson98} is trivially satisfied.

For case (b), the result \eqref{eq:Htbound} was proved for $t=1$ in \cite{CaChGiHe23}, using a quite different argument based on Poincar\'e inequalities. The fact that it also holds in case (c), again for $t=1$, follows from the result for case (b), applied in the setting of $\R^{n-1}$, and 
Lemma \ref{ConnectionPlanarScreens}. 

The above 
establishes \eqref{eq:Htbound} for the particular case $t=t_d$. The bound \eqref{eq:Hsboundh} then follows by \rev{Theorem \ref{thm:density} and} Corollary \ref{cor:td}, noting that in case (b) Assumption \ref{ass:TildeCirc} holds by Remark \ref{rem:TildeCirc}(iii). 

The final statement then follows from Theorem \ref{thm:IEMConvergence}, since, for the scattering problem, and the choices of $J$ under consideration, the solutions $\phi$ and $\zeta$ possess some extra regularity by Proposition \ref{prop:epsilon} (see the argument in Remark \ref{rem:regularity} for $\phi$, and argue similarly for $\zeta$).
\end{proof}

\begin{rem}[\bf Convergence rates]
\label{rem:ConvRates}
Suppose that, in addition to the assumptions of Theorem \ref{thm:ApproxDisjoint}, 
\rev{Hypothesis \ref{conj} holds}. \rev{Then,}   arguing as in Remark \ref{rem:regularity}, assuming the datum $g$ is sufficiently smooth, we will have the maximum possible regularity for $\phi$, i.e., 
$\phi\in H^s_\Gamma$ for every $-1<s<-(n-d')/2$. %
Then, assuming that the bounds in Theorem \ref{thm:IEMConvergence} are sharp, in numerical experiments we should expect to see errors in the computation of $\phi$ and of linear functionals of $\phi$ roughly proportional to $h^{1+(d'-n)/2}$ and $h^{2+d'-n}$, respectively. 
\rev{(For the latter, assume also that $\varphi$ in Theorem \ref{thm:Convergence} is sufficiently smooth so that $\zeta \in H^s_\Gamma$ for every $-1<s<-(n-d')/2$.)} 

\rev{If, additionally,} $\Gamma$ is homogeneous, with $\rho_m=\rho $ for $m=1,\ldots,M$, for some $0<\rho<1$, in which case $d=\dimH(\Gamma)=\log(M)/\log(1/\rho)$, 
\rev{then the meshes defined by \eqref{eq:LhDef} are uniform, and taking \mbox{$h=\rho ^\ell\diam{\Gamma}$}, for some $\ell\in \N$, gives $L_h = I_\ell=\{1,\ldots,M\}^\ell$ and $V_N=Y_h=\tr^*({\rm span}\left(\{\chi_{\bm}\}_{\bm\in I_\ell}\right))$.}    
In this case, since $h$ is proportional to $\rho^\ell$ and $\rho=M^{-1/d}$, we should see errors \rev{in $\phi$ and in linear functionals of $\phi$}   proportional to $(M^{d'/d})^{-\ell/2}$ and $(M^{d'/d})^{-\ell}$, respectively, in the case $n=2$, and proportional to $(M^{d'/d}\rho )^{-\ell/2}$ and $(M^{d'/d}\rho )^{-\ell}$, respectively, in the case $n=3$. 
\end{rem}

\subsection{Numerical quadrature}
\label{sec:Quadrature}
To implement our method we need suitable numerical quadrature rules to evaluate the integrals \eqref{eq:Galerkin}, \eqref{eq:GalerkinRHS} and \eqref{eq:varphivecdefn}. For this we generalise the approach taken for the screen case in \cite{HausdorffBEM}. Here we give only the main ideas, and refer the reader to Appendix \ref{sec:QuadratureAppendix}, \cite[\S5.4]{HausdorffBEM}, and \cite{HausdorffQuadrature,NonDisjointQuad} for details.

Suppose that $\Gamma$ is \rev{an OSC-IFS} attractor, %
and that, as in \S\ref{sec:GalerkinBounds}, we are using the approximation space $V_N=Y_h$ given by \eqref{eq:YHDef}.
Suppose that $g$ is given by \eqref{eqn:gDefScatteringProblem} and $\varphi$ by \eqref{eq:varphi}, with $u^i$ and $v$ both $C^\infty$ in a neighbourhood of $\Hull(\Gamma)$, the convex hull of $\Gamma$.
Suppose that we adopt the canonical $\IL_2(\Gamma)$-normalised basis \eqref{eq:VkBasisDefn}, %
so that $f^{j}= %
\chi_{\bm(j)}$, $j=1,\ldots,N$, 
where $N:= |L_h|$, with $L_h$ given by \eqref{eq:LhDef}, and $(\bm(1),\ldots,\bm(N))$ is some ordering of the elements of $L_h$.
Then, where $\mu_{\bm}:=\cH^d(\Gamma_\bm)$ for $\bm\in L_h$,  the integrals to be evaluated are, for $i,j\in\{1,\ldots,N\}$,
\begin{align}
\label{eqn:GalerkinElements}
A_{ij}&=\mu_{\bm(i)}^{-1/2}\mu_{\bm(j)}^{-1/2}\int_{\Gamma_{\bm(i)}}\int_{\Gamma_{\bm(j)}} \Phi(x,y) \, \rd\cH^d(y)\rd\cH^d(x),
\end{align}
\begin{align}
\label{eqn:GalerkinElementsRHS}
b_{i}&=-\mu_{\bm(i)}^{-1/2}\int_{\Gamma_{\bm(i)}}
u^i(x)
\, \rd\cH^d(x),
\qquad\qquad
\vec{\varphi}_i= \mu_{\bm(i)}^{-1/2}\int_{\Gamma_{\bm(i)}}
v(x)
\, \rd\cH^d(x).
\end{align}

Since
$u^i$ and $v$ are
smooth in a neighbourhood of $\Gamma$, \eqref{eqn:GalerkinElementsRHS}
can be evaluated using the composite barycentre rule of \cite[Defn~3.1]{HausdorffQuadrature}, cf.~\cite[(97)-(99)]{HausdorffBEM}. This involves decomposing the mesh element $\Gamma_{\bm(i)}$
into smaller self-similar sub-elements whose vector indices are taken from the index set $L_{h_Q}$, for some maximum quadrature element diameter $h_Q\leq h$,
and applying a one-point quadrature rule on each sub-element.
Similarly, provided that $\Gamma_{\bm(i)}$ and $\Gamma_{\bm(j)}$ are disjoint, \eqref{eqn:GalerkinElements} can be evaluated using a tensor product version of this composite barycentre rule (defined in \cite[Defn~3.5]{HausdorffQuadrature}), cf.~\cite[(92)]{HausdorffBEM}.

When $\Gamma_{\bm(i)}$ and $\Gamma_{\bm(j)}$ are not disjoint, the integral in \eqref{eqn:GalerkinElements} is
singular.
Singularity subtraction reduces the problem to the evaluation of
\begin{align}
\label{eqn:GalerkinElementsPhi0}
\int_{\Gamma_{\bm(i)}}\int_{\Gamma_{\bm(j)}} \Phi_{\rm sing}(x,y) \, \rd\cH^d(y)\rd\cH^d(x),
\end{align}
where $\Phi_{\rm sing}(x,y) = -\log(|x-y|)/(2\pi)$ if $n=2$, and $\Phi_{\rm sing}(x,y)=1/(4\pi|x-y|)$ if $n=3$. The integral of $\Phi-\Phi_{\rm sing}$ is regular and can be evaluated using the
tensor product
composite barycentre rule, cf.~\cite[(94)]{HausdorffBEM}.
The treatment of \eqref{eqn:GalerkinElementsPhi0} depends on the nature of $\Gamma$.

If $\Gamma$ is disjoint (e.g.\ the Cantor set, Figure \ref{fig:2Dexs}(e)) then \eqref{eqn:GalerkinElementsPhi0} is singular if and only if $i=j$, 
in which case \eqref{eqn:GalerkinElementsPhi0} can be evaluated using the quadrature rules of \cite[\S4.3]{HausdorffQuadrature}, cf.~\cite[(95)-(96)]{HausdorffBEM}.
These rules exploit the self-similarity of $\Gamma$
and the homogeneity of $\Phi_{\rm sing}$ %
to write the singular integral \eqref{eqn:GalerkinElementsPhi0} in terms of regular integrals, which can be evaluated by the
composite barycentre rule.

If $\Gamma$ is non-disjoint (e.g.\ the Koch curve, or the Koch snowflake, Fig.~\ref{fig:2Dexs}(f), (g)) then the situation is more complicated, because, in addition to the self-interaction case $i=j$, \eqref{eqn:GalerkinElementsPhi0} can also be singular for $i\neq j$, if $\Gamma_{\bm(i)}$ and $\Gamma_{\bm(j)}$ intersect at a point or at a higher-dimensional set. For certain non-disjoint attractors, it holds that: (i) all singular instances of \eqref{eqn:GalerkinElementsPhi0} that arise in our discretization can be written in terms of one of a finite collection of ``fundamental'' singular integrals, which capture the different singular interactions that can occur between mesh elements; and (ii)
these fundamental singular integrals together satisfy a small linear system of equations that can be solved in closed form in terms of regular integrals, which can be evaluated using the composite barycentre rule.
A general algorithm for identifying the fundamental singular integrals and deriving the associated linear system was presented in \cite[Algorithm 1]{NonDisjointQuad}, along with explicit formulas for the Sierpinski triangle, the Vicsek fractal, the Sierpinski carpet, and the Koch snowflake. These formulas were applied in the context of screen scattering problems in \cite[\S7.3]{NonDisjointQuad}. In Appendix \ref{sec:QuadratureAppendix} %
we briefly explain the methodology of \cite{NonDisjointQuad}, and derive explicit formulas for the case of the Koch curve, which was not considered in \cite{NonDisjointQuad}.

The accuracy of the quadrature approximations described above for the evaluation of \eqref{eqn:GalerkinElements} and  \eqref{eqn:GalerkinElementsRHS}
can be controlled by a single parameter $h_Q\leq h$, which represents the maximum diameter of the sub-elements used in the composite barycentre rule.
Using the results of \cite{HausdorffQuadrature} 
one can prove quadrature error estimates. 
\rev{The following theorem is a generalisation of \cite[Thm 5.14]{HausdorffBEM}. While \cite[Thm 5.14]{HausdorffBEM} was stated for the special case where $\Gamma\subset\R^{n-1}\times \{0\}$ is a planar screen, it extends trivially to our more general context, with minor notational adjustments, 
because the planarity of $\Gamma$ was not used in its proof.} 
\rev{We recall that $\Hull(E)$ denotes the convex hull of $E\subset\R^n$, and we denote by $\|\cdot\|_2$ both the Euclidean norm on $\IC^N$ and the induced matrix norm on $\IC^{N\times N}$.}

\begin{thm}
\label{thm:Quadrature}
Let $\Gamma$ be an \rev{OSC-}IFS attractor. %
Let $\underline{\underline{A}}^Q$, $\vb^Q$ and $\vec{\varphi}^Q$ denote the approximations of
\eqref{eqn:GalerkinElements} and  \eqref{eqn:GalerkinElementsRHS}
obtained via the quadrature described above, using a maximum sub-element diameter $0<h_Q\leq h$ in the composite barycentre rule.
\begin{enumerate}[(i)]
\item
Let $u^i$ satisfy the Helmholtz equation in some open neighbourhood of $\Hull(\Gamma)$.
Then %
\begin{align}
\label{eq:bquadEst}
\|\vb-\vb^Q\|_2 \leq h_Q^2
|u^i|_{2,\Hull(\Gamma)}\cH^d(\Gamma)^{1/2},
\end{align}
where $|u^i|_{2,\Hull(\Gamma)}:=\max_{x\in\Hull(\Gamma)}\max_{\substack{\alpha\in \N_0^n\\|\alpha|=2}}|D^\alpha u^i(x)|$.
\item
Let $v$ be $C^\infty$ in a neighbourhood of $\Hull(\Gamma)$.
Given $\psi_N\in V_N=Y_h$, let $J^Q(\psi_N)$ be defined by \eqref{eq:JphiN} with $\vec{\varphi}$ replaced by $\vec{\varphi}^Q$, and let $\vec{\psi}$ denote the coefficient vector of $\psi_N$ in
the basis $\{e^i=\tr^*f^i\}_{i=1}^N$.
Then,
there exists
$C>0$, independent of $h$, $h_Q$, $v$ and $\psi_N$, such that %
\begin{align} \label{eq:JJQ}
|J(\psi_N)-J^Q(\psi_N)| &\leq h_Q^2 |v|_{2,\Hull(\Gamma)} \|\vec\psi\|_2 \cH^d(\Gamma)^{1/2}.
\end{align}
\item
\rev{Suppose that $\Gamma$ is \textit{hull-disjoint}, meaning that $\Hull(\Gamma_m)\cap \Hull(\Gamma_{m'})=\emptyset$ for every $m\neq m'\in\{1,\ldots,M\}$.} 
Then there exists
$C>0$, independent of $h$ and $h_Q$, such that
\begin{align}
\label{eq:AquadEst2}
\|\underline{\underline{A}} -\underline{\underline{A}} ^Q\|_2 \leq Ch_Q h^{-(n-1)}
\cH^d(\Gamma).
\end{align}
If, further, $\Gamma$ is homogeneous, then
\begin{align}
\label{eq:AquadEst1}
\|\underline{\underline{A}} -\underline{\underline{A}} ^Q\|_2 \leq Ch_Q^2 h^{-n}
\cH^d(\Gamma).
\end{align}
\end{enumerate}
\end{thm}

\rev{While Theorem \ref{thm:Quadrature}(iii) is stated only for hull-disjoint attractors (because that was the setting considered in \cite{HausdorffBEM}), we expect it should be possible to prove similar results for non-disjoint OSC-IFS attractors, by combining the results of \cite[\S4.3]{HausdorffQuadrature} with those of \cite[\S6]{NonDisjointQuad}. 
Indeed, numerical experiments (not reported here) suggest that $\|\underline{\underline{A}} -\underline{\underline{A}} ^Q\|_2 = O(h_Q^2)$ as $h_Q\to 0$ for all the examples considered in \S\ref{sec:Numerics}, including the Koch snowflake, which is both non-disjoint and non-homogeneous.
However, we leave the proof of this for future work.}

In principle,
the quadrature error estimates of
Theorem \ref{thm:Quadrature}
could be combined with the semi-discrete convergence estimates of Theorem \ref{thm:IEMConvergence} to obtain a fully discrete analysis for our IE method (under appropriate assumptions, such as disjointness), with conditions on how small $h_Q$ should be in order to maintain the convergence rates in Theorem \ref{thm:IEMConvergence}.
For brevity we do not embark on such an analysis here, but refer the interested reader to \cite[\S5.4]{HausdorffBEM} where the analogous analysis was carried out for screen problems.
In practice, our numerical results in \S\ref{sec:Numerics} suggest that in many cases it may be sufficient to decrease $h_Q$ in proportion to $h$ in order to achieve the predicted rates.

\section{Numerical results}
\label{sec:Numerics}

In this section we present numerical results obtained using our Galerkin IEM for scattering by various fractals $\Gamma$, each 
\rev{an OSC-IFS} attractor%
\footnote{Our method is implemented in Julia and is available to download at \href{https://github.com/AndrewGibbs/IFSIntegrals}{github.com/AndrewGibbs/IFSIntegrals}}. 
For each example we assume plane wave incidence, i.e.\ the datum $g$ is as in \eqref{eqn:gDefScatteringProblem} with $u^i(x)=e^{\ri k\vartheta\cdot x}$ and $|\vartheta|=1$, and we compute the Galerkin IEM solution by solving \eqref{eq:cvec}, using the piecewise-constant quasi-uniform-mesh approximation space $V_N=Y_h$, with $Y_h$ defined as in \rf{eq:YHDef}, so that each element is a scaled copy of $\Gamma$, and with the basis functions as defined above \eqref{eqn:GalerkinElements}. We approximate the  scattered field $u$ and/or the far field $u^\infty$,  which each (as discussed above \eqref{eq:Jdef}) take the form of the linear functional \eqref{eq:Jdef} with $\varphi$ given by \eqref{eq:varphi}, by the discretisation \eqref{eq:JphiN}. These calculations require evaluation of the integrals  \eqref{eq:Galerkin}, \eqref{eq:GalerkinRHS}, and \eqref{eq:varphivecdefn}. We approximate these by the methods detailed in \S\ref{sec:Quadrature}, using a maximum quadrature element size $h_Q=C_Qh$. \rev{We choose} $C_Q:=\max\{\rho_m^2: m=1,\ldots, M\}$, where $\rho_m$, $m=1,..,M$, are the contraction factors of the IFS, except for the higher wavenumber simulations \rev{for $k\geq 20$ in Figures 
\ref{fig:CantorKoch}, \ref{fig:KochSnowflake} and 
\ref{fig:Tetrahedron}}, where we use $C_Q:=\max\{\rho_m^4: m=1,\ldots, M\}$. 
To validate the accuracy of our quadrature, a number of our experiments were repeated using smaller values for $h_Q$, and the difference in the results was found to be negligible.%

When we plot errors we use as our ``exact'' solution a more accurate Galerkin-IEM solution. %
Most of our experiments are for homogeneous attractors, in which case our mesh is uniform with $N=M^\ell$, for some $\ell\in\N$, we denote the corresponding approximate scattered, total, and far fields by $u_\ell$, $u^t_\ell := u^i+u_\ell$, and $u^\infty_\ell$, respectively, and the ``exact'' solution is the solution for $\ell=\ellref$, for some $\ellref$ that we note for each example. Where we plot $L^\infty$ relative error estimates these are
	\begin{equation}\label{eq:relerrs}
		\frac{\|u_{\ellref}-u_\ell\|_{L^\infty}}{\|u_{\ellref}\|_{L^\infty}}\quad\text{and}\quad\frac{\|u^\infty_{\ellref}-u^\infty_\ell\|_{L^\infty}}{\|u^\infty_{\ellref}\|_{L^\infty}},
\end{equation}
where the $L^\infty$ norms are discrete norms taken over a set of points detailed for each example.

\subsection{Examples in 2D space} \label{sec:2D}

\begin{figure}%
\centering
\subfl{}{\includegraphics[height=50mm]{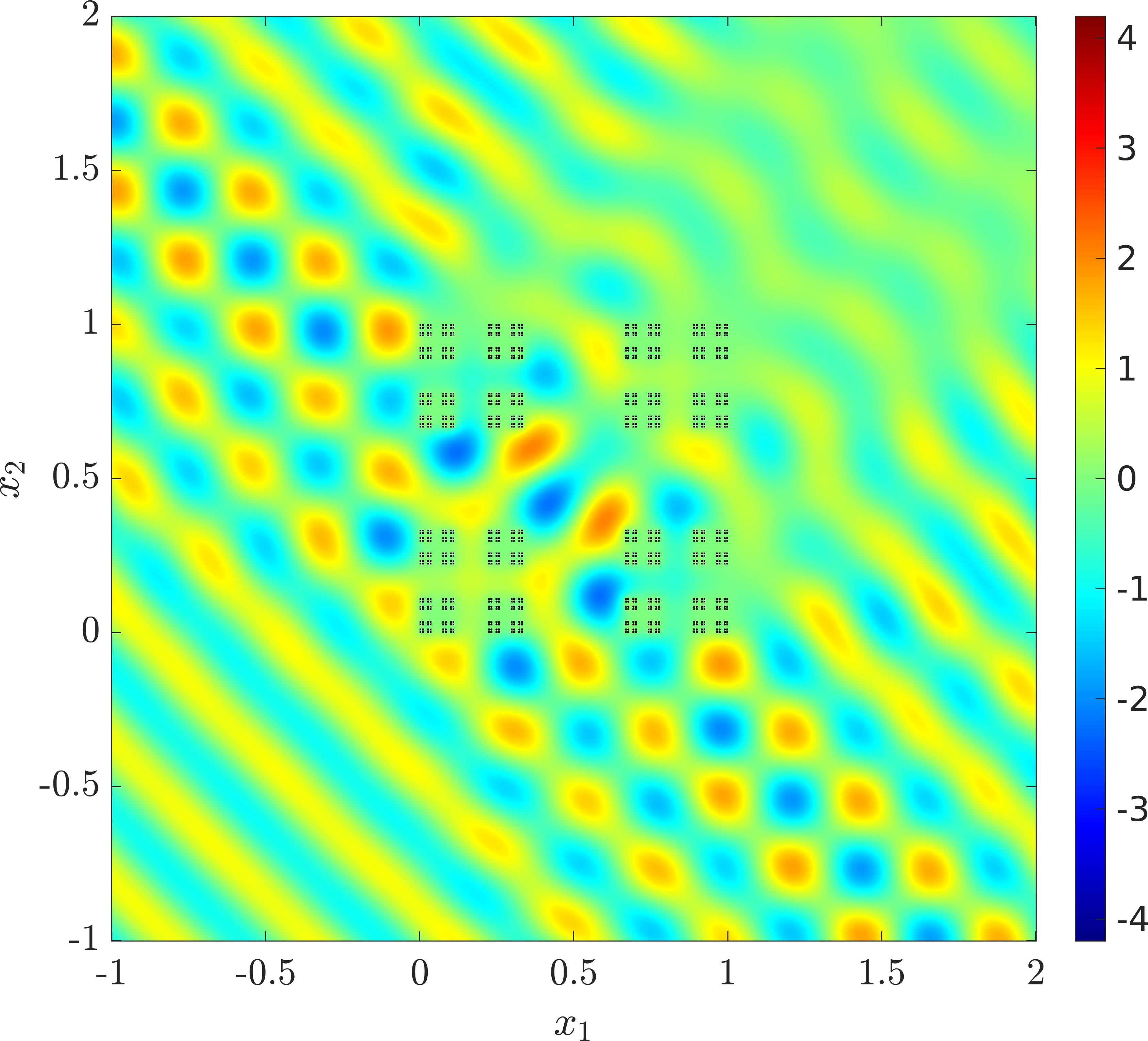}}
\hspace{2mm}
\subfl{}{\includegraphics[height=50mm]{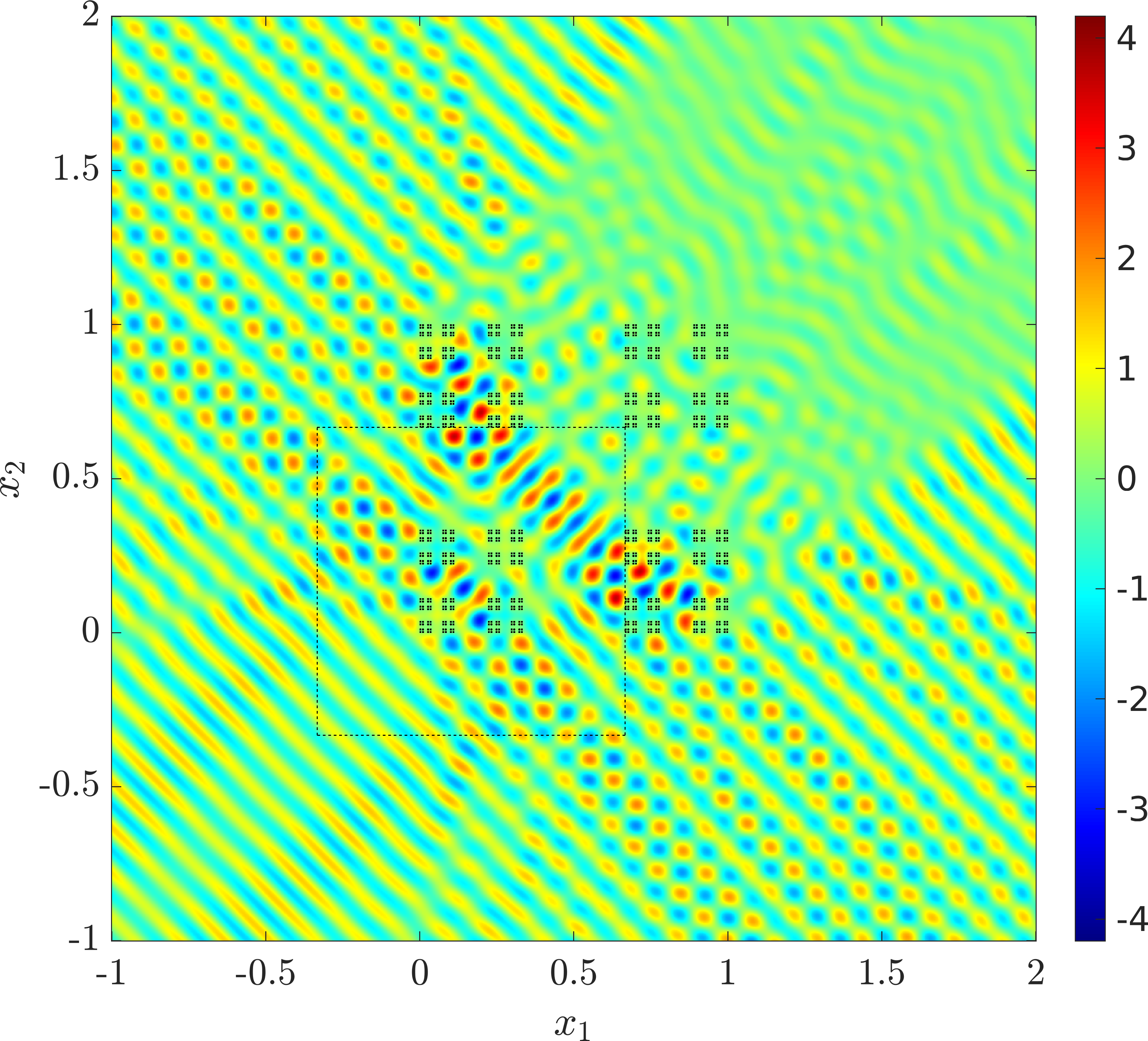}}\\
\subfl{}{\includegraphics[height=50mm]{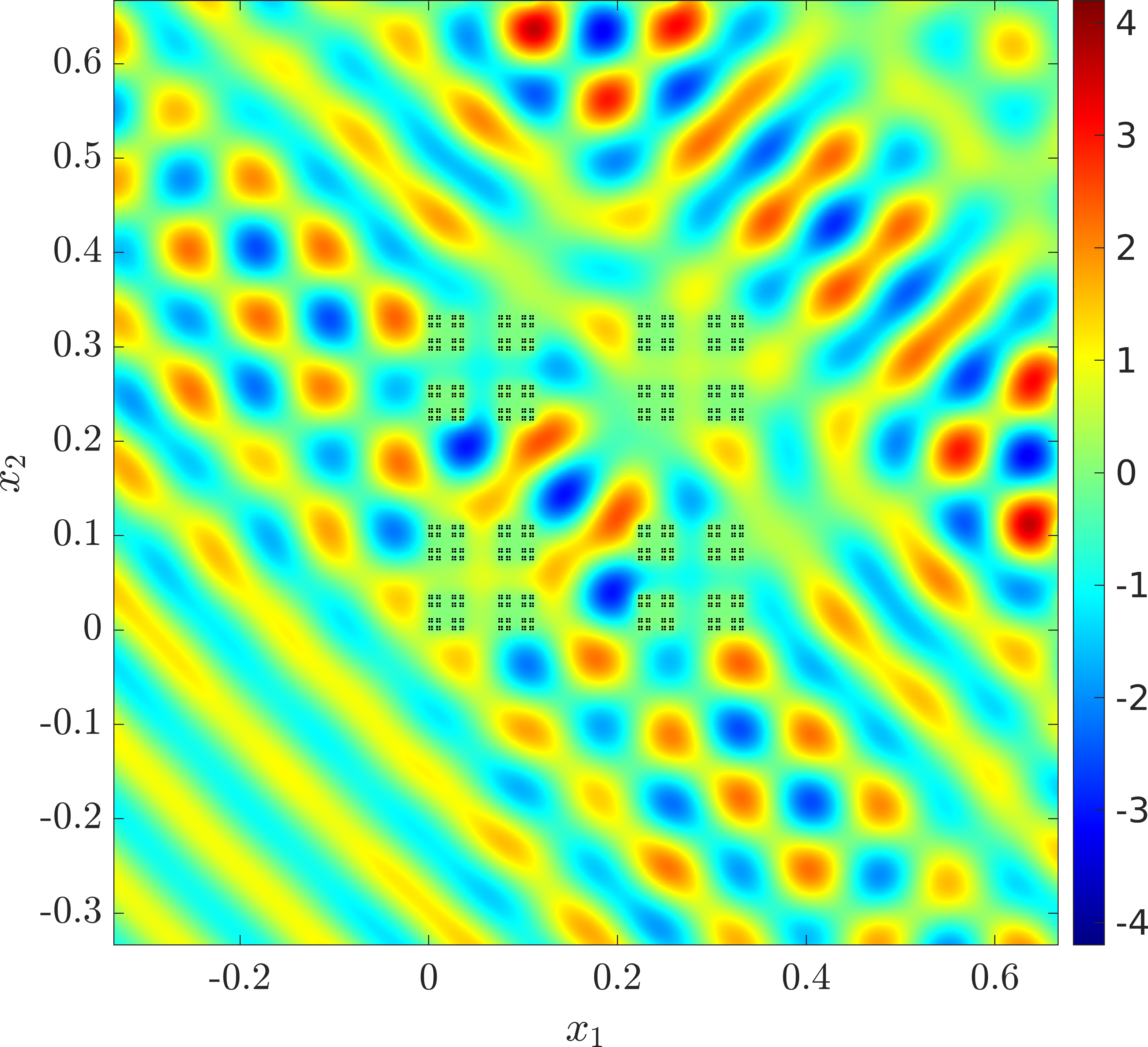}}
\hspace{2mm}
\subfl{}
{\includegraphics[height=42.5mm]%
{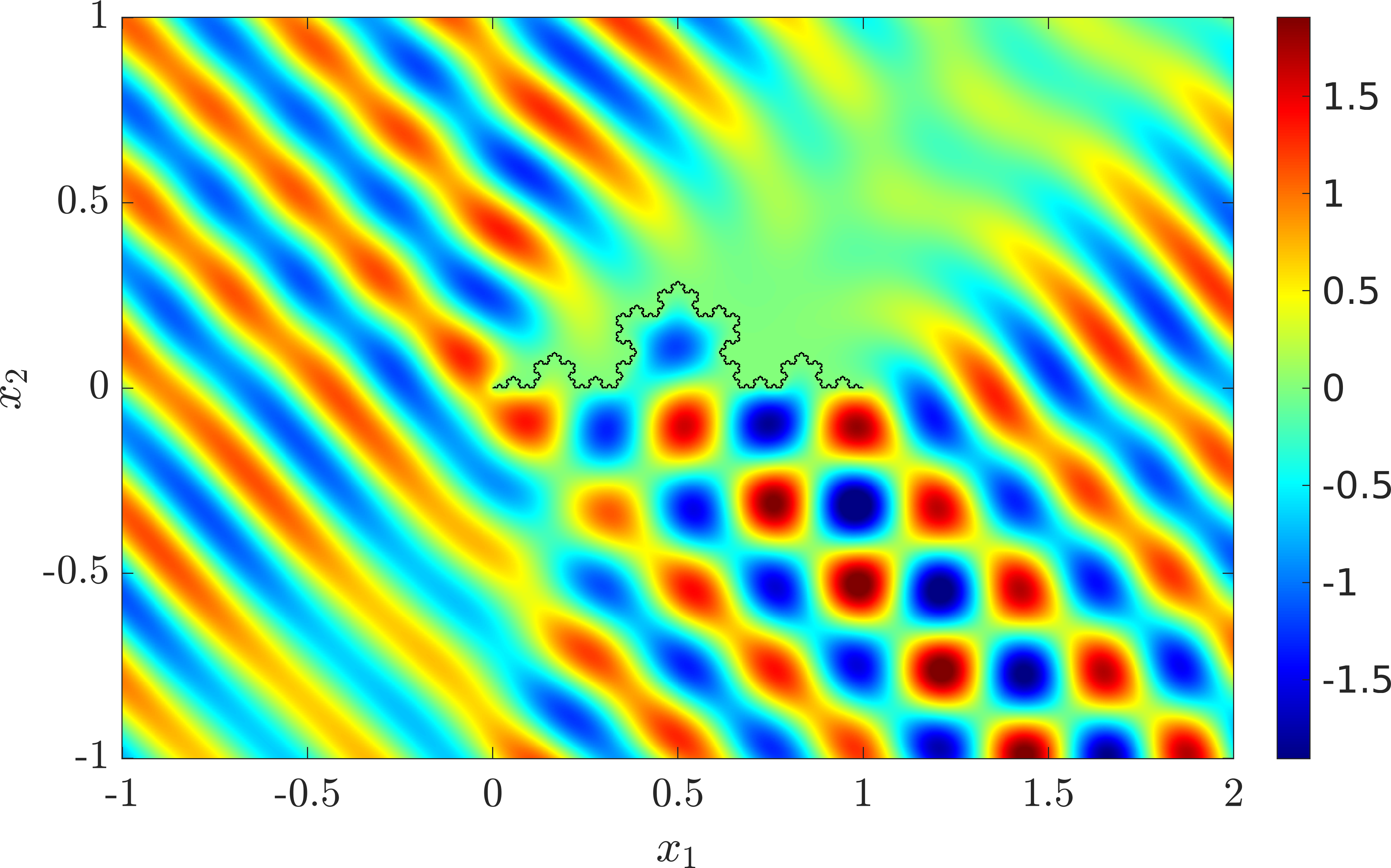}}
\caption{Scattering by a middle third Cantor dust, (a)-(c), and a Koch curve, (d). 
See \rev{\S\ref{sec:2D}(\ref{sec:DustKoch})}. %
\label{fig:CantorKoch}
}
\end{figure}

\subsubsection{Cantor dust and Koch curve}
\label{sec:DustKoch}

Plots of ${\rm Re}(u_\ell^t)$ are shown in  Figure \ref{fig:CantorKoch} for two fractal scatterers. The first, see (a)-(c), is the middle-third Cantor dust, $\Gamma = C\times C$, where $C$ is the Cantor set defined in Example \ref{ex:Cantor}, and the other, panel (d), is the Koch curve of Example \ref{ex:Koch}. For both scatterers the IFS is homogeneous, 
with $M=4$, $\rho=1/3$, and hence $d=\dimH(\Gamma)=\log4/\log3\approx 1.26$\rev{; see, e.g., \cite[Eqn.~(125)]{HausdorffBEM} for the Cantor dust IFS}.
In all plots  the incident plane wave has  direction $\vartheta = (1,1)/\sqrt{2}$, and we take $k=20$ and $\ell=4$ in (a) and (d), $k=60$ and $\ell=5$ in (b)-(c), so that $h=1/3^4\approx 0.0123$ in (d), $h= \sqrt{2}/3^4\approx 0.0175$ in (a), and $h= \sqrt{2}/3^5\approx 0.00582$ in (b)-(c). The key difference between (a) and (b) is the tripling  of $k$, so that the wavelength $\lambda=2\pi/k$ reduces from $\lambda\approx 0.314$ in (a) to $\lambda \approx 0.105$ in (b). The wave field in (a) does not appear to resolve details  beyond level 2, i.e. it appears that $u^t\approx 0$ in the convex hull of each of the sixteen $\Gamma_\bm$ with $\bm\in I_2$ (in the notation of \S\ref{sec:GalerkinBounds}). This is unsurprising as $u^t=0$ on each level 2 component, $\Gamma_\bm$, with $\bm\in I_2$, and each is comprised of four level 3 components on which $u^t=0$ and whose separation is only $1/3^3 \approx 0.118 \lambda$, i.e., is a small fraction of $\lambda$. In (b), where $\lambda$ is reduced by a factor 3, the wave field appears to resolve detail down to level 3, i.e.\ to resolve details of $1/3$ the size. To see this more clearly the region inside the dotted boundary is blown up by a factor 3 in (c). After this scaling in fact, thanks to the incidence direction we have chosen, the part of the plot (c) in $[-1/3, 1/3]^2$ is very similar to the field plotted in (a) in $[-1,1]^2$.

\begin{figure}%
\centering
\subfl{Scattered-field relative $L^\infty$ errors}{\includegraphics[width=0.37\textwidth]{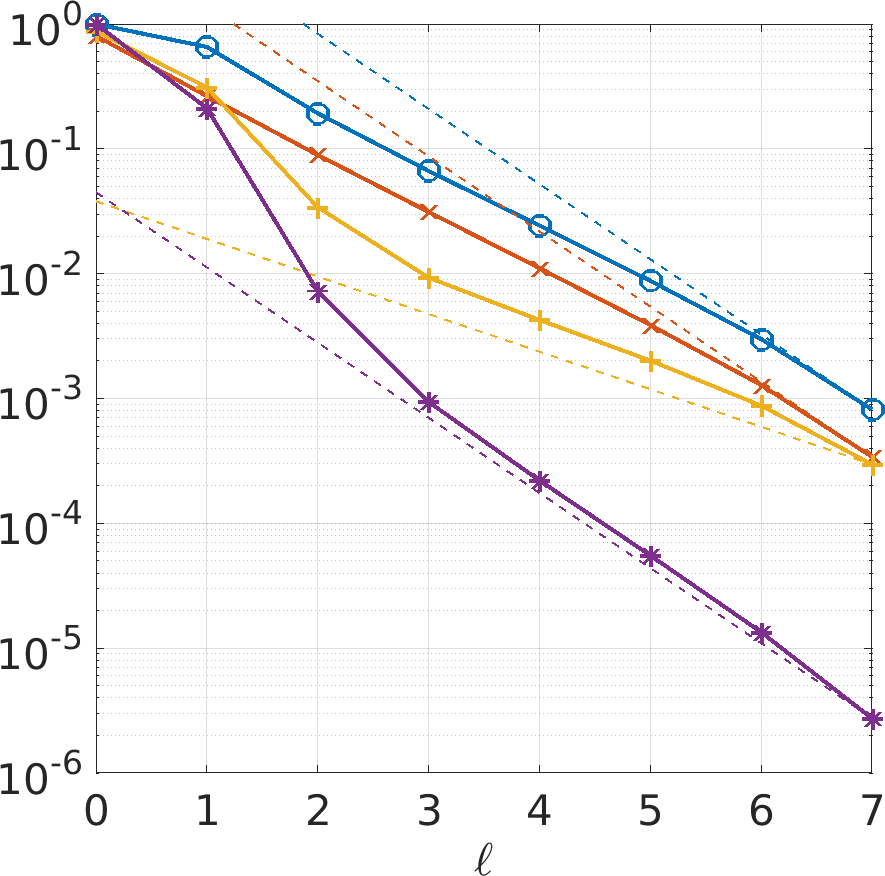}}
\hspace{2mm}
\subfl{Far-field relative $L^\infty$ errors}{\includegraphics[width=0.59\textwidth]{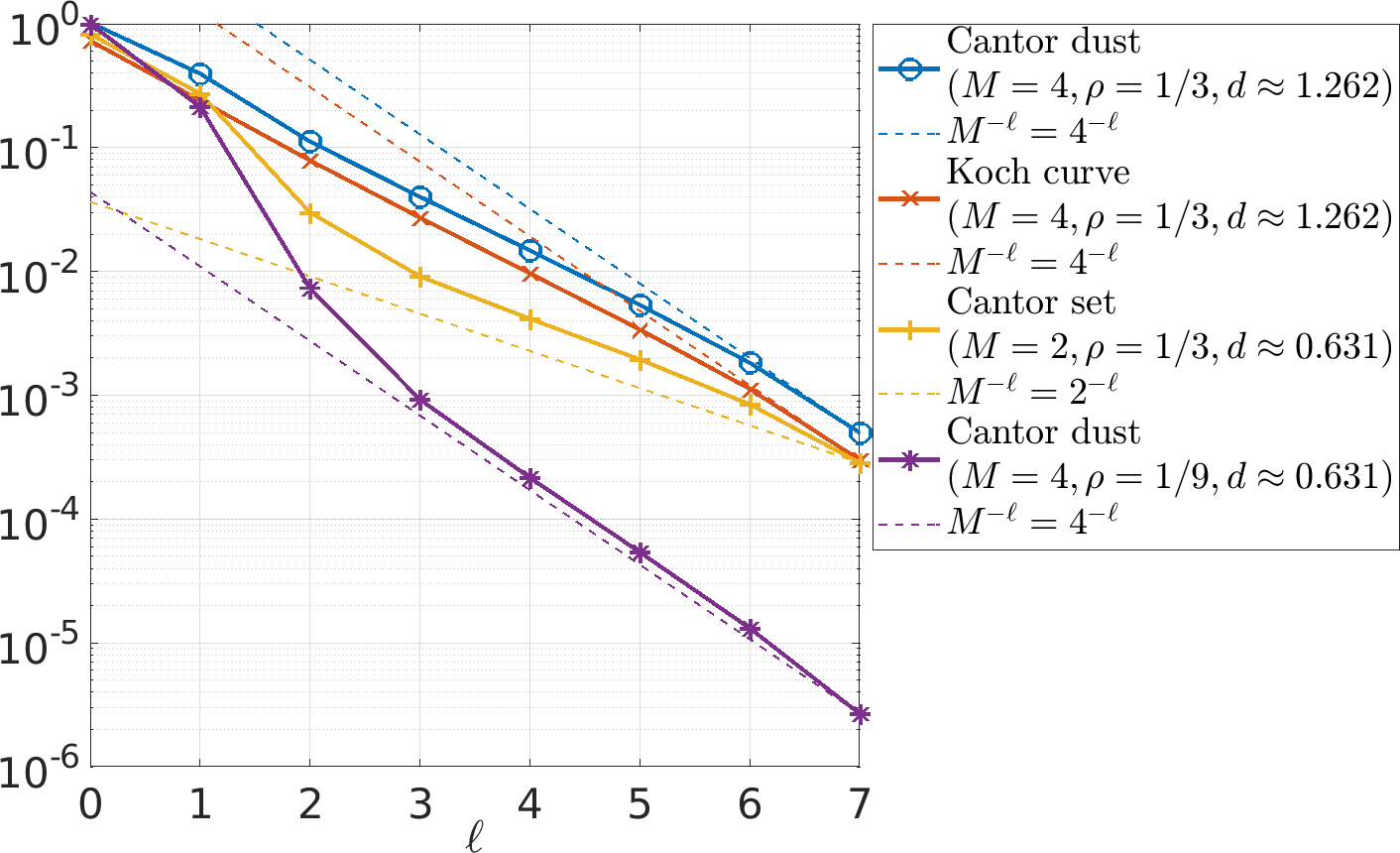}}
\caption{Plots of the discrete relative-error estimates \eqref{eq:relerrs} for a range of 2D examples\rev{; see \S\ref{sec:2D}(\ref{sec:convplot})}.}
\label{fig:2DErrorPlots}
\end{figure}

\subsubsection{Convergence plots} \label{sec:convplot}
In Figure \ref{fig:2DErrorPlots} we show the discrete $L^\infty$ relative errors \eqref{eq:relerrs} for a range of 2D examples, namely the Koch curve (Ex.~\ref{ex:Koch}), the Cantor set $C\times\{0\}$ (Ex.~\ref{ex:Cantor}), and the Cantor dust $C\times C$ with two different values of $\rho$, with  plane wave incidence direction $\vartheta=(1,-1)/\sqrt{2}$ and wavenumber $k=5$. To compute the scattered-field relative error given by \eqref{eq:relerrs}  we sample  at $50$ points equispaced along each edge of the square $(-1,2)\times(-1.5,1.5)$ (200 points in total), and for the far-field we sample at $50$ equispaced points on the circle $\IS^1$. We use, for each scatterer, $\ellref=\ell_{\max}+2$, where $\ell_{\max}$ is the largest $\ell$ for which results are shown.

Also plotted in Figure \ref{fig:2DErrorPlots} are graphs of $cM^{-\ell}$, with $c>0$ chosen  to fit each error curve. In the cases where Theorem \ref{thm:ApproxDisjoint} applies (all except the Koch curve), 
\rev{then if Hypothesis \ref{conj} holds}  
we expect, by Remark \ref{rem:ConvRates}, errors to be roughly proportional to  $M^{-\ell}$. In the examples with $d\approx 0.631$ the relative errors do seem to be proportional to $M^{-\ell}$ for larger $\ell$, supporting 
\rev{Hypothesis \ref{conj}} 
in these cases. 
The errors in the two examples with $d\approx 1.262$ seem to decrease at the same rate, suggesting the same solution regularity \rev{in both cases}, and that the error estimate of Theorem \ref{thm:ApproxDisjoint} may hold also for the Koch curve even though $\Gamma$ is not disjoint \rev{in that case}. But the convergence is slower than $M^{-\ell}$,  suggesting that 
\rev{Hypothesis \ref{conj} does not hold in these cases.}

\subsubsection{The Koch snowflake} \label{sec:Koch}

\begin{figure}
\centering
\subfl{${\rm Re}(u^t)$ for volume approach}
{\includegraphics[width=0.475\textwidth]{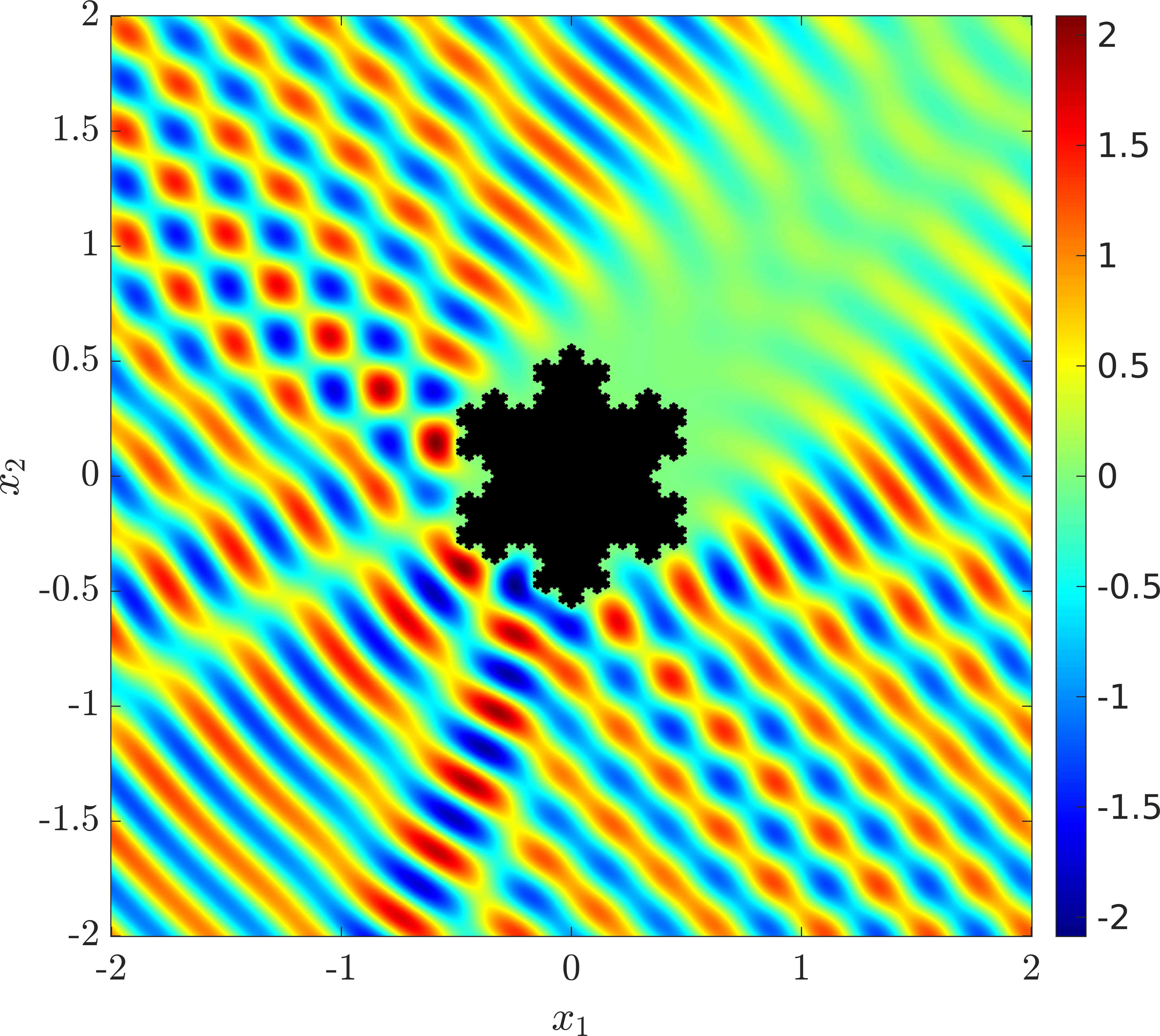}}
\hspace{3mm}
\subfl{${\rm Re}(u^t)$ for boundary approach}
{\includegraphics[width=0.475\textwidth]{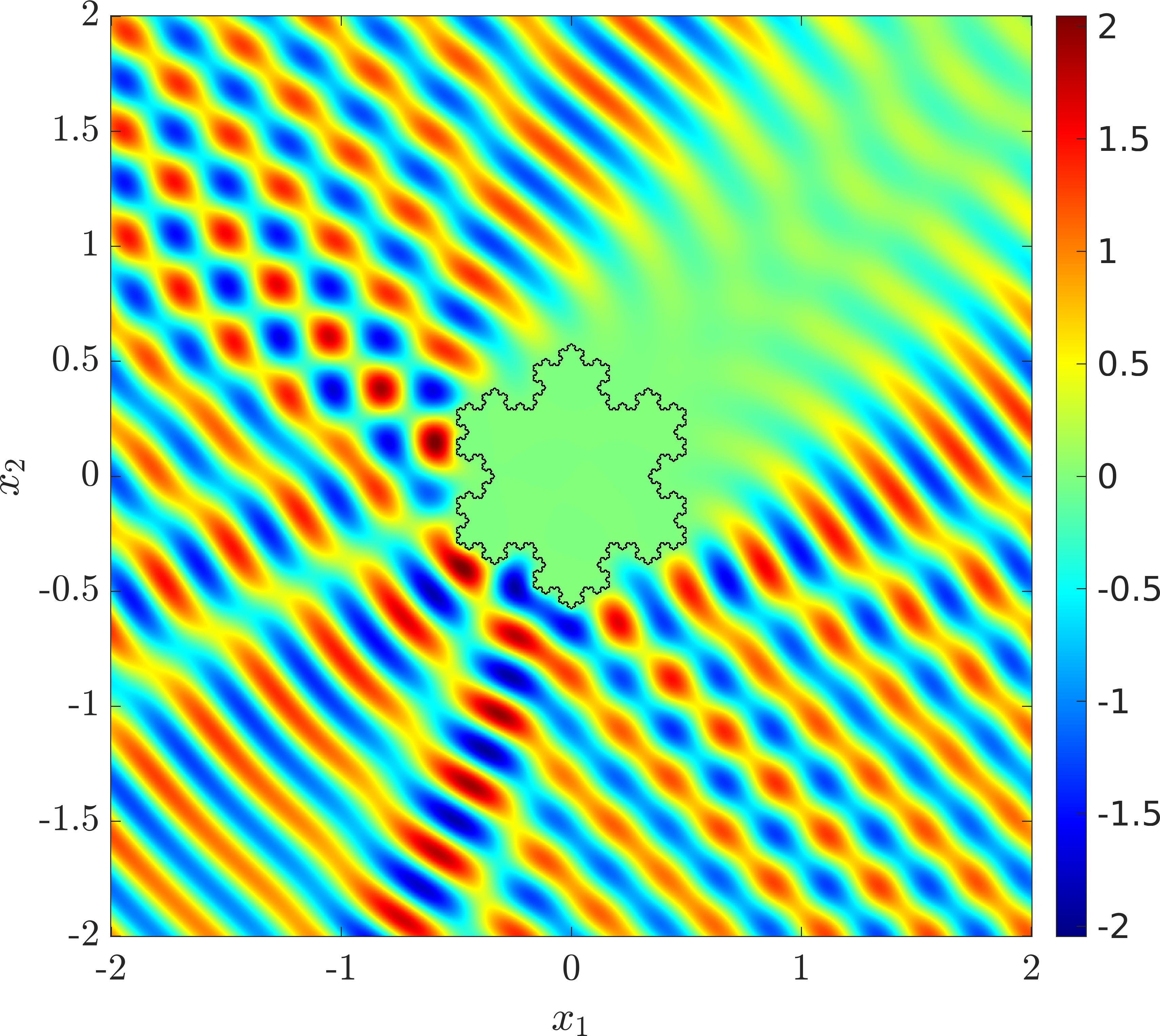}}\\
\subfl{$|\phi_N|$ for volume approach, $h = 0.22$}{\includegraphics[width=.48\textwidth]{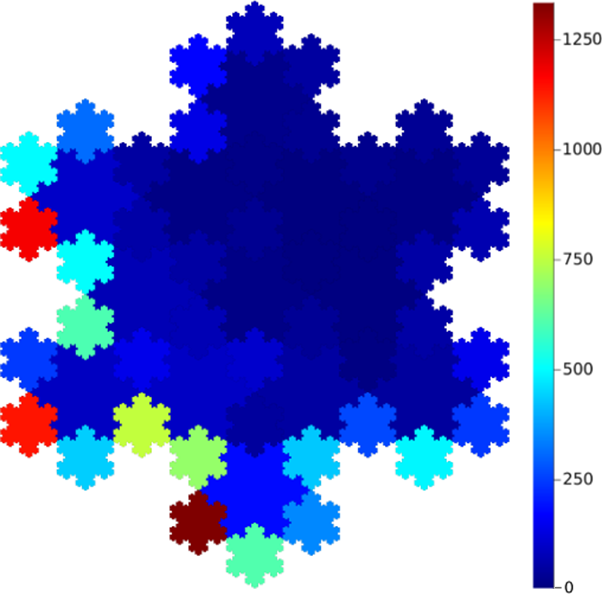}}
\hspace{2mm}
	\subfl{$|\phi_N|$ for volume approach, $h=0.074$}{\includegraphics[width=.48\textwidth]{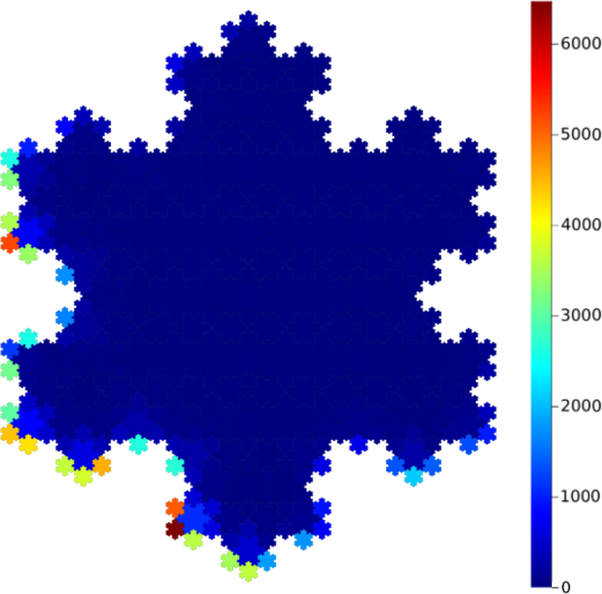}}\\
	\subfl{$|\phi_N|$ for volume approach, $h=0.025$}{\includegraphics[width=.48\textwidth]{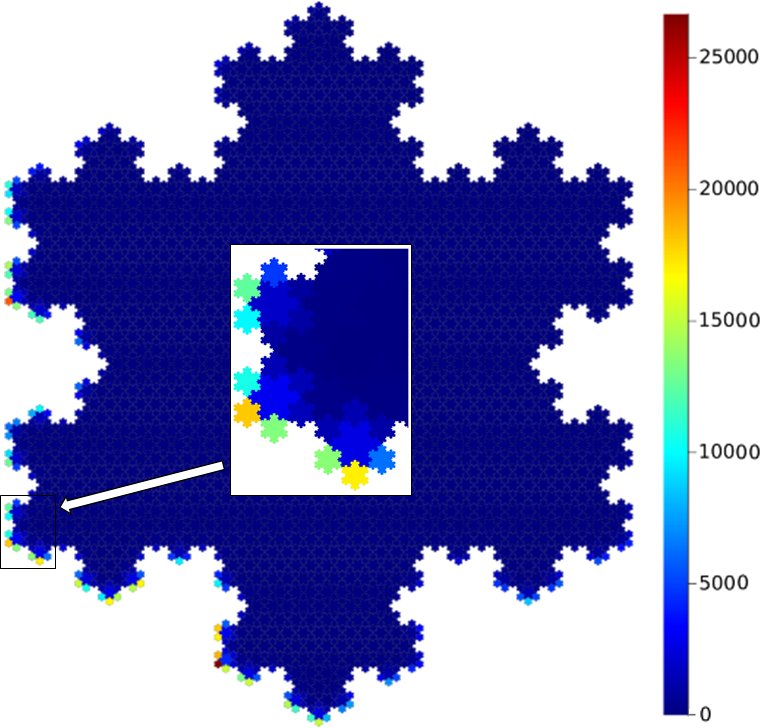}}
\hspace{2mm}
\subfl{Far-field relative $L^\infty$ errors}{\includegraphics[width=.48\textwidth]%
{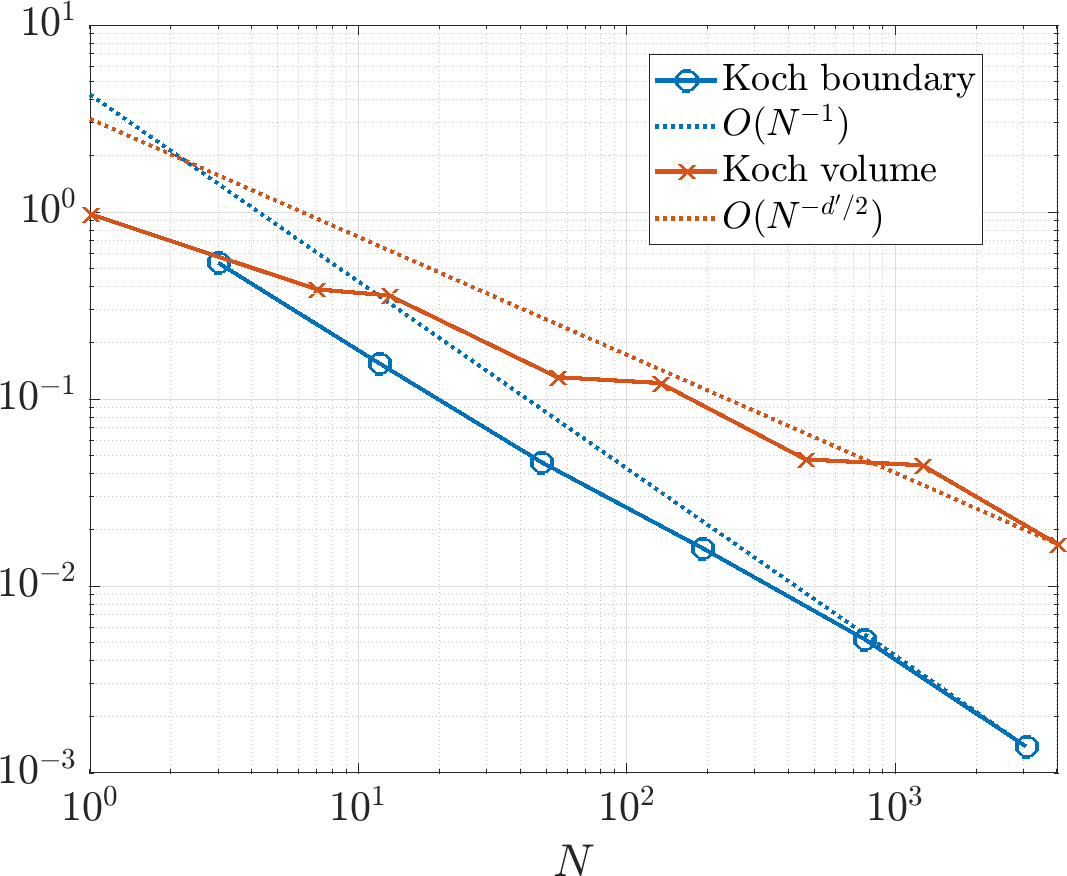}}
\caption{\rev{Scattering by the Koch snowflake; see \S\ref{sec:2D}(\ref{sec:Koch}). (a) and (b) show ${\rm Re}(u^t)$ computed with both approaches, for $k=20$. For the same $k$, (c)-(e) show
$|\phi_N|$, where $\phi_N$ is the Galerkin IEM solution for the volume approach, for three different $h$ values. 
(f) plots the far-field relative $L^\infty$ error for both approaches against the total number of degrees of freedom, for $k=5$.}
\label{fig:KochSnowflake}
}
\end{figure}

In Figure \ref{fig:KochSnowflake} we show approximations to ${\rm Re}(u^t)$ for scattering by a Koch snowflake for the same incident plane wave as Figure \ref{fig:CantorKoch}, computed in two different ways, illustrating Remark \ref{rem:Alt}. In Figure \ref{fig:KochSnowflake}(a) we solve the IE by our Galerkin IEM with $N=4039$ on the solid Koch snowflake $\Gamma$, shown in Fig.~\ref{fig:2Dexs}(g), which  is the attractor of a non-homogeneous IFS with $M=7$ as noted in Ex.~\ref{ex:KochSnowflake}. We refer to this as the \emph{volume approach}. In Figure \ref{fig:KochSnowflake}(b) we solve the IE by our Galerkin IEM on $\partial \Gamma$, the boundary of the snowflake. We refer to this as the \emph{boundary approach}. In contrast to all our other examples, $\partial \Gamma$ is not an IFS attractor, but it  \emph{is} the union of three IFS attractors (rotated copies of the Koch curve of Ex.~\ref{ex:Koch}, each the attractor of an IFS with $M=4$), and so $\partial \Gamma$ \emph{is} a \rev{$d'$-set}, with $d':=\dimH(\partial \Gamma)=\log(4)/\log(3)\approx 1.262$. In Figure \ref{fig:KochSnowflake}(b) we use $M^4=256$ degrees of freedom on each Koch curve comprising $\partial \Gamma$, so that $N=768$. 

In the boundary approach, to assemble the Galerkin matrix $\underline{\underline{A}} $, we view it as a $3\times 3$ block matrix, each block corresponding to interactions between two of the three Koch curves. The diagonal blocks correspond to self-interactions for a single Koch curve, and these matrix elements are approximated by quadrature as described at the beginning of the section. The off-diagonal blocks are assembled using the composite barycentre rule using the same value of $h_Q$ as for the diagonal blocks. %

Proposition \ref{prop:components2} and Remark \ref{rem:Ass} tell us that the IE solution in the volume approach is supported on $\partial \Gamma$, and that the IE solutions and scattered fields for the two approaches coincide %
as long as $k^2$ is not a Dirichlet eigenvalue of $-\Delta$ in $\Omega_-:=\Gamma^\circ$, the interior of the snowflake. It appears that $k=20$ is not one of these resonant wavenumbers as the fields in Figure \ref{fig:KochSnowflake}(a) and (b) coincide and the field in $\Omega_-$ is zero in (b), in agreement with the boundary condition for the volume approach that $u^t\in \tH^1(\Omega)$, where $\Omega = \Gamma^c$. 

In Figure \ref{fig:KochSnowflake}(c)-(e) we plot the modulus of the piecewise-constant Galerkin IEM solution $\phi_N\in V_N=Y_h\subset H_\Gamma^{-1}$ corresponding to Figure \ref{fig:KochSnowflake}(a), for $h \approx 0.22$, $h\approx0.074$, and $h\approx0.025$. Since $\phi_N$ is constant on each element, the meshes used for each $h$ are discernible in Figure \ref{fig:KochSnowflake}(c)-(e) (each element is a scaled copy of the original snowflake $\Gamma$). Each solution $\phi_N$ is highly peaked near $\partial \Gamma$, especially where  $\partial \Gamma$ is illuminated by the incident wave, and is much smaller away from $\partial \Gamma$; these effects are increasingly marked as $h$ is reduced. This is unsurprising as, by Theorem \ref{thm:Convergence}, $\phi_N\to \phi$ in $H_\Gamma^{-1}$ as $h\to0$, and $\phi$ is supported in $\partial \Gamma$.  %

In Figure \ref{fig:KochSnowflake}(f) we explore convergence of the far-field approximations $u_\ell^\infty$ computed by the volume and boundary approaches, showing computations for $h=\diam(\Gamma)/3^{\ell/2}$ for $\ell=0,1,\ldots, 6$ for the volume approach, and $h=3\diam(\Gamma)/3^\ell$ for $\ell=0,1,\ldots,4$ for the boundary approach. For each approach we use as our ``exact'' solution the boundary approach solution with $\ell=\ellref = 7$. Figure \ref{fig:KochSnowflake}(f) shows  the relative errors \eqref{eq:relerrs} in $u_\ell^\infty$ for both methods, for plane wave incidence direction $\vartheta=(1,-1)/\sqrt{2}$ and $k=5$, %
with the $L^\infty$ norms computed using the same discrete set of points as in %
Figure \ref{fig:2DErrorPlots}. In the volume approach every second increment in $\ell$ has a smaller reduction in error. At these increments, the elements adjacent to $\partial \Gamma$, which is the support of the solution, are not being subdivided, as a consequence of the definition of the approximation space $Y_h$. 
The convergence rate results of Theorem \ref{thm:ApproxDisjoint} apply to the volume approach but not to the boundary approach as $\partial \Gamma$ is not an IFS attractor. But, assuming these estimates apply in both cases, and if 
\rev{Hypothesis \ref{conj} holds}, 
so that the solution $\phi\in H_{\partial \Gamma}^{-1}$ has its maximum possible regularity, then 
\rev{as in} Remark \ref{rem:ConvRates} we anticipate errors decreasing roughly in proportion to $h^{d'}$ in both cases, i.e.\ proportional to $N^{-1}$ and $N^{-d'/2}\approx N^{-0.631}$ in the respective boundary and volume cases.   Both approaches appear to be converging somewhat more slowly than these conjectured theoretical rates, but, \rev{of the two},  the boundary approach is certainly converging more rapidly. 
It is plausible that the boundary approach, in which only $\partial \Gamma$ is discretised, should  be more efficient, given that the solution is supported on $\partial \Gamma$. But the IE on $\partial \Gamma$ is not well-posed for all $k>0$, in contrast to the IE on $\Gamma$, and there must be scope to improve the efficiency of the volume approach by using graded versions of our meshes, concentrating elements near $\partial \Gamma$ (cf.~\cite{KhMe:03,cefalo2014optimal}).

\subsection{Examples in 3D space} \label{sec:3D}

\begin{figure}
\centering
\subfl{$\rho=1/2$ and $d=2$}{\includegraphics[width=.45\textwidth]
{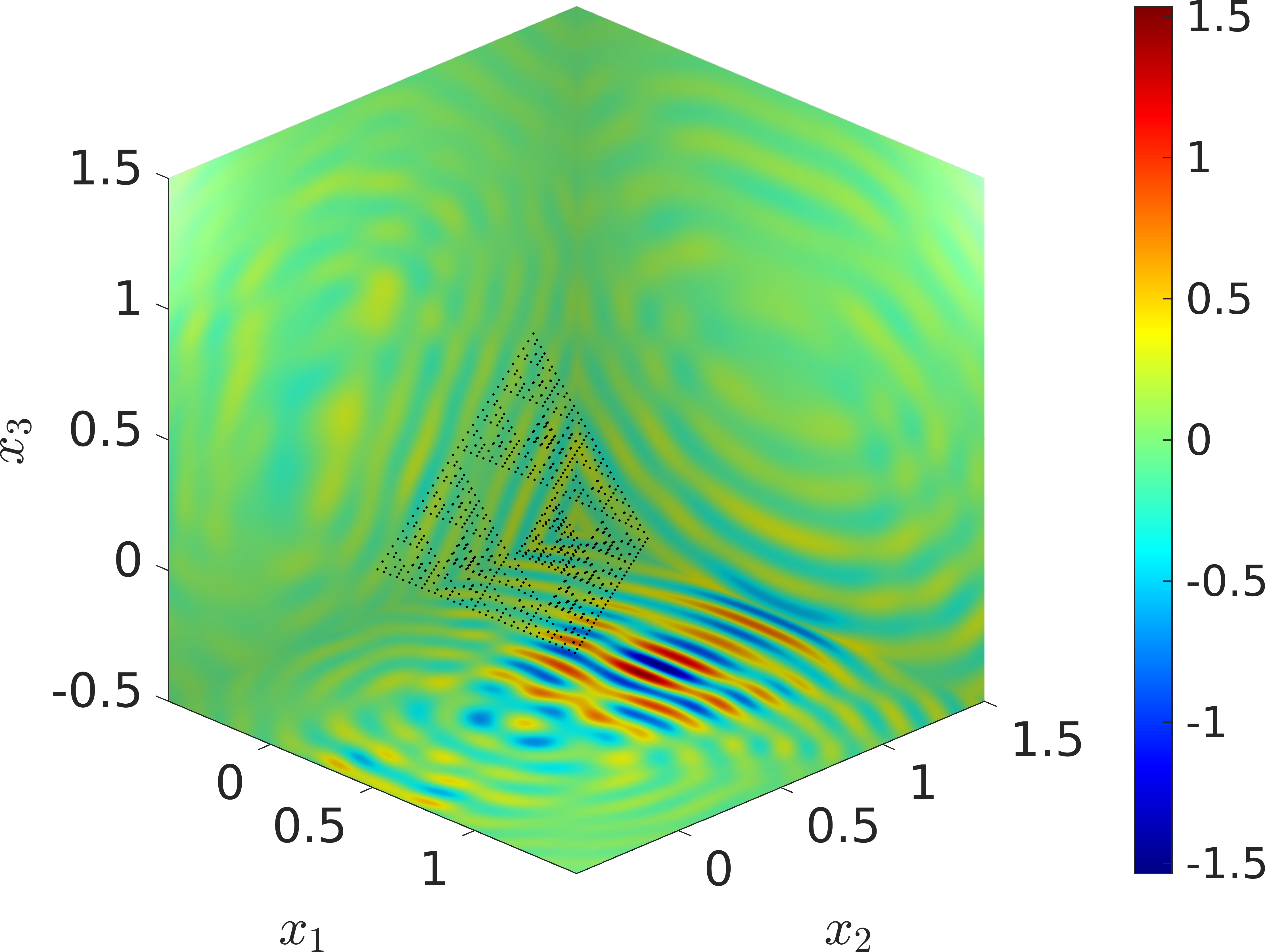}}
\hspace{3mm}
\subfl{$\rho=3/8$ and $d=\log(4)/\log(8/3)\approx 1.413$}{\includegraphics[width=.45\textwidth]
{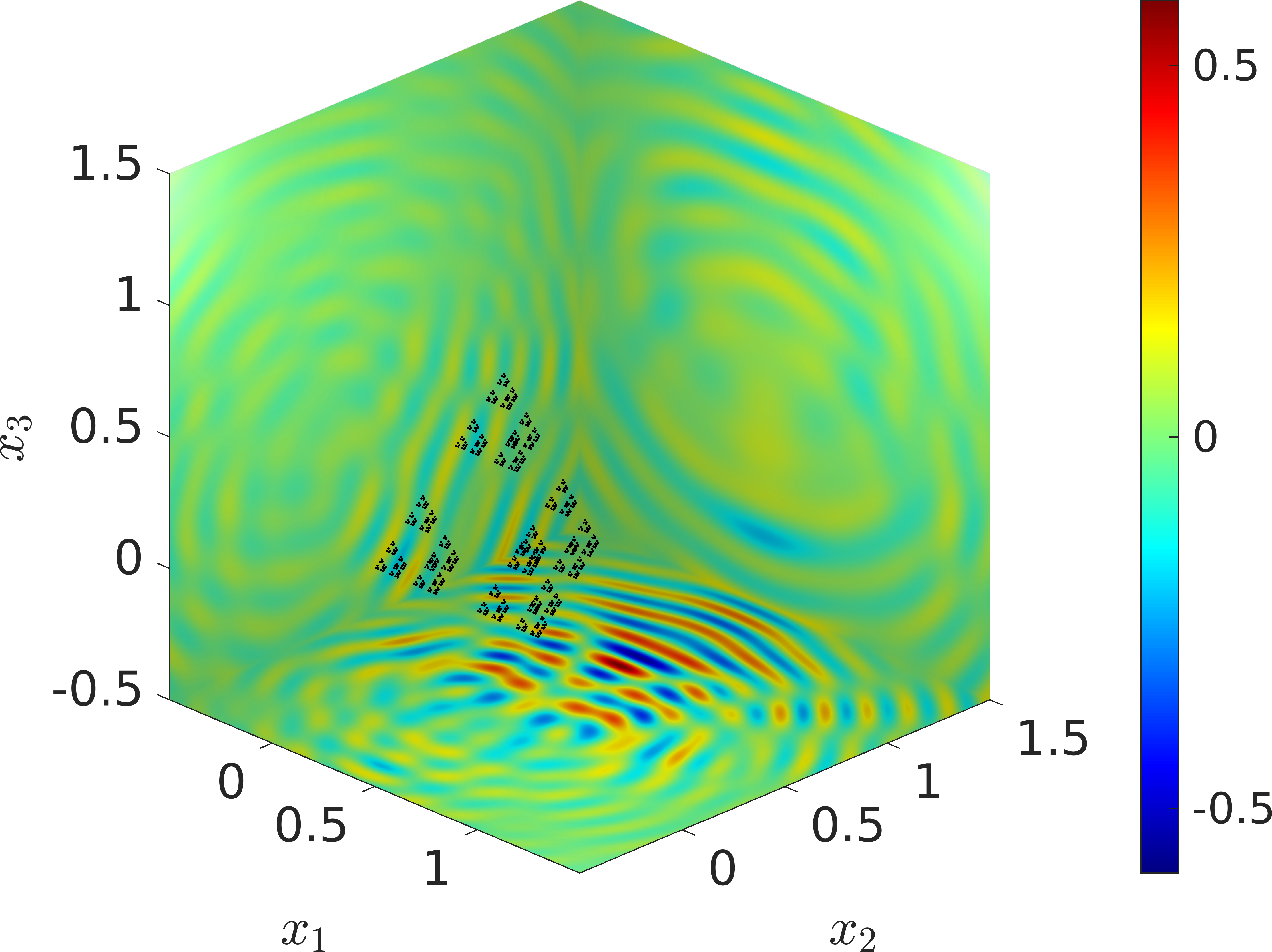}}
\caption{${\rm Re}(u_\ell)$ for scattering by the Sierpinski tetrahedra %
with $k=50$ and $\ell=7$. \rev{See \S\ref{sec:3D}.}
\label{fig:Tetrahedron}
}
\end{figure}

In Figure \ref{fig:Tetrahedron} we show the real parts of the scattered fields created by the two Sierpinski tetradehra of Figure \ref{fig:Tet}, which are attractors of the homogeneous IFS of Example \ref{ex:SierpinskiTetrahedron} with $M=4$ and $d=\log 4/\log(1/\rho)$.
The plane wave incidence direction is $\vartheta=(0, 1, -1)/\sqrt{2}$, $k=50$, and both approximations were computed with $\ell=7$, corresponding to $N=16384$.

\begin{figure}
\centering
\subfl{Far-field relative $L^\infty$ errors}{\includegraphics[height=48mm]{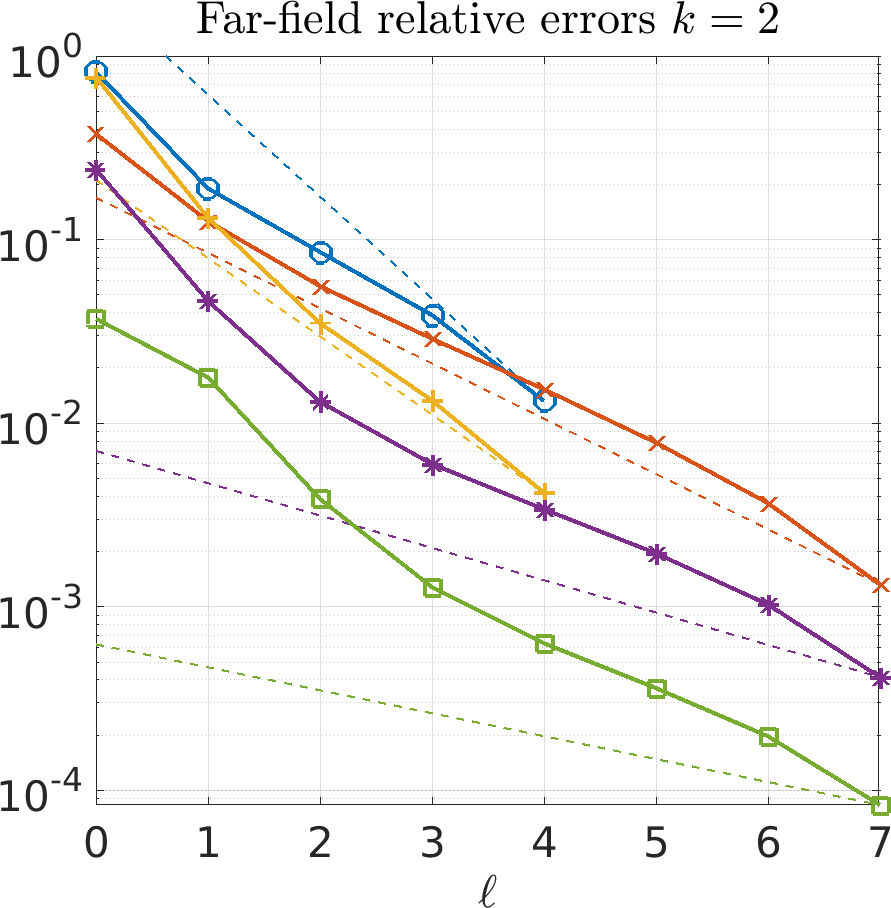}}
\hspace{2mm}
\subfl{Far-field $L^\infty$ absolute increment errors}{\includegraphics[height=48mm]{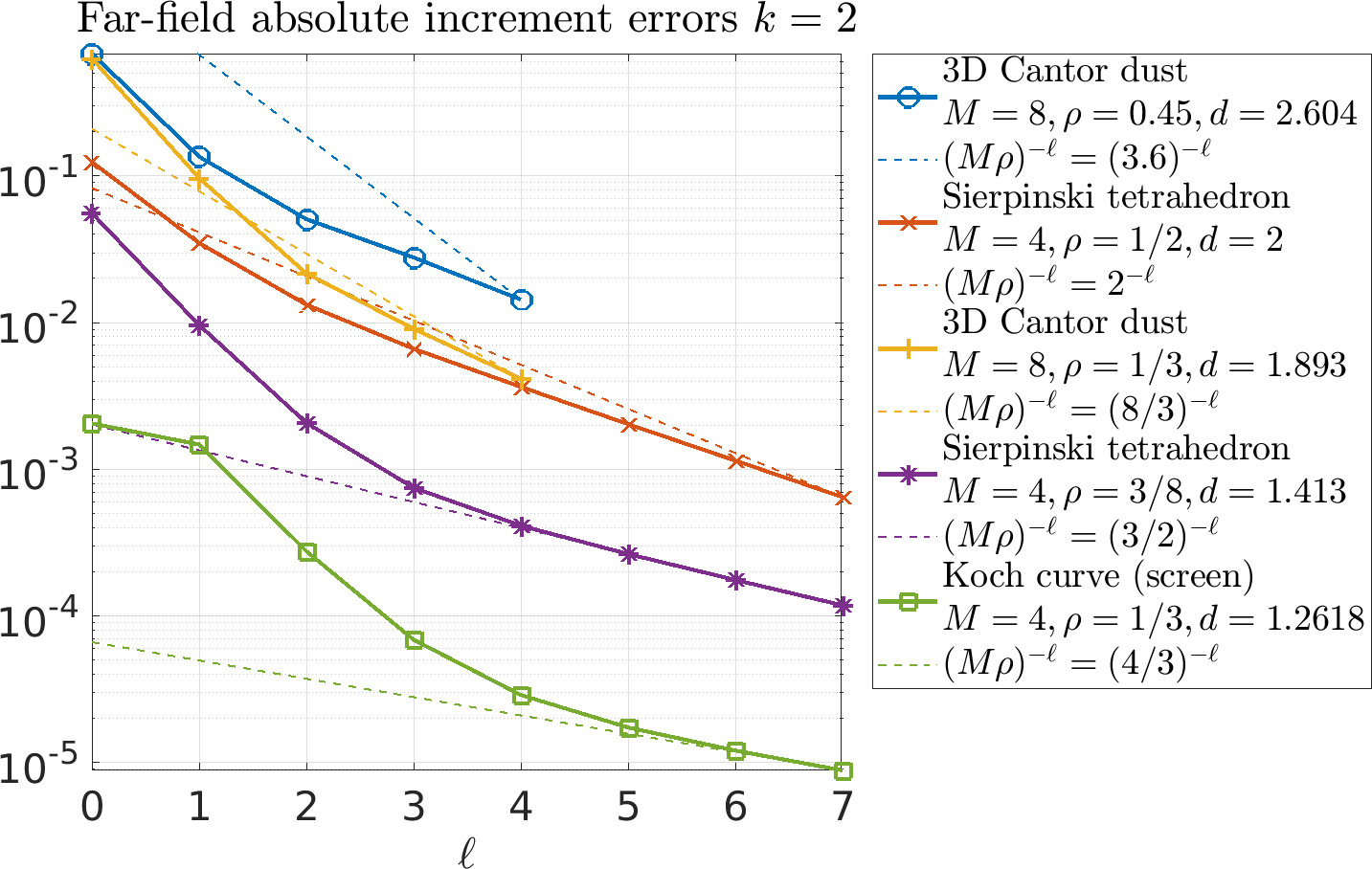}}
\caption{$L^\infty$ convergence plots for a range of 3D examples. \rev{See \S\ref{sec:3D}.}}
\label{fig:3DErrorPlots}
\end{figure}

In Figure \ref{fig:3DErrorPlots} we show $L^\infty$ far-field errors for the same incidence direction and $k=2$ for a range of 3D examples, namely: the Sierpinski tetrahedra of Figures \ref{fig:Tet} and \ref{fig:Tetrahedron} and Example \ref{ex:SierpinskiTetrahedron}; 3D Cantor Dusts, i.e., $C\times C\times C$ where $C$ is the Cantor set of Example \ref{ex:Cantor}, with $\rho=1/3$ and $\rho=0.45$; the Koch curve of Figure \ref{fig:KochCurveDecomposition} embedded in 3D space, i.e. $K\times \{0\}$, where $K\subset \R^2$ is the Koch curve of Example \ref{ex:Koch}. All these scatterers $\Gamma$ have $d=\dimH(\Gamma) > 1$ (see Figure \ref{fig:3DErrorPlots}) so that $H_\Gamma^{-1}$ is non-trivial by Remark \ref{rem:cap}, the Galerkin IEM is applicable, and the solution $\phi\in H_\Gamma^{-1}$ to the IE  \eqref{eq:IE} is non-zero (since \eqref{eq:IE} is equivalent to \eqref{eq:IEnew} and  $\tr g=-u^i|_\Gamma$ is non-zero), so also (by \eqref{eq:NewtonHelmholtz}) the scattered field $u=\cA\phi$ is non-zero.

To compute the discrete $L^\infty$ relative errors \eqref{eq:relerrs} shown in Figure \ref{fig:3DErrorPlots}(a) we sample at $200$ points on the sphere $\IS^2$, chosen so that the points form a uniform grid in spherical coordinate space $[0,\pi]\times[0,2\pi]$, and we use, for each scatterer, $\ellref=\ell_{\max}+1$, where $\ell_{\max}$ is the largest $\ell$ for which results are shown. This choice of $\ellref$, constrained by computational resources, is not large enough for $u^\infty_{\ellref}$ to be a sufficiently accurate ``exact'' solution to 
see convergence rates clearly. Thus we also plot in Figure \ref{fig:3DErrorPlots}(b) the absolute increment errors $\|u^\infty_\ell-u^\infty_{\ell+1}\|_{L^\infty}$ for $\ell=0,\ldots,\ellref-1$. 
As discussed in \cite[\S6.2]{HausdorffBEM}, if, for some $c>0$ and $0<\alpha<1$, $\|u^\infty_\ell-u^\infty_{\ell+1}\|_{L^\infty}=c\alpha^\ell$ for all $\ell\geq \ell_0$ then, by the triangle rule, $\|u^\infty-u^\infty_{\ell}\|_{L^\infty}\leq \frac{c}{1-\alpha}\, \alpha^\ell$ for $\ell\geq \ell_0$. Thus convergence rates can be deduced from Figure \ref{fig:3DErrorPlots}(b). 
By Remark \ref{rem:ConvRates}, which applies to all the examples except the Koch curve \rev{and the Sierpinski tetrahedron with $d=2$}, we expect,  
\rev{if Hypothesis \ref{conj} holds}, 
to see errors roughly proportional to $(M\rho)^{-\ell}$. This rate is observed in Figure \ref{fig:3DErrorPlots}(b) for sufficiently large $\ell$ for all the cases with $d<2$, but the convergence is significantly slower than $(M\rho)^{-\ell}$ for the example with $d\approx 2.6$. These results, and the convergence results reported in \S\ref{sec:2D}, suggest that 
\rev{Hypothesis \ref{conj}}  
does not hold in cases where $d'=\dimH(\partial \Gamma)>n-1$ (note $\partial \Gamma=\Gamma$ for the scatterers in Figures  \ref{fig:2DErrorPlots} and \ref{fig:3DErrorPlots}), but may hold in cases where $d'<n-1$. %
They suggest moreoever that 
\rev{Hypothesis \ref{conj}} 
and the estimates of Theorem  \ref{thm:IEMConvergence} may hold for the Koch curve screen in 3D, even though $\Gamma$ is non-disjoint in this case.

\appendix
\section{Singular quadrature on fractals}
\label{sec:QuadratureAppendix}

\noindent In this appendix we briefly outline the methodology of \cite{HausdorffQuadrature,NonDisjointQuad} for the derivation of representation formulas for singular integrals on fractals
in terms of regular integrals.
We assume throughout that $\Gamma$ is the attractor of an IFS satisfying the OSC.

The basic singular integral we consider is
\[ I_{\Gamma,\Gamma}:=\int_\Gamma\int_\Gamma \widetilde\Phi_t(|x-y|)\,\rd\cH^d(x)\rd\cH^d(y), \]
where
\begin{align*}
	\label{}
	\tilde\Phi_t(r):=
	\begin{cases}
		r^{-t}, & t>0,\\
		\log{r}, & t=0.
	\end{cases}
\end{align*}
The cases $t=0$ and $t=1$ are those relevant to \S4(b), %
since $\Phi_{\rm sing}(x,y) = -\tilde\Phi_0(|x-y|)/(2\pi)$ for $n=2$, and $\Phi_{\rm sing}(x,y) = \tilde\Phi_1(|x-y|)/(4\pi)$ for $n=3$,
but for completeness we consider the general case.
The integral $I_{\Gamma,\Gamma}$ is finite if and only if $t<d$, where $d$ is the Haudsorff dimension of $\Gamma$ (see, e.g., \cite[Cor.~A.2]{HausdorffQuadrature} or \cite[Cor.~2.3]{HausdorffBEM}), so we assume henceforth that $0\leq t<d$.

\subsection{Similarity}
\label{sec:Similarity}

The representation formulas and quadrature rules presented in \cite{HausdorffQuadrature,NonDisjointQuad} are based on decomposing $I_{\Gamma,\Gamma}$ as a sum of integrals over self-similar subsets of $\Gamma$.
For $\bm,\bmp\in \bigcup_{\ell=0}^\infty I_\ell$
we
define
\[ I_{\bm,\bmp}:=\int_{\Gamma_\bm}\int_{\Gamma _\bmp}\widetilde\Phi_t(|x-y|)\,\rd\cH^d(x)\rd\cH^d(y),
\]
which
is singular
when $\Gamma_\bm\cap\Gamma_\bmp$ is non-empty, and
regular
otherwise.
Key to the methodology of \cite{HausdorffQuadrature,NonDisjointQuad} is
that many of these integrals $I_{\bm,\bmp}$, for different choices of $\bm,\bmp$, can be related to each other using the self-similarity of $\Gamma$ and the homogeneity of $\tilde\Phi_t$, namely that, for $\rho>0$,
\begin{align}
	\label{eq:PhitProp}
	\tilde\Phi_t(\rho r)= \begin{cases}
		\rho^{-t}\tilde\Phi_t(r), & t>0,\\
		\log\rho + \tilde\Phi_t(r), & t=0.
	\end{cases}
\end{align}

The following result is a consequence of \cite[Props~3.2 \& 3.3, Rem.~2.1]{NonDisjointQuad}.
Here $|\bm|$ denotes the length of the vector index $\bm$, with $|0|$ interpreted as $0$ in the case $\Gamma_{\bm}=\Gamma_0=\Gamma$, and we define $\vartheta_t$ by $\vartheta_t:=0$ for $t>0$ and $\vartheta_t:=\cH^d(\Gamma)^2$ for $t=0$.
Adopting the terminology of \cite{NonDisjointQuad}, when the conditions of Proposition \ref{prop:Similarity} hold we say that the integrals $I_{\bm,\bm'}$ and $I_{\bn,\bn'}$ are \emph{similar}.
The
condition \eqref{eq:SimilarityCond} stipulates that there exist similarities mapping $\Gamma_\bm$ to $\Gamma_\bn$, and $\Gamma_\bmp$ to $\Gamma_\bnp$, defined in terms of the IFS and symmetry properties of $\Gamma$, that are compatible in a certain sense. 
We note that one can always take the isometries $T$ and $T'$ to be the identity map in the following.
\begin{prop}
	\label{prop:Similarity}
	Let $\bm,\bmp,\bn,\bnp\in \bigcup_{\ell=0}^\infty I_\ell$.
	Let $T$ and $T'$ be isometries of $\R^n$ such that $T(\Gamma)=T'(\Gamma)=\Gamma$.
	Suppose there exists $\varrho>0$ such that
	\begin{align}
		\label{eq:SimilarityCond}
		|s_{\bm}(T(s_{\bn}^{-1}(x)))-s_{\bmp}(T'(s_{\bnp}^{-1}(y)))| = \varrho|x-y|, \quad  x,y \in \R^n.
	\end{align}
	Then
	\begin{align}
		\label{eq:varrhoRelation}
		\varrho = \frac{\rho_\bm}{\rho_\bn}=
		\frac{\rho_\bmp}{\rho_\bnp}
	\end{align}
	and
	\begin{align}
		\label{eq:Proportional}
		I_{\bm,\bm'} =
		\varrho^{2d-t}I_{\bn,\bn'} + \vartheta_t(\rho_\bm\rho_\bmp)^d\log{\varrho}, \qquad t\in [0,d).
	\end{align}
	Furthermore, if $\Gamma$ is homogeneous then $|\bm|-|\bn|=|\bmp|-|\bnp|$, $\varrho = \rho^{|\bm|-|\bn|}$, and, noting that 
	$\rho^d = 1/M$, \eqref{eq:Proportional} can be written as
	\begin{align}
		\label{eq:ProportionalHomogeneous}
		I_{\bm,\bm'} =
		\frac{\rho^{-t(|\bm|-|\bn|)}}{M^{2(|\bm|-|\bn|)}}I_{\bn,\bn'} + \vartheta_t \frac{(|\bm|-|\bn|)}{M^{|\bm|+|\bmp|}}\log{\rho},
		\qquad t\in [0,d).
	\end{align}
\end{prop}

The methodology of \cite{HausdorffQuadrature,NonDisjointQuad} involves attempting to
\begin{itemize}
	\item[(i)] identify a finite number $n_s$ of ``fundamental'' singular integrals \[I_{\bm_{s,1},\bmp_{s,1}}, \ldots, I_{\bm_{s,n_s},\bmp_{s,n_s}},\] with $I_{\bm_{s,1},\bmp_{s,1}}=I_{\Gamma,\Gamma}$, such that any other singular integral $I_{\bm,\bmp}$, for $\bm,\bmp\in L_h$, is similar to one of them in the sense of Proposition \ref{prop:Similarity} (via suitable $T,T'$);
	\item[(ii)] derive an $n_s$-by-$n_s$ linear system of equations satisfied by these fundamental singular integrals, that can be solved to express them (and hence any similar singular integral) in terms of regular integrals that can be computed numerically (e.g.\ using the composite barycentre rule).
\end{itemize}

An algorithm for attempting to achieve (i) and (ii) in the general case is presented in \cite[Algorithm 1]{NonDisjointQuad}, and the earlier results in \cite[\S4.3]{HausdorffQuadrature} can be viewed as a specialisation of this algorithm to the case of a disjoint attractor, for which $n_s=1$.
The basic idea is to start from $I_{\Gamma,\Gamma}$ and combine repeated subdivision of the integration domain with repeated applications of Proposition \ref{prop:Similarity} to determine when integrals are similar.
Whether the algorithm succeeds in achieving (i) and (ii) depends on the example being studied, but in \cite[\S5]{NonDisjointQuad} it was shown to succeed for a number of well-known examples of non-disjoint attractors including the Sierpinski triangle, Vicsek fractal, Sierpinski carpet and the Koch snowflake. Rather than repeating the full details of the algorithm here, we instead exemplify the procedure in the case of the Koch curve (see Fig.~1(f)), which was not studied in \cite{NonDisjointQuad}.

\subsection{Example - Koch curve}

The Koch curve (see Ex.~2.3, Fig.~1(f) and  Fig.~3)   %
$\Gamma\subset\R^2$ is the attractor of a  homogeneous IFS with $M=4$ and $\rho=1/3$, so that $d = \log{4}/\log{3}$ and $\rho^d=1/4$. The only isometries of $\R^2$ under which $\Gamma$ is invariant are the identity, and reflection in the line $x=1/2$.

	To make the notation more compact, given $\bm = (m_1,m_2,\ldots,m_\ell)$ and $\bmp = (m_1',m_2',\ldots,m_{\ell'}')$ we write $\Gamma_\bm$, $\Gamma_\bmp$ and $I_{\bm,\bmp}$ as $\Gamma_{m_1m_2\ldots m_\ell}$, $\Gamma_{m_1'm_2'\ldots m_{\ell'}'}$ and $I_{m_1m_2\ldots m_\ell,m_1'm_2'\ldots m_{\ell'}'}$, respectively. For example, we write $I_{(1,3),(2,4)}$ as $I_{13,24}$. This compact notation is unambiguous because $M<10$, which ensures that each entry in $\bm $ and $\bmp$ is a single-digit integer.
	
	Using Proposition \ref{prop:Similarity} (with $\varrho=1$ and appropriate choices of $T$ and $T'$) one can check that
	$I_{1,1}=I_{2,2}=I_{3,3}=I_{4,4}$, $I_{1,2}=I_{3,4}$, $I_{1,3}=I_{2,4}$, and $I_{i,j}=I_{j,i}$ for all $i,j\in\{1,\ldots,4\}$.
	Hence a level 1 decomposition of $I_{\Gamma,\Gamma}$ (illustrated in Fig.~3(a)) %
	gives
	\begin{align}
		\label{eq:KochDecomp1}
		I_{\Gamma,\Gamma} = \sum_{i=1}^4 \sum_{j=1}^4 I_{i,j} = 4I_{1,1} + 4I_{1,2} + 2I_{2,3} + R_{\Gamma,\Gamma},
	\end{align}
	where $R_{\Gamma,\Gamma}$ is a sum of regular integrals, given by
	\[ R_{\Gamma,\Gamma}:=4I_{1,3} + 2I_{1,4}.\]
	The integral $I_{1,1}$ is singular, but by Proposition \ref{prop:Similarity} is similar to $I_{\Gamma,\Gamma} = I_{0,0}$, with
	\begin{align}
		\label{eq:KochDecomp2}
		I_{1,1} = \frac{3^{t}}{16}I_{\Gamma,\Gamma} - \vartheta_t\frac{\log{3}}{16}.
	\end{align}
	The integrals $I_{1,2}$ and $I_{2,3}$ are also singular, since $\Gamma_1$ and $\Gamma_2$ intersect at a point, as do $\Gamma_2$ and $\Gamma_3$. But they are not similar to $I_{\Gamma,\Gamma}$, so we apply a level 2 decomposition (see  Fig.~3(b)),  %
	writing
	\begin{align}
		\label{eq:KochDecomp3}
		I_{1,2} = I_{14,21}+ R_{1,2} \qquad\text{and}\qquad I_{2,3} = I_{24,31}+ R_{2,3},
	\end{align}
	where $R_{1,2}$ and $R_{2,3}$ are both sums of regular integrals, given by
	\[R_{1,2} = 2I_{11,21} + 2I_{11,22} + 2I_{11,23} + I_{11,24} + 2I_{12,21} + 2I_{12,22} + I_{12,23} + 2I_{13,21} + I_{13,22},\]
	\[R_{2,3} = 4I_{21,31} + 2I_{21,32} + 2I_{21,33} + 2I_{21,34} + 2I_{22,31} + 2I_{23,31} + I_{23,32}.\]
	The integrals $I_{14,21}$ and $I_{24,31}$ are both singular, but are similar to $I_{1,2}$ and $I_{2,3}$, with
	\begin{align}
		\label{eq:KochDecomp4}
		I_{14,21} = \frac{3^{t}}{16}I_{1,2} - \vartheta_t\frac{\log{3}}{256}, \qquad\text{and}\qquad I_{24,31} = \frac{3^{t}}{16}I_{2,3} - \vartheta_t\frac{\log{3}}{256}.
	\end{align}
	Combining \eqref{eq:KochDecomp1}-\eqref{eq:KochDecomp4},  we find that the vector of fundamental singular integrals $(I_{\Gamma,\Gamma},I_{1,2},I_{2,3})^T$ satisfies the linear system
	\begin{align}
		\label{}
		\left(\begin{array}{ccc}
			\sigma_1 & -4 & -2 \\
			0 & \sigma_2 & 0 \\
			0 & 0 & \sigma_2
		\end{array}\right)
		\left(\begin{array}{c}
			I_{\Gamma,\Gamma}\\
			I_{1,2}\\
			I_{2,3}
		\end{array}\right)
		=
		\left(\begin{array}{c}
			R_{\Gamma,\Gamma} - \vartheta_t \frac{\log{3}}{4}\\
			R_{1,2}- \vartheta_t \frac{\log{3}}{256}\\
			R_{2,3}- \vartheta_t \frac{\log{3}}{256}
		\end{array}\right), \qquad t\in [0,d),
	\end{align}
	where
	\[\sigma_1 = 1-\frac{3^t}{4},
	\quad \sigma_2 = 1-\frac{3^t}{16}.\]
		Solving the system gives
		\begin{align}
			\label{eq:IGammaGamma_Koch_final1}
			I_{2,3} = \frac{1}{\sigma_2}\Big(R_{2,3} -  \vartheta_t \frac{\log{3}}{256}\Big),
			\qquad I_{1,2} = \frac{1}{\sigma_2}\Big(R_{1,2} - \vartheta_t \frac{\log{3}}{256}\Big),
		\end{align}
		and
		\begin{align}
			\label{eq:IGammaGamma_Koch_final2}
			I_{\Gamma,\Gamma} = \frac{1}{\sigma_1}\left(R_{\Gamma,\Gamma} + \frac{2}{\sigma_2}\Big(2R_{1,2}+R_{2,3}\Big) -  \vartheta_t \Big(32+\frac{3}{\sigma_2}\Big)\frac{\log{3}}{128}\right).
		\end{align}
		
		Having derived the representation formulas \eqref{eq:IGammaGamma_Koch_final1}-\eqref{eq:IGammaGamma_Koch_final2}, one can obtain numerical approximations of $I_{\Gamma,\Gamma}$, $I_{1,2}$ and $I_{2,3}$ by combining these with numerical evaluations of
		$R_{\Gamma,\Gamma}$, $R_{1,2}$ and $R_{2,3}$, e.g.\ using the composite barycentre rule with some maximum mesh width $\tilde{h}>0$.
		
		To apply these results in the assembly of the Galerkin matrix considered in  \S4(b),   %
		we note that since we are using the approximation space $V_N=Y_h$ on the mesh $L_h$ of $\Gamma$, any singular instance of   (4.17) %
		will be similar to one of $I_{\Gamma,\Gamma}$, $I_{1,2}$ and $I_{2,3}$, say $I_{i,j}$.
		To evaluate  (4.17)
		we apply Proposition \ref{prop:Similarity} to obtain a formula for  (4.17)
		in terms of $I_{i,j}$, which can be evaluated using the value for $I_{i,j}$ already computed (as discussed above).
		For instance, referring back to  Fig.~3(b)
		with $\bm = (1,1,3)$ and $\bmp = (1,1,4)$ the integral  (4.17)
		is similar to $I_{1,2}$, with
		\[ I_{113,114} = \frac{3^{2t}}{256}I_{1,2} - \vartheta_t \frac{\log{3}}{2048}.\]
		To ensure parity between the quadrature accuracies for such singular instances of  (4.17),
		and the regular instances of  (4.17),
		which were computed using the composite barycentre rule with maximum mesh width $h_Q$, we take $\tilde{h}=h_Q/h$.
		
		\bigskip
		\small
		
		\subsubsection*{Acknowledgements}
		{SC-W was supported by EPSRC grant EP/V007866/1, DH and AG by EPSRC grants EP/S01375X/1 and EP/V053868/1, AM by the PRIN project ``NA-FROM-PDEs''
			and by PNRR-M4C2-I1.4-NC-HPC-Spoke6, funded by the European Union - Next Generation EU, and AC by CIDMA (Center for Research and Development in Mathematics and Applications) and FCT (Foundation for Science and Technology) within project UIDB/04106/2020 
			(\href{https://doi.org/10.54499/UIDB/04106/2020}{doi.org/10.54499/UIDB/04106/2020}). AG, SC-W, DH and AM thank the Isaac Newton Institute for Mathematical Sciences for support and hospitality during the programme ``Mathematical Theory and Applications of Multiple Wave Scattering'', 
			supported by EPSRC grant EP/R014604/1. AG acknowledges use of the UCL Myriad High Performance Computing Facility (Myriad@UCL) and associated support services. 
			We thank the reviewers for their many helpful comments.}
		\bibliography{BEMbib_short2014}%
		\bibliographystyle{siam}
		
\end{document}

%% file: macros.tex
\newcommand{\rf}[1]{(\ref{#1})}

\newcommand{\ri}{{\mathrm{i}}}
\newcommand{\re}{{\mathrm{e}}}
\newcommand{\rd}{\mathrm{d}}

\newcommand{\R}{\mathbb{R}}

\newcommand{\N}{\mathbb{N}}

\newcommand{\C}{\mathbb{C}}

\newcommand{\cA}{\mathcal{A}}

\newcommand{\cH}{\mathcal{H}}

\newcommand{\bs}[1]{\mathbf{#1}}

\newcommand{\bn}{\mathbf{n}}

\newcommand{\bx}{\mathbf{x}}
\newcommand{\by}{\mathbf{y}}

\newcommand{\pdone}[2]{\dfrac{\partial {#1}}{\partial {#2}}}

\newtheorem{thm}{Theorem}[section]
\newtheorem{lem}[thm]{Lemma}

\newtheorem{prop}[thm]{Proposition}
\newtheorem{cor}[thm]{Corollary}
\newtheorem{rem}[thm]{Remark}

\newtheorem{ass}[thm]{Assumption}
\newtheorem{example}[thm]{Example} 
\newtheorem{hyp}[thm]{Hypothesis}

\newcommand{\tH}{\widetilde{H}}
\newcommand{\dimH}{{\rm dim_H}}

\newcommand{\Rn}{{(\R^n)}}